%% file: KL_ParaIden.tex
\documentclass[a4paper,11pt]{article}
\usepackage{amsmath}
\usepackage{amssymb}
\usepackage{amsthm}
\usepackage{MnSymbol}
\usepackage{paralist} 
\usepackage{enumitem} 
\usepackage{hyperref}
\usepackage{url}
\usepackage{bm} 
\usepackage{color}
\usepackage[margin=30mm]{geometry}

\usepackage{pgfplots}
\usepackage{xcolor}
\usepackage{tikz}
\usetikzlibrary{plotmarks} 
 \pgfplotsset{compat=1.14}
 
 \usepgfplotslibrary{external}
      \tikzexternaldisable
 
\newcommand{\N}{\mathbb{N}}
\newcommand{\R}{\mathbb{R}}

\renewcommand{\div}{ \, \mathrm{div}  }

\newcommand{\dd}{\, d}
\newcommand{\dx}{\, dx}
\newcommand{\dt}{\, dt}
\newcommand{\dz}{\, dz}
\newcommand{\pd}{\partial}
\newcommand{\pdnu}{\pd_{\bm{\nu}}}
\newcommand{\mean}[1]{\overline{#1}}

\newcommand{\PP}{\mathcal{P}}
\renewcommand{\AA}{\mathcal{A}}
\newcommand{\CC}{\mathcal{C}}

\newcommand{\abs}[1]{\left| #1 \right|}
\newcommand{\norm}[1]{\| #1 \|}

\newcommand{\inner}[2]{\langle #1 , #2 \rangle}

\newcommand{\eps}{\varepsilon}
\newcommand{\Laplace}{\Delta}
 
\theoremstyle{plain}
\newtheorem{thm}{Theorem}[section]

\theoremstyle{plain}
\newtheorem{remark}{Remark}[section]

\newtheorem{assump}{Assumption}[section]
\numberwithin{equation}{section}

\begin{document}
\title{Parameter identification via optimal control for a
Cahn--Hilliard-chemotaxis system with a variable mobility\footnotemark[3]}

\author{Christian Kahle \footnotemark[1] \and Kei Fong Lam \footnotemark[2]}

\date{\today}

\maketitle

\renewcommand{\thefootnote}{\fnsymbol{footnote}}
\footnotetext[1]{
Zentrum Mathematik, 
Technische Universit\"{a}t M\"{u}nchen, 
85748 Garching bei M\"{u}nchen, 
Germany 
(\tt{Christian.Kahle@ma.tum.de})
}
\footnotetext[2]{
Fakult\"at f\"ur Mathematik, 
Universit\"at Regensburg, 
93040 Regensburg, 
Germany 
(\tt{Kei-Fong.Lam@mathematik.uni-regensburg.de})
}
\footnotetext[3]
{
The first author gratefully acknowledges the support by the Deutsche
Forschungsgemeinschaft (DFG) through the International Research Training Group
IGDK 1754 ``Optimization and Numerical Analysis for Partial Differential Equations with Nonsmooth Structures".
}

\maketitle

\begin{abstract}
We consider the inverse problem of identifying parameters in a variant of the diffuse interface model for tumour growth
model proposed by Garcke, Lam, Sitka and Styles (Math. Models Methods Appl. Sci. 2016).  The model contains three constant parameters; namely the tumour growth rate, the chemotaxis parameter and the nutrient  consumption rate.  We study the inverse problem from the viewpoint of PDE-constrained optimal control theory and establish first order optimality conditions.  A chief difficulty in the theoretical analysis lies in proving high order continuous dependence of the strong solutions on the parameters, in order to show the solution map is continuously Fr\'{e}chet differentiable when the model has a variable mobility.  Due to technical restrictions, our results hold only in two dimensions for sufficiently smooth domains.  Analogous results for polygonal domains are also shown for the case of constant mobilities.  Finally, we propose a discrete scheme for the numerical simulation of the tumour model and solve the inverse problem using a trust-region Gauss--Newton approach.
\end{abstract}
  


\section{Introduction}
Our main subject of study is the following Cahn--Hilliard system that has been proposed 
in \cite{GLSS} for the modelling of tumour growth: 
 Let $\Omega \subset \R^{2}$ be a bounded domain with boundary $\pd \Omega$, $T > 0$ 
 denotes a fixed terminal time.  
 Let $Q := \Omega \times (0,T)$ and $\Gamma := \pd \Omega \times (0,T)$, 
 then we study the following system of equations:
\begin{subequations}
\label{Intro:CH}
\begin{alignat}{3}
\varphi_{t} & = \div (m(\varphi) \nabla (\mu  - \chi \sigma) ) + \PP f(\varphi) g(\sigma) && \text{ in } Q, \label{CH:1} \\
\mu & = \beta \eps^{-1} \Psi'(\varphi) - \beta \eps \Laplace \varphi && \text{ in } Q, \label{CH:2} \\
\sigma_{t} & = \Laplace \sigma - \CC h(\varphi) \sigma && \text{ in } Q, \label{CH:3} \\
0 & = \pdnu \varphi = \pdnu \mu = \pdnu \sigma && \text{ on } \Gamma, \\
\varphi(0) & = \varphi_{0}, \quad \sigma(0) = \sigma_{0} && \text{ in } \Omega,
\end{alignat}
\end{subequations}
where $\pdnu f := \nabla f \cdot \bm{\nu}$ denotes the normal derivative 
of $f$ on the boundary $\pd \Omega$ with unit normal $\bm{\nu}$.  
In the above diffuse interface model for a two-component mixture of tumour cells 
(given by the region $\{\varphi = 1\}$) and healthy host cells 
(given by the region $\{\varphi = -1\}$) that are separated by thin 
transition layers $\{\abs{\varphi} < 1\}$, a nutrient is present, 
whose concentration we denote by the variable $\sigma$, while $\mu$ is a chemical potential for the variable $\varphi$.  The parameter $\beta$ is a positive parameter associated to the surface tension, $\eps$ is a positive parameter measuring the
interfacial thickness, and $m(\varphi) \geq 0$ is a mobility function for $\varphi$.  
The function $\Psi'$ is the derivative of a potential function $\Psi$ which has two equal minima at $\pm 1$.  While the model \eqref{Intro:CH} itself is valid also for three spatial dimensions, the analysis that
we carry out throughout the paper is restricted to two dimensions since several required estimates do not hold in three dimensions.

There are several mechanisms that are driving the dynamics of the tumour and the nutrient.  
In \eqref{CH:1}, the term $\PP f(\varphi) g(\sigma)$ models the growth of the tumour by nutrient consumption, 
where $\PP \geq 0$ can be seen as a tumour proliferation rate, $f$ is a smooth indicator function for the growing front, i.e., the interface between healthy cells and tumour cells, and $g$ models how nutrient is used for proliferation.
In \eqref{CH:3} the term $\CC
h(\varphi) \sigma$ models the consumption of the nutrient only by the tumour cells at a
rate $\CC \geq 0$. Here $h(\varphi)$ is a smooth indicator function for the tumour region.  We refer the reader to Section~\ref{sec:Numerics} for a specific choice of $f$, $g$, and $h$ that are used for numerical simulations.  In \eqref{CH:1} the term $- \chi \div (m(\varphi) \nabla \sigma)$ contributes to the chemotaxis effect \cite{GLNeu, GLSS}, 
that is, tumour cells tend to move towards regions of high nutrient concentration.  

Numerical simulations of \eqref{Intro:CH} and its variants in \cite{CLLW,CLNie,GLNS,GLSS} illustrated that, 
in the presence of the chemotaxis mechanisms, 
the tumour display morphological instabilities akin to the invasive branched tubular structures observed in \cite[Fig. 6C]{Pennacchietti}.  
The morphological instabilities can be explained with the help of a linear stability analysis performed in 
\cite[\S 4]{GLSS}, which shows that in the presence of the chemotaxis mechanisms, 
small initial perturbations may grow in size. 
 While there has been a great progression in recent years in identifying specific genetic
  mutations that has dramatically improved cancer prognosis and treatment, 
  there are still many chemical and biological processes occurring at 
multiple time and spatial scales that are poorly understood.  
Although the multitude of mechanisms described above, namely nutrient consumption and diffusion, 
tumour proliferation and chemotaxis, 
are not derived from first principles but from phenomenological continuum modelling, 
we obtain a tractable description of the growth dynamics in which 
the model can be further analysed and simulated.

Let us mention that the above model \eqref{Intro:CH} differs slightly from the model derived in \cite{GLSS} and studied in \cite{GLDiri,GLNeu}.  
Chiefly, we  neglect the active transport mechanism and the effects of apoptosis.  
The former manifests itself as an additional cross-diffusion-type term $-\chi \Laplace \varphi$
in the right-hand side of \eqref{CH:3}, and the latter as an additional term $-\AA f(\varphi)$ 
on the right-hand side of \eqref{CH:1}, where $\AA$ represents a constant apoptosis rate.  
These are neglected due to the following reasons:
\begin{itemize}
\item A characteristic of malign tumour cells is the genetic feature in which certain regulatory proteins have been switched off after mutation and the cells do not undergo apoptosis.  Thus, often in simulations the value of $\AA$ is small or zero.
\item A linear stability analysis performed in \cite{GLSS} of radially symmetric solutions to the sharp interface models  shows that the chemotaxis mechanism has a more significant effect than the active transport mechanism at developing morphological instabilities.
\item Numerically, we have observed at the active transport mechanism can cause negative values of $\sigma$ 
to appear if $\chi$ is chosen too large. This can be explained from the limiting sharp interface model obtained from sending $\eps \to 0$ in \eqref{Intro:CH}.  
In this singular limit, the domain $\Omega$ is partitioned into 
two subdomains $\Omega_{T} = \{ \varphi \approx 1\}$ and $\Omega_{H} = \{ \varphi \approx -1 \}$, which are separated by a time-dependent hypersurface $\Sigma = \{ \varphi \approx 0 \}$.  It turns out that when the active transport is present, the nutrient $\sigma$ experiences a jump across $\Sigma$ with values equal to $2 \chi$ (see \cite[(3.34)]{GLSS}), i.e., the difference between the traces of $\sigma$ from $\Omega_{T}$ and $\Omega_{H}$ at $\Sigma$ is $2 \chi$.  Hence, for inappropriate values of $\chi$, the values of $\sigma$ in $\Omega_{H}$ may become negative, and this is not desirable as $\sigma$ represents a concentration and thus should be non-negative. 

\end{itemize}

The primary goal of the present work is to develop a methodology that allows practitioners to compare 
model simulations with experimental data.  More precisely, given a set of data, identify the optimal parameter 
values for $\PP$, $\CC$ and $\chi$ so that the resulting model predictions and the data are close in some sense.  
Our approach makes use of standard techniques from parabolic optimal control theory \cite{HPUU, Troltzsch}.  
This approach interprets the inverse problem in the sense of an optimal control problem, 
governed by a system of partial differential equations, namely the tumour model, and the parameters are interpreted as controls.

An alternative approach for parameter identification is the Bayesian inversion
 (also known as Bayesian model calibration) \cite{Allmaras,Madzvamuse,Stuart_Bayes},
  which is a probabilistic framework that incorporates the uncertainties associated 
  to measurements and the relative probabilities of different optimal parameters given the data.  
  In contrast to the optimal control approach, where the output is, in some sense,
the best among all other possible parameters at matching the data, 
the Bayesian approach outputs a posterior distribution that combines 
a priori information and the data.  While certain aspects of the Bayesian approach, 
such as quantification of uncertainty, observing correlations between parameters 
via joint probability distributions and requirement of only weaker notions of 
well-posedness for the forward model, are rather appealing, one drawback is the 
significant increase in computational cost associated to the approximation of
 the posterior distribution, which involves many simulations of the forward model. 
  Therefore, in the current work we restrict ourselves to the optimal control approach 
  and leave the Bayesian approach for future research.  
  We refer the interested reader to the works of \cite{Collis,Hawkins,Lima2, Lima} 
  for the application of the Bayesian framework to tumour models.

The main analytical novelty of our present work is as follows: 
 We include in our analysis a variable mobility for the Cahn--Hilliard equation \eqref{CH:1}. 
While the strong well-posedness to \eqref{Intro:CH}
in two dimensions can be deduced from the results of \cite{LamWu}, 
new difficulties are encountered in order to show the Fr\'{e}chet differentiability of the
 control-to-state mapping in the presence 
of a variable mobility.  
To handle the additional terms arising from the variable mobility, 
we prove high order continuous dependence (see Theorem \ref{thm:Ctsdep} below) 
on the parameters $\PP$, $\chi$ and $\CC$, namely  
$L^{\infty}(0,T;H^{2}(\Omega)) \cap L^{2}(0,T;H^{4}(\Omega))$ for $\varphi$, 
$L^{\infty}(0,T;L^{2}(\Omega)) \cap L^{2}(0,T;H^{2}(\Omega))$ for $\mu$, 
and $L^{\infty}(0,T;H^{1}(\Omega)) \cap L^{2}(0,T;H^{2}(\Omega))$ for $\sigma$.  

\paragraph{Notation.}
We use the notation $L^{p} := L^{p}(\Omega)$ and $W^{k,p} := W^{k,p}(\Omega)$ 
for any $p \in [1,\infty]$, $k > 0$ to denote the standard Lebesgue spaces and Sobolev 
spaces equipped with the norms $\norm{\cdot}_{L^{p}}$ and $\norm{\cdot}_{W^{k,p}}$.  
In the case $p = 2$ we use notation $\norm{\cdot}_{H^{k}} := \norm{\cdot}_{W^{k,2}}$.  
The space $H^{2}_{N}(\Omega)$ denotes the set $\{f \in H^{2}(\Omega) : \pdnu f  = 0$ on $\pd \Omega \}$.  
The space-time cylinder $\Omega \times (0,T)$ is denoted as $Q$, and for any $t \in (0,T)$,
we denote $\Omega \times (0,t)$ as $Q_{t}$. 
For any Banach space $X$, its dual is denoted as $X'$, 
and for any $p \in [1,\infty]$, the norm of the Bochner space $L^{p}(0,T;X)$ will sometimes
be abbreviated as $\norm{\cdot}_{L^{p}(X)}$.  
We will often use the isometric isomorphism $L^{p}(Q_{t}) \cong L^{p}(0,t;L^{p})$ 
for $t \in (0,T]$ and any $p \in [1,\infty)$.  
Furthermore, we use the notation $L^{p}(Q) := L^{p}(Q_{T})$.

\paragraph{Preliminaries.}
We state some inequalities that will be useful in the subsequent analysis:
\begin{itemize}
\item \emph{The Gagliardo--Nirenberg interpolation inequality} for $d = 2$ in bounded domains $\Omega$ with Lipschitz boundary:  For $f \in H^{1}(\Omega)$, there exists a positive constant $C$ depending only on $\Omega$ such that
\begin{align}\label{GN:alt}
\norm{f}_{L^{4}} \leq C \left ( \norm{f}_{L^{2}}^{\frac{1}{2}} \norm{\nabla f}_{L^{2}}^{\frac{1}{2}} + \norm{f}_{L^{2}}\right ).
\end{align}
\item \emph{The Br\'{e}zis--Gallouet interpolation inequality} for $d=2$ in bounded domains $\Omega$ with smooth boundary \cite{BG} or with the cone property \cite{Engler}:  There exists a positive constant $C$ depending only on $\Omega$ such that
\begin{align}\label{Bre}
\norm{g}_{L^{\infty}} \leq C \norm{g}_{H^{1}}\sqrt{\ln (1+ \norm{g}_{H^{2}})} + C \norm{g}_{H^{1}}, \quad \forall \, g \in H^{2}(\Omega).
\end{align}
\item Elliptic estimates: Let $\Omega$ be a convex bounded domain or a bounded domain with Lipschitz boundary.  If $f \in H^{2}(\Omega)$ satisfies $\pdnu f = 0$ on $\pd \Omega$, then there exists a positive constant $C$ depending only on $\Omega$ such that
\begin{align}\label{H2EllEst}
\norm{f}_{H^{2}} \leq C \left ( \norm{\Laplace f}_{L^{2}} + \norm{f}_{L^{2}} \right ).
\end{align}
Let $\Omega$ be a bounded domain with $C^{4}$-boundary.  If $f \in H^{4}(\Omega)$ satisfies $\pdnu f = \pdnu (\Laplace f) = 0$ on $\pd \Omega$, then there exists a positive constant $C$ depending only on $\Omega$ such that
\begin{align}\label{H4EllEst}
\norm{f}_{H^{4}} \leq C \left ( \norm{\Laplace^{2} f}_{L^{2}} + \norm{f}_{L^{2}} \right ).
\end{align}
\end{itemize}


\paragraph{Plan of paper.}  
In Sec.~\ref{sec:State} the strong well-posedness of solutions to \eqref{Intro:CH} is discussed, 
and high order continuous dependence on the parameters $\PP$, $\chi$ and $\CC$ are shown.  
In Sec.~\ref{sec:Min}, the optimization problem is stated and the existence of minimizers is shown via the direct method.  
In Sec.~\ref{sec:Opt}, the linearized state system, the Fr\'{e}chet differentiability of the control-to-state mapping and the adjoint system are studied.  Then, we derive the first order necessary optimality conditions for minimizers.  The results in Secs.~\ref{sec:State}--\ref{sec:Opt} require rather high regularities for the boundary of the domain $\Omega$ and in Sec.~\ref{sec:ConstMob} we show how these requirements can be lowered.  In Sec.~\ref{sec:discreteScheme} we briefly propose a finite element approximation for the numerical realization of the parameter identification problem, that we use in
Sec.~\ref{sec:Numerics} to show some numerical examples.


\section{Strong well-posedness of the state equations}
\label{sec:State}
In view of the analysis performed in the literature \cite{Colli:tumour, GLR:Opt}, 
strong solutions are needed to show Fr\'{e}chet differentiability of the solution 
mapping of the tumour model.  Henceforth, in this section,
 we establish the existence of strong solutions to \eqref{Intro:CH}, 
as well as continuous dependence on the parameters $\PP, \CC, \chi$ 
and initial data $\varphi_{0}, \sigma_{0}$.  

Let us comment that, although our results are restricted to two dimensions, they slightly generalize the results of previous works \cite{Colli:tumour, GLR:Opt} by considering a variable mobility $m(\varphi)$ and non-zero chemotaxis effects $\chi > 0$.

\begin{assump}\label{assump:Wellposed}
\
\begin{enumerate}[label=$(\mathrm{A \arabic*})$, ref = $\mathrm{A \arabic*}$]
\item \label{assump:Initial} $\Omega \subset \R^{2}$ is a bounded domain with $C^{4}$-boundary $\pd \Omega$.  The initial conditions $\varphi_{0}, \sigma_{0}$ belong to the space $H^{3}(\Omega) \cap H^{2}_{N}(\Omega)$.
\item \label{assump:h} The functions $m, h, f, g \in C^{2}(\R)$ are bounded with bounded derivatives and there exists positive constants $n_{0}$ and $n_{1}$, such that $0 \leq h(s)$ and $n_{0} \leq m(s) \leq n_{1}$ for all $s \in \R$.
\item \label{assump:para} The parameters $\PP$, $\chi$ and $\CC$ are non-negative and constant in time and space, and the constants $\beta$ and $\eps$ are positive and fixed.
\item \label{assump:Pot} The potential $\Psi \in C^{3}(\R)$ is non-negative and there exist positive constants $R_{1}$, $R_{2}$, $R_{3}$, $R_{4}$, $R_{5}$ such that for all $s, t \in \R$,
\begin{align}
\notag \Psi(s) \geq R_{1} \abs{s}^{2} - R_{2}, \quad \abs{\Psi'(s)} & \leq R_{3} \left ( 1 + \Psi(s) \right ), \quad \abs{\Psi'''(s)} \leq R_{4} \left ( 1 + \abs{s}^{q-1} \right ), \\
\notag \abs{\Psi'(s) - \Psi'(t)} & \leq R_{5} \left ( 1 + \abs{s}^{r} + \abs{t}^{r} \right ) \abs{s - t}, \\
\notag \abs{\Psi''(s) - \Psi''(t)} & \leq R_{5} \left ( 1 + \abs{s}^{r-1} + \abs{t}^{r-1} \right ) \abs{s - t}, \\
\label{ctsdep:Psi'''} \abs{\Psi'''(s) - \Psi'''(t)} & \leq R_{5} \left ( 1 + \abs{s}^{r-2} + \abs{t}^{r-2} \right ) \abs{s - t}
\end{align}
for some exponents $q \in [1,\infty)$ and $r \in [2,\infty)$.
\end{enumerate}
\end{assump}

\begin{remark}
Let us point out that the $C^{3}$-assumption for $\Psi$ and the $C^{4}$-assumption on $\pd \Omega$ are required to show both existence and continuous dependence of strong solutions to the Cahn--Hilliard system with a variable mobility $m(\varphi)$.  In Sec.~\ref{sec:ConstMob} for a constant mobility $m = 1$, we can relax the requirements to $\Psi \in C^{2}(\R)$ satisfying the above properties aside from \eqref{ctsdep:Psi'''} and $\pd \Omega$ is a convex polygonal boundary.
\end{remark}

\begin{thm}[Strong existence]\label{thm:Exist}
Under Assumption \ref{assump:Wellposed}, there exists a triplet of functions $(\varphi, \mu, \sigma)$ with 
\begin{align*}
\varphi & \in L^{\infty}(0,T;H^{3}(\Omega) \cap H^{2}_{N}(\Omega)) \cap L^{2}(0,T;H^{4}(\Omega)) \cap H^{1}(0,T;L^{2}(\Omega)), \\
\mu & \in L^{\infty}(0,T;H^{1}(\Omega)) \cap L^{2}(0,T;H^{3}(\Omega) \cap H^{2}_{N}(\Omega)) \\
\sigma & \in L^{\infty}(0,T;H^{3}(\Omega) \cap H^{2}_{N}(\Omega)) \cap H^{1}(0,T;H^{2}(\Omega)),
\end{align*}
satisfying $\varphi(0) = \varphi_{0}$, $\sigma(0) = \sigma_{0}$ in $L^{2}(\Omega)$ and \eqref{CH:1}-\eqref{CH:3} a.e. in $Q$ with the property that
\begin{align}\label{Holder}
\varphi, \sigma \in C^{0}([0,T];C^{1,\delta}(\overline{\Omega})) \text{ for any } 0 < \delta < 1.
\end{align}  
Furthermore, there exists a positive constant $C$ not depending on $(\varphi, \mu, \sigma)$ such that
\begin{align}\label{Bounds}
\norm{\varphi}_{L^{\infty}(H^{3}) \cap L^{2}(H^{4}) \cap H^{1}(L^{2})} + \norm{\mu}_{L^{\infty}(H^{1}) \cap L^{2}(H^{3})} + \norm{\sigma}_{L^{\infty}(H^{3}) \cap H^{1}(H^{2})} \leq C.
\end{align}
\end{thm}

\begin{remark}
Let us mention that due to the compact embedding $H^{3}(\Omega) \subset \subset C^{1,\delta}(\overline{\Omega})$ for any $0 < \delta < 1$, and standard compactness results for Bochner spaces we have the compact embedding
\begin{align*}
L^{\infty}(0,T;H^{3}(\Omega)) \cap H^{1}(0,T;L^{2}(\Omega)) \subset \subset C^{0}([0,T];C^{1,\delta}(\overline{\Omega})).
\end{align*}
Furthermore, as the initial conditions $\varphi_{0}, \sigma_{0}$ belong to $H^{3}(\Omega) \subset C^{1,\delta}(\overline{\Omega})$, the assertion \eqref{Holder} makes sense.
\end{remark}

\begin{proof}
The proof for the regularities for $\varphi$ and $\mu$, and for 
\begin{align*}
\sigma \in L^{\infty}(0,T;H^{1}(\Omega)) \cap L^{2}(0,T;H^{2}_{N}(\Omega)) \cap H^{1}(0,T;L^{2}(\Omega))
\end{align*}
can be adapted from the proof of  \cite[Theorem 2]{LamWu} 
by setting $\bm{v} =
\bm{0}$ and $n(\varphi) = 1$, see also \cite[Remark 3.3]{LamWu}.  
Thus, we omit the details and focus on showing 
the new estimates $\sigma \in L^{\infty}(0,T;H^{3} \cap H^{2}_{N})$ 
and $\sigma_{t} \in L^{2}(0,T;H^{2})$.

Testing \eqref{CH:3} with $\Laplace \sigma_{t}$, integrating in time and using that $\sigma_{0} \in H^{2}(\Omega)$ yields
\begin{align*}
& \norm{\Laplace \sigma}_{L^{\infty}(0,T;L^{2})} + \norm{\nabla \sigma_{t}}_{L^{2}(Q)} \\
& \quad \leq C \left ( \norm{\sigma}_{L^{2}(0,T;L^{\infty})} \norm{\nabla \varphi}_{L^{\infty}(0,T;L^{2})} + \norm{\nabla \sigma}_{L^{2}(Q)} + \norm{\Laplace \sigma_{0}}_{L^{2}} \right ),
\end{align*}
after employing the boundedness of $h$ and $h'$.  
Together with the elliptic estimate \eqref{H2EllEst} this implies 
that $\sigma \in L^{\infty}(0,T;H^{2}) \cap H^{1}(0,T;H^{1})$.  
Moreover, since $\sigma_{t} + \CC h(\varphi) \sigma \in L^{2}(0,T;H^{1})$, 
elliptic regularity ensures also that $\sigma \in L^{2}(0,T;H^{3})$.

Next, differentiating \eqref{CH:3} with respect to time, 
then testing with $\sigma_{t}$ and using that 
$\sigma_{t}(0) := \Laplace \sigma_{0} - \CC h(\varphi_{0}) \sigma_{0} \in H^{1}(\Omega)$ 
yields $\sigma_{t} \in L^{\infty}(0,T;L^{2}) \cap L^{2}(0,T;H^{1})$.  
Meanwhile, testing the differentiated equation with $\Laplace \sigma_{t}$ 
gives $\sigma_{t} \in L^{\infty}(0,T;H^{1}) \cap L^{2}(0,T;H^{2})$.  
Now, since $\sigma_{t} + \CC h(\varphi) \sigma \in L^{\infty}(0,T;H^{1})$, 
we deduce from elliptic regularity the assertion $\sigma \in L^{\infty}(0,T;H^{3})$.

The estimate \eqref{Bounds} comes from applying weak/weak* lower semi continuity of the Bochner
norms.  Note that although the constant $C$ depends on the parameters $(\PP, \chi, \CC)$, it does not depend on their reciprocals.
\end{proof}

We now state the continuous dependence result on the parameter $\PP$, $\CC$, $\chi$ and on 
the initial conditions $\varphi_{0}$ and $\sigma_{0}$.  
The following result modifies the argument used in \cite[\S 6]{LamWu}, and it is worth mentioning that, 
in the presence of a variable mobility $m(\varphi)$, we require the high order 
estimates such as $\nabla \varphi \in L^{\infty}(0,T;L^{\infty})$ and $\mu \in L^{\infty}(0,T;H^{1}) \cap L^{2}(0,T;H^{3})$ for the following proof.

\begin{thm}[Continuous dependence]\label{thm:Ctsdep}
Let $\{(\varphi_{i}, \mu_{i}, \sigma_{i})\}_{i = 1}^{2}$ denote two solutions to \eqref{Intro:CH} obtained from Theorem \ref{thm:Exist} corresponding to data $\{ \varphi_{0,i}, \sigma_{0,i}, \PP_{i}, \CC_{i}, \chi_{i}\}_{i=1}^{2}$.  Then, there exists a positive constant C, not depending on the differences $\varphi_{1} - \varphi_{2}$, $\mu_{1} - \mu_{2}$, $\sigma_{1} - \sigma_{2}$, $\PP_{1} - \PP_{2}$, $\chi_{1} - \chi_{2}$, $\CC_{1} - \CC_{2}$, $\varphi_{0,1} - \varphi_{0,2}$ and $\sigma_{0,1} - \sigma_{0,2}$, such that
\begin{equation}\label{Ctsdep:result}
\begin{aligned}
& \norm{\varphi_{1} - \varphi_{2}}_{L^{\infty}(H^{2}) \cap L^{2}(H^{4})} + \norm{\sigma_{1} - \sigma_{2}}_{L^{\infty}(H^{1}) \cap L^{2}(H^{2})} + \norm{\mu_{1} - \mu_{2}}_{L^{\infty}(L^{2}) \cap L^{2}(H^{2})} \\
& \quad \leq C \left ( \abs{\PP_{1} - \PP_{2}} + \abs{\CC_{1} - \CC_{2}} + \abs{\chi_{1} - \chi_{2}} + \norm{\varphi_{0,1} - \varphi_{0,2}}_{H^{2}} + \norm{\sigma_{0,1} - \sigma_{0,2}}_{H^{1}} \right ).
\end{aligned}
\end{equation}
In particular, the solution obtained in Theorem \ref{thm:Exist} is unique.
\end{thm}

Before we give the proof, let us remark that the above 
theorem shows continuous dependence with rather high Sobolev norms, 
even compared to previous works in the literature 
\cite{GLDiri, GLNeu, GLR:Opt, LamWu}. 
 In fact, as we will see later in Section \ref{sec:Fdiff}, 
 to show that the solution operator associated to \eqref{Intro:CH} is 
 Fr\'{e}chet differentiable when the mobility $m$ is variable, 
 the regularities as stated in \eqref{Ctsdep:result} are indispensable.

\begin{proof}
Let $\{(\varphi_{i}, \mu_{i}, \sigma_{i})\}_{i=1}^{2}$ denote two strong solutions to \eqref{Intro:CH} corresponding to the data $\{(\varphi_{0,i}, \sigma_{0,i}, \PP_{i}, \CC_{i}, \chi_{i}\}_{i=1}^{2}$.  Denoting $h_{i} = h(\varphi_{i})$, $\hat{h} := h_{1} - h_{2}$ and likewise for $f$, $g$, $m$, $\Psi'$, the differences $(\hat{\varphi}, \hat{\mu}, \hat{\sigma})$ with initial data $\hat{\varphi}(0) = \hat{\varphi}_{0}$ and $\hat{\sigma}(0) = \hat{\sigma}_{0}$ satisfy
\begin{subequations}
\begin{align}
\hat{\varphi}_{t} & = \div \left ( \hat{m} \nabla (\mu_{1} - \chi_{1} \sigma_{1}) + m_{2} \nabla (\hat{\mu} - \hat{\chi} \sigma_{1} - \chi_{2} \hat{\sigma}) \right ) + \hat{\PP} f_{1} g_{1} + \PP_{2} \hat{f} g_{1} + \PP_{2} f_{2} \hat{g}, \label{Ctsdep:1}\\
\hat{\mu} & = \beta \eps^{-1} \hat{\Psi}' - \beta \eps \Laplace \hat{\varphi}, \label{Ctsdep:2} \\
\hat{\sigma}_{t} & = \Laplace \hat{\sigma} - \hat{\CC} h_{1} \sigma_{1} - \CC_{2} \hat{h} \sigma_{1} - \CC_{2} h_{2} \hat{\sigma}. \label{Ctsdep:3}
\end{align}
\end{subequations}

In the following the symbol $C$ denotes positive constants not depending on the differences $\hat{\varphi}$, $\hat{\mu}$, $\hat{\sigma}$, $\hat{\PP} := \PP_{1} - \PP_{2}$, $\hat{\chi} := \chi_{1} - \chi_{2}$ and $\hat{\CC} := \CC_{1} - \CC_{2}$, and may vary from line to line.  

\paragraph{First estimate.} We test \eqref{Ctsdep:1} with $\beta \eps \hat{\varphi}$, \eqref{Ctsdep:2} with $m(\varphi_{2}) \hat{\mu}$, and $K \hat{\varphi}$, \eqref{Ctsdep:3} with $J \hat{\sigma}$, for positive constants $J,K$ yet to be determined, and upon adding we arrive at the differential inequality (after neglecting the non-negative term $J \CC_{2} h_{2} \abs{\hat{\sigma}}^{2}$)
\begin{align*}
& \frac{1}{2} \frac{\dd}{\dt} \left ( \beta \eps \norm{\hat{\varphi}}_{L^{2}}^{2} + J \norm{\hat{\sigma}}_{L^{2}}^{2} \right ) + J \norm{\nabla \hat{\sigma}}_{L^{2}}^{2} + K \beta \eps \norm{\nabla \hat{\varphi}}_{L^{2}}^{2} + n_{0} \norm{\hat{\mu}}_{L^{2}}^{2} \\
& \quad \leq \int_{\Omega} J \CC_{2} \abs{\sigma_{1}} \abs{\hat{h}} \abs{\hat{\varphi}} + \beta \eps^{-1} K \abs{\hat{\Psi}'} \abs{\hat{\varphi}} + K \abs{\hat{\mu}} \abs{\hat{\varphi}} + \beta \eps^{-1} m_{2} \abs{\hat{\Psi}'} \abs{\hat{\mu}} \dx \\
& \qquad + \int_{\Omega}  \beta \eps \PP_{2} \left ( \abs{g_{1}} \abs{\hat{f}} + \abs{f_{2}} \abs{\hat{g}} \right ) \abs{\hat{\varphi}} + \beta \eps \abs{f_{1}} \abs{g_{1}} \abs{\hat{\PP}} \abs{\hat{\varphi}} + J h_{1} \abs{\sigma_{1}} \abs{\hat{\CC}} \abs{\hat{\sigma}} \dx \\
& \qquad + \int_{\Omega} \beta \eps m_{2} \abs{\hat{\chi}} \abs{\nabla \sigma_{1}} \abs{\nabla \hat{\varphi}} + \beta \eps m_{2} \chi_{2} \abs{\nabla \hat{\sigma}} \abs{\nabla \hat{\varphi}} \dx \\
& \qquad + \int_{\Omega} \beta \eps \abs{m_{2}'} \abs{\nabla \varphi_{2}} \abs{\nabla \hat{\varphi}} \abs{\hat{\mu}} + \beta \eps \abs{\nabla (\mu_{1} - \chi_{1} \sigma_{1})} \abs{\hat{m}} \abs{\nabla \hat{\varphi}} \dx \\
& \quad =: I_{1} + I_{2} + I_{3} + I_{4}. 
\end{align*}
Recalling the boundedness and Lipschitz continuity of $f$, $g$, $h$ and $m$ from \eqref{assump:h}, the Gagliardo--Nirenberg inequality \eqref{GN:alt}, \eqref{assump:Pot}, as well as the boundedness of $\sigma_{i}, \varphi_{i}$ a.e. in $Q$, we observe that for 
\begin{align*}
I_{1} & \leq \frac{n_{0}}{4} \norm{\hat{\mu}}_{L^{2}}^{2} + C\left ( 1 + \norm{\varphi_{i}}_{L^{\infty}(L^{\infty})}^{2r} \right )\norm{\hat{\varphi}}_{L^{2}}^{2}, \\
I_{2} & \leq C \left ( \norm{\hat{\varphi}}_{L^{2}}^{2} + \norm{\hat{\sigma}}_{L^{2}}^{2} + \abs{\hat{\PP}}^{2} + \abs{\hat{\CC}}^{2} \right ), \\
I_{3} & \leq C \norm{\nabla \sigma_{1}}_{L^{\infty}(L^{2})}^{2} \abs{\hat{\chi}}^{2} + \frac{\beta^{2} \eps^{2} n_{1}^{2} \chi_{2}^{2}}{2} \norm{\nabla \hat{\sigma}}_{L^{2}}^{2} + \norm{\nabla \hat{\varphi}}_{L^{2}}^{2}, \\
I_{4} & \leq \frac{\beta^{2} \eps^{2}}{n_{0}} \norm{m'_{2} \nabla \varphi_{2}}_{L^{\infty}(L^{\infty})}^{2} \norm{\nabla \hat{\varphi}}_{L^{2}}^{2} + \frac{n_{0}}{4} \norm{\hat{\mu}}_{L^{2}}^{2} \\
& \quad + C \beta \eps \norm{\nabla (\mu_{1} - \chi_{1} \sigma_{1})}_{L^{4}} \left ( \norm{\hat{\varphi}}_{L^{2}}^{\frac{1}{2}} \norm{\nabla \hat{\varphi}}_{L^{2}}^{\frac{1}{2}} + \norm{\hat{\varphi}}_{L^{2}} \right) \norm{\nabla \hat{\varphi}}_{L^{2}} \\
& \leq \left ( 1 + \frac{\beta^{2} \eps^{2}}{n_{0}} \norm{m'_{2} \nabla \varphi_{2}}_{L^{\infty}(L^{\infty})}^{2} \right ) \norm{\nabla \hat{\varphi}}_{L^{2}}^{2} + \frac{n_{0}}{4} \norm{\hat{\mu}}_{L^{2}}^{2} \\
& \quad + C \left ( 1 + \norm{\nabla (\mu_{1} - \chi_{1} \sigma_{1})}_{L^{4}}^{4} \right ) \norm{\hat{\varphi}}_{L^{2}}^{2}.
\end{align*}
Collecting the estimates then yields the differential inequality
\begin{align*}
& \frac{1}{2} \frac{\dd}{\dt} \left ( \beta \eps \norm{\hat{\varphi}}_{L^{2}}^{2} + J \norm{\hat{\sigma}}_{L^{2}}^{2} \right ) + \left ( J - \frac{\beta^{2} \eps^{2} n_{1}^{2} \chi_{2}^{2}}{2} \right ) \norm{\nabla \hat{\sigma}}_{L^{2}}^{2} \\
& \qquad  + \left ( K \beta \eps - 2 - \frac{\beta^{2} \eps^{2}}{n_{0}} \norm{m'(\varphi_{2})\nabla \varphi_{2}}_{L^{\infty}(L^{\infty})}^{2} \right ) \norm{\nabla \hat{\varphi}}_{L^{2}}^{2} + \frac{n_{0}}{2} \norm{\hat{\mu}}_{L^{2}}^{2} \\
& \quad \leq C \left ( 1 + \norm{\nabla (\mu_{1} - \chi \sigma_{1})}_{L^{4}}^{4} \right )\left ( \norm{\hat{\varphi}}_{L^{2}}^{2} + \norm{\hat{\sigma}}_{L^{2}}^{2} + \abs{\hat{\chi}}^{2} + \abs{\hat{\PP}}^{2} + \abs{\hat{\CC}}^{2} \right ),
\end{align*}
where $C$ is a positive constant depending on $J$, $K$, $\norm{\varphi_{i}}_{L^{\infty}(0,T;L^{\infty})}^{2r}$, $\norm{\sigma_{1}}_{L^{\infty}(0,T;H^{1})}^{2}$.  Choosing $J$ and $K$ sufficiently large so that the coefficients on the left-hand side are positive, and using $\nabla \mu_{1}$, $\nabla \sigma_{1} \in L^{\infty}(0,T;L^{2}) \cap L^{2}(0,T;H^{1}) \subset L^{4}(Q)$ and a Gronwall argument, we arrive at 
\begin{equation}\label{ctsdep:est:1}
\begin{aligned}
& \sup_{t \in (0,T)} \left ( \norm{\hat{\varphi}(t)}_{L^{2}}^{2} + \norm{\hat{\sigma}(t)}_{L^{2}}^{2} \right ) + \norm{\nabla \hat{\sigma}}_{L^{2}(Q)}^{2} + \norm{\nabla \hat{\varphi}}_{L^{2}(Q)}^{2} + \norm{\hat{\mu}}_{L^{2}(Q)}^{2} \\
& \quad \leq C \left ( \abs{\hat{\chi}}^{2} + \abs{\hat{\PP}}^{2} + \abs{\hat{\CC}}^{2} + \norm{\hat{\varphi}_{0}}_{L^{2}}^{2} + \norm{\hat{\sigma}_{0}}_{L^{2}}^{2} \right ) =: C \mathcal{Y}.
\end{aligned}
\end{equation}

\paragraph{Second estimate.}
From the assumption \eqref{assump:Pot} and \eqref{ctsdep:est:1}, we infer that
\begin{align*}
\norm{\hat{\Psi}'}_{L^{2}(Q)}^{2} \leq C \left ( 1 + \norm{\varphi_{i}}_{L^{\infty}(L^{\infty})}^{2r} \right ) \norm{\hat{\varphi}}_{L^{2}(Q)} \leq C \mathcal{Y}.
\end{align*}
Then, viewing \eqref{Ctsdep:2} as an elliptic equation for $\hat{\varphi}$ and by virtue of elliptic regularity,
\begin{align}\label{ctsdep:est:2}
\norm{\hat{\varphi}}_{L^{2}(0,T;H^{2})}^{2} \leq C \left ( \norm{\hat{\varphi}}_{L^{2}(0,T;H^{1})}^{2} + \norm{\hat{\Psi}'}_{L^{2}(Q)}^{2} + \norm{\hat{\mu}}_{L^{2}(Q)}^{2} \right ) \leq C \mathcal{Y}.
\end{align}

\paragraph{Third estimate.}
Testing \eqref{Ctsdep:3} with $\hat{\sigma}_{t}$, integrating in time yields
\begin{equation}\label{ctsdep:est:3}
\begin{aligned}
\norm{\nabla \hat{\sigma}}_{L^{\infty}(0,T;L^{2})}^{2} + \norm{\hat{\sigma}_{t}}_{L^{2}(Q)}^{2} & \leq C \left ( \norm{\sigma_{1}}_{L^{\infty}(0,T;L^{2})}^{2} \abs{\hat{\CC}}^{2} + \norm{\sigma_{1}}_{L^{\infty}(0,T;L^{2})}^{2} \norm{\hat{\varphi}}_{L^{2}(0,T;L^{\infty})}^{2} \right ) \\
& \quad  + C \left ( \norm{\hat{\sigma}}_{L^{2}(Q)}^{2} + \norm{\nabla \hat{\sigma}_{0}}_{L^{2}}^{2} \right ) \leq C \left ( \mathcal{Y} + \norm{\nabla \hat{\sigma}_{0}}_{L^{2}}^{2} \right ),
\end{aligned}
\end{equation}
and via elliptic regularity,
\begin{align}\label{Ctsdep:sigmaH2}
\norm{\hat{\sigma}}_{L^{2}(0,T;H^{2})}^{2} \leq C \left ( \mathcal{Y} + \norm{\nabla \hat{\sigma}_{0}}_{L^{2}}^{2} \right ).
\end{align}

\paragraph{Fourth estimate.} From \eqref{Ctsdep:1}, the boundedness of $\nabla (\mu_{i} - \chi_{i} \sigma_{i})$ in $L^{4}(Q)$, we easily infer that
\begin{equation}\label{Ctsdep:varphi_t}
\begin{aligned}
\norm{\hat{\varphi}_{t}}_{L^{2}(0,t;(H^{1})')} & \leq n_{1} \norm{\nabla \hat{\mu}}_{L^{2}(0,t;L^{2})} + C \norm{\nabla (\mu_{1} - \chi_{1} \sigma_{1})}_{L^{4}(Q)} \norm{\hat{\varphi}}_{L^{4}(Q)} \\
& \qquad  + C \left ( \norm{\hat{\sigma}}_{L^{2}(0,T;H^{1})} + \norm{\hat{\varphi}}_{L^{2}(Q)} + \abs{\hat{\chi}}  + \abs{\hat{\PP}} \right ) \\
& \leq n_{1} \norm{\nabla \hat{\mu}}_{L^{2}(0,t;L^{2})} + C \mathcal{Y}^{\frac{1}{2}}
\end{aligned}
\end{equation}
for any $t \in (0,T)$.  Meanwhile, from \eqref{assump:Pot} and $\nabla \varphi_{i} \in L^{\infty}(0,T;L^{\infty})$, we have
\begin{equation}\label{Ctsdep:nabla:Psi}
\begin{aligned}
\norm{\nabla \hat{\Psi}'}_{L^{2}}^{2} & \leq \int_{\Omega} \abs{\Psi''(\varphi_{1})}^{2} \abs{\nabla \hat{\varphi}}^{2} + \abs{\nabla \varphi_{2}}^{2} \abs{\Psi''(\varphi_{1}) - \Psi''(\varphi_{2})}^{2} \dx \\
& \leq C \left ( 1 + \norm{\varphi_{i}}_{L^{\infty}(L^{\infty})}^{2q} \right ) \norm{\nabla \hat{\varphi}}_{L^{2}}^{2} + C \left ( 1 + \norm{\varphi_{i}}_{L^{\infty}(L^{\infty})}^{2(r-1)} \right ) \norm{\hat{\varphi}}_{L^{2}}^{2}  \\
& \leq C \norm{\hat{\varphi}}_{H^{1}}^{2}.
\end{aligned}
\end{equation}
Then, testing \eqref{Ctsdep:1} with $\hat{\mu}$ and \eqref{Ctsdep:2} with $\hat{\varphi}_{t}$, and upon summing gives
\begin{equation}\label{Ctsdep:est:4:pre}
\begin{aligned}
& \frac{\beta \eps}{2} \norm{\nabla \hat{\varphi}(t)}_{L^{2}}^{2} + \int_{0}^{t} \int_{\Omega} m_{2} \abs{\nabla \hat{\mu}}^{2} \dx \\
& \quad \leq \frac{\beta \eps}{2} \norm{\nabla \hat{\varphi}_{0}}_{L^{2}}^{2} + \norm{\nabla (\mu_{1} - \chi \sigma_{1})}_{L^{4}(Q)} \norm{\hat{m}}_{L^{4}(Q)} \norm{\nabla \hat{\mu}}_{L^{2}(0,t;L^{2})} \\
& \qquad + C \left ( \abs{\hat{\chi}} \norm{\nabla \sigma_{1}}_{L^{2}(Q)}  + \norm{\nabla \hat{\sigma}}_{L^{2}(Q)} \right ) \norm{\nabla \hat{\mu}}_{L^{2}(0,t;L^{2})}  \\
& \qquad + \norm{\hat{\mu}}_{L^{2}(Q)} \left ( \abs{\hat{P}} + \norm{\hat{\varphi}}_{L^{2}(Q)} + \norm{\hat{\sigma}}_{L^{2}(Q)} \right ) + \beta \eps^{-1} \norm{\hat{\Psi}'}_{L^{2}(0,T;H^{1})} \norm{\hat{\varphi}_{t}}_{L^{2}(0,t;(H^{1})')}
\end{aligned}
\end{equation}
for any $t \in (0,T)$.  From \eqref{ctsdep:est:1}, \eqref{Ctsdep:varphi_t} and \eqref{Ctsdep:nabla:Psi}, it holds
\begin{align*}
& \norm{\hat{m}}_{L^{4}(Q)} \leq C \norm{\hat{\varphi}}_{L^{4}(Q)} \leq C \norm{\hat{\varphi}}_{L^{2}(0,T;H^{1})}^{\frac{1}{2}} \norm{\hat{\varphi}}_{L^{\infty}(0,T;L^{2})}^{\frac{1}{2}} \leq C \mathcal{Y}^{\frac{1}{2}}, \\
& \beta \eps^{-1} \norm{\hat{\Psi}'}_{L^{2}(0,T;H^{1})} \norm{\hat{\varphi}_{t}}_{L^{2}(0,t;(H^{1})')} \leq \frac{n_{0}}{2} \norm{\nabla \hat{\mu}}_{L^{2}(0,t;L^{2})}^{2} + C \mathcal{Y},
\end{align*}
so that we obtain from \eqref{Ctsdep:est:4:pre} (also recalling \eqref{Ctsdep:varphi_t})
\begin{align}\label{Ctsdep:nablamu}
\norm{\nabla \hat{\varphi}}_{L^{\infty}(0,T;L^{2})}^{2} + \norm{\nabla \hat{\mu}}_{L^{2}(Q)}^{2} + \norm{\hat{\varphi}_{t}}_{L^{2}(0,T;(H^{1})')}^{2} \leq C \left ( \mathcal{Y} + \norm{\nabla \hat{\varphi}_{0}}_{L^{2}}^{2} \right ).
\end{align}
In light of the above estimate, \eqref{Ctsdep:nabla:Psi} and elliptic regularity yield
\begin{align*}
\norm{\hat{\Psi}'}_{L^{\infty}(0,T;H^{1})}^{2} & \leq C \norm{\hat{\varphi}}_{L^{\infty}(0,T;H^{1})}^{2} \leq C \left ( \mathcal{Y} + \norm{\nabla \hat{\varphi}_{0}}_{L^{2}}^{2} \right ), \\
\norm{\hat{\varphi}}_{L^{2}(0,T;H^{3})} & \leq C \left ( \mathcal{Y} + \norm{\nabla \hat{\varphi}_{0}}_{L^{2}}^{2} \right ).
\end{align*}

\paragraph{Fifth estimate.} 
A short calculation using \eqref{assump:Pot}, $\varphi_{i}, \nabla \varphi_{i} \in L^{\infty}(0,T;L^{\infty})$ and $\Laplace \varphi_{i} \in L^{2}(0,T;H^{2})$ shows that
\begin{equation}\label{Ctsdep:Lap:Psi}
\begin{aligned}
\norm{\Laplace \hat{\Psi}'}_{L^{2}}^{2} & \leq \int_{\Omega} \abs{(\Psi_{1}''' - \Psi_{2}''') \abs{\nabla \varphi_{1}}^{2} + \Psi_{2}''' \nabla \hat{\varphi} \cdot \nabla \varphi_{1} + \Psi_{2}''' \nabla \varphi_{2} \cdot \nabla \hat{\varphi} }^{2} \dx \\
& \quad + \int_{\Omega} \abs{(\Psi_{1}'' - \Psi_{2}'') \Laplace \varphi_{1}  + \Psi_{2}'' \Laplace \hat{\varphi}}^{2} \dx \\
& \leq C \left ( 1 + \norm{\Laplace \varphi_{1}}_{L^{\infty}}^{2} \right ) \left ( \norm{\hat{\varphi}}_{L^{2}}^{2} + \norm{\nabla \hat{\varphi}}_{L^{2}}^{2} + \norm{\Laplace \hat{\varphi}}_{L^{2}}^{2} \right ).
\end{aligned}
\end{equation}
Then, consider testing \eqref{Ctsdep:1} with $\Laplace^{2} \hat{\varphi}$ and keeping in mind the identity
\begin{align*}
\div (m_{2} \nabla \hat{\mu}) \Laplace^{2} \hat{\varphi} = (m'(\varphi_{2}) \nabla \varphi_{2} \cdot \nabla \hat{\mu}) \Laplace^{2} \hat{\varphi} + m_{2} \beta \eps^{-1} \Laplace \hat{\Psi}' \Laplace^{2} \hat{\varphi} - m_{2} \beta \eps \abs{\Laplace^{2} \hat{\varphi}}^{2},
\end{align*}
we arrive at (denoting $L := \mu_{1} - \chi_{1} \sigma_{1}$)
\begin{align*}
& \frac{1}{2} \frac{\dd}{\dt} \norm{\Laplace \hat{\varphi}}_{L^{2}}^{2} + n_{0} \beta \eps \norm{\Laplace^{2} \hat{\varphi}}_{L^{2}}^{2} \\
& \quad \leq \int_{\Omega} \left (\hat{m} \Laplace L + m'(\varphi_{1}) \nabla \hat{\varphi} \cdot \nabla L + (m'(\varphi_{1}) - m'(\varphi_{2})) \nabla \varphi_{2} \cdot \nabla L \right ) \Laplace^{2} \hat{\varphi} \dx \\
& \qquad - \int_{\Omega} \left ( m_{2} \Laplace (\hat{\chi} \sigma_{1} + \chi_{2} \hat{\sigma}) + m'(\varphi_{2}) \nabla \varphi_{2} \cdot \nabla (\hat{\chi} \sigma_{1} + \chi_{2} \hat{\sigma}) \right )\Laplace^{2} \hat{\varphi} \dx \\
& \qquad + \int_{\Omega} \left ( \hat{P} f_{1} g_{1} + \PP_{2} \hat{f} g_{1} + \PP_{2} f_{2} \hat{g} + m'(\varphi_{2}) \nabla \varphi_{2} \cdot \nabla \hat{\mu} + m_{2} \beta \eps^{-1} \Laplace \hat{\Psi}' \right ) \Laplace^{2} \hat{\varphi} \dx.
\end{align*}
Recalling that $\mu_{i}, \sigma_{i} \in L^{\infty}(0,T;H^{1}) \cap L^{2}(0,T;H^{3})$, so that $\Laplace L \in L^{2}(0,T;H^{1})$, $\nabla L \in L^{\infty}(0,T;L^{2}) \cap L^{2}(0,T;H^{2})$, and keeping in mind the continuous dependence estimates \eqref{Ctsdep:sigmaH2}, \eqref{Ctsdep:nablamu} and \eqref{Ctsdep:Lap:Psi}, a short calculation shows that
\begin{align*}
& \frac{1}{2} \norm{\Laplace \hat{\varphi}(s)}_{L^{2}}^{2} - C \int_{0}^{s} \left ( 1 + \norm{\Laplace \varphi_{1}}_{L^{\infty}}^{2} \right ) \norm{\Laplace \hat{\varphi}}_{L^{2}}^{2} \dt + \frac{1}{2} n_{0} \beta \eps \norm{\Laplace^{2} \hat{\varphi}}_{L^{2}(0,s;L^{2})}^{2} \\
& \quad \leq C \norm{\Laplace L}_{L^{2}(L^{3})}^{2} \norm{\hat{\varphi}}_{L^{\infty}(L^{6})}^{2} + C \norm{\nabla L}_{L^{2}(L^{\infty})}^{2} \norm{\hat{\varphi}}_{L^{\infty}(H^{1})}^{2}  \\
& \qquad + C \abs{\hat{\chi}}^{2} \norm{\sigma_{1}}_{L^{2}(H^{2})}^{2} + C \norm{\hat{\sigma}}_{L^{2}(H^{2})}^{2} + C \left ( \abs{\hat{P}}^{2} + \norm{\hat{\varphi}}_{L^{2}(Q)}^{2} + \norm{\hat{\sigma}}_{L^{2}(Q)}^{2} \right ) \\
& \qquad + C \norm{\nabla \hat{\mu}}_{L^{2}(Q)}^{2} + C \left ( 1 + \norm{\Laplace \varphi_{1}}_{L^{2}(L^{\infty})}^{2} \right ) \norm{\hat{\varphi}}_{L^{\infty}(H^{1})}^{2} + C \norm{\Laplace \hat{\varphi}_{0}}_{L^{2}}^{2} \\
& \quad \leq C \left ( \mathcal{Y} + \norm{\hat{\sigma}_{0}}_{H^{1}}^{2} + \norm{\hat{\varphi}_{0}}_{H^{2}}^{2} \right )
\end{align*}
for all $s \in (0,T)$.  A Gronwall argument and by virtue of elliptic regularity \eqref{H2EllEst} and \eqref{H4EllEst} we have
\begin{align*}
\norm{\hat{\varphi}}_{L^{\infty}(0,T;H^{2}) \cap L^{2}(0,T;H^{4})} \leq C \left ( \mathcal{Y} + \norm{\hat{\sigma}_{0}}_{H^{1}}^{2} + \norm{\hat{\varphi}_{0}}_{H^{2}}^{2} \right ).
\end{align*}
Thanks to the estimates for $\hat{\varphi}$ in $L^{\infty}(0,T;H^{2}) \cap L^{2}(0,T;H^{4})$ and the estimate \eqref{Ctsdep:Lap:Psi}, we infer from \eqref{Ctsdep:2} and also taking the Laplacian of \eqref{Ctsdep:2},
\begin{align*}
\norm{\hat{\mu}}_{L^{\infty}(0,T;L^{2}) \cap L^{2}(0,T;H^{2})}  \leq C \left ( \mathcal{Y} + \norm{\hat{\sigma}_{0}}_{H^{1}}^{2} + \norm{\hat{\varphi}_{0}}_{H^{2}}^{2} \right ).
\end{align*}
\end{proof}

\section{The parameter identification problem}\label{sec:Min}

We now consider the parameter identification problem formulated as an optimization problem: 
Given functions $\varphi_{Q} : Q \to \R$ and $\varphi_{\Omega} : \Omega \to \R$, 
$\beta_{Q}$, $\beta_{\Omega}$ non-negative constants such that $\beta_{Q} + \beta_{\Omega} > 0$, 
and non-negative constants $\beta_{\PP}$, $\beta_{\chi}$ and $\beta_{\CC}$ such that $\beta_{\PP} +
\beta_{\chi} + \beta_{\CC} > 0$. 
Further given some fixed non-negative constants $\PP_{d}$, $\chi_{d}$ and $\CC_{d}$, 
which can be seen as a priori knowledge for the parameters.
We define the optimal control problem
\begin{equation}
\tag{$P$} \label{prob:Min:P}
\begin{aligned}
\min J(\varphi, \PP, \chi, \CC) & := 
\frac{\beta_{Q}}{2} \norm{\varphi - \varphi_{Q}}_{L^{2}(Q)}^{2} + \frac{\beta_{\Omega}}{2} \norm{\varphi(T) - \varphi_{\Omega}}_{L^{2}}^{2}  \\
& \quad + \frac{\beta_{\PP}}{2} \abs{\PP - \PP_{d}}^{2} + \frac{\beta_{\chi}}{2} \abs{\chi - \chi_{d}}^{2}
+ \frac{\beta_{\CC}}{2} \abs{\CC - \CC_{d}}^{2}, \\
& \text{ subject to } \varphi \text{ solving } \eqref{Intro:CH} \text{ and } (\PP, \chi, \CC) \in
\mathcal{U}_{\mathrm{ad}},
\end{aligned}
\end{equation}
where, for fixed positive constants $\PP_{\infty}$, $\chi_{\infty}$ and $\CC_{\infty}$, 
we define the admissible set of controls as
\begin{align*}
\mathcal{U}_{\mathrm{ad}} := \{ (\PP, \chi, \CC) \in \R^{3} : 0 \leq \PP \leq \PP_{\infty}, \; 0 \leq \chi \leq \chi_{\infty}, \; 0 \leq \CC \leq \CC_{\infty} \}.
\end{align*}
For the coming mathematical analysis, 
we set $\PP_{d} = \chi_{d} = \CC_{d} = 0$, as the computations 
for the original problem \eqref{prob:Min:P} and the shifted problem are similar.
  
The unique solvability of \eqref{Intro:CH} from Theorem \ref{thm:Exist} and Theorem \ref{thm:Ctsdep} allows us to define a solution operator $\mathcal{S}$ as
\begin{align*}
\mathcal{S}(\PP, \chi, \CC) = (\varphi, \mu, \sigma),
\end{align*}
where $(\varphi, \mu, \sigma)$ is the unique strong solution to \eqref{Intro:CH} corresponding to parameters $(\PP, \chi, \CC)$ and fixed initial data $(\varphi_{0}, \sigma_{0})$.  We use the notation $\mathcal{S}_{1}(\PP, \chi, \CC) = \varphi$ for the first component of $\mathcal{S}(\PP, \chi, \CC)$.

\begin{thm}\label{thm:mini}
Let $\varphi_{Q} \in L^{2}(Q)$ and $\varphi_{\Omega} \in L^{2}(\Omega)$.  Then, there exists at least one minimizer $(\PP_{*}, \chi_{*}, \CC_{*})$ to the optimization problem.  That is, $\varphi_{*} = \mathcal{S}_{1}(\PP_{*}, \chi_{*}, \CC_{*})$ with
\begin{align}\label{Minimizer}
J(\varphi_{*}, \PP_{*}, \chi_{*}, \CC_{*}) = \inf_{\substack{(a,b,c) \; \in \;
\mathcal{U}_{\mathrm{ad}} \\ \text{ s.t. } \phi \; = \; \mathcal{S}_{1}(a,b,c)}} J(\phi, a, b, c).
\end{align}
\end{thm}

\begin{proof}
The proof follows from standard application of the direct method.  We briefly sketch the details.  Let us remark that in the case where $\beta_{\PP}$ is zero, then we treat $\PP$ as a prescribed constant, redefine $\mathcal{U}_{\mathrm{ad}}$ as $\{ (\chi, \CC) \in \R^{2} : 0 \leq \chi \leq \chi_{\infty}, \; 0 \leq \CC \leq \CC_{\infty} \}$ and seek to minimize $J(\varphi, \chi, \CC)$.  Therefore, without loss of generality, in the subsequent analysis we will assume that $\beta_{\PP}$, $\beta_{\chi}$ and $\beta_{\CC}$ are positive.  Then as the functional $J$ is non-negative, this allows us to deduce the
existence of a minimizing sequence $\{\PP_{n}, \chi_{n}, \CC_{n}\}_{n \in \N} \subset \mathcal{U}_{\mathrm{ad}}$ with corresponding solution $\{(\varphi_{n}, \mu_{n}, \sigma_{n})\}_{n \in \N}$ to \eqref{Intro:CH} with fixed initial data $(\varphi_{0}, \sigma_{0})$ such that
\begin{align*}
\lim_{n \to \infty} J(\varphi_{n}, \PP_{n}, \chi_{n}, \CC_{n}) \quad = \inf_{\substack{(a,b,c) \; \in \; \mathcal{U}_{\mathrm{ad}} \\ \text{ s.t. } \phi \; = \; \mathcal{S}_{1}(a,b,c)}} J(\phi, a, b, c).
\end{align*}
By the definition of $\mathcal{U}_{\mathrm{ad}}$, the estimate \eqref{Bounds} and standard compactness results yield
\begin{align*}
\varphi_{n_{j}} & \to \varphi_{*} \text{ strongly in } L^{2}(Q) \cap C^{0}([0,T];L^{2}), \\
\PP_{n_{j}} & \to \PP_{*}, \quad \chi_{n_{j}} \to \chi_{*}, \quad \CC_{n_{j}} \to \CC_{*}
\end{align*}
along subsequences to a limit function $\varphi \in L^{\infty}(0,T;H^{3}) \cap L^{2}(0,T;H^{4}) \cap H^{1}(0,T,;L^{2})$ and limit parameters $(\PP_{*}, \chi_{*}, \CC_{*}) \in \mathcal{U}_{\mathrm{ad}}$.  Applying the weak lower semicontinuity of the $L^{2}(Q)$- and $L^{2}(\Omega)$-norms then leads to \eqref{Minimizer}.
\end{proof}

\section{Optimality conditions}\label{sec:Opt}
For a fixed triplet $(\PP, \chi, \CC) \in \mathcal{U}_{\mathrm{ad}}$ with 
corresponding strong solution $(\varphi, \mu, \sigma)$ to \eqref{Intro:CH}, 
let $\bm{u} := (u_{\PP}, u_{\chi}, u_{\CC}) \in \R^{3}$ be an arbitrary
 vector such that $(\PP_{u}, \chi_{u}, \CC_{u}) \in \mathcal{U}_{\mathrm{ad}}$ 
 where $\PP_{u} := \PP + u_{\PP}$, $\chi_{u} := \chi + u_{\chi}$ and $\CC_{u} := \CC + u_{\CC}$.  
 Denoting the unique strong solution to \eqref{Intro:CH} corresponding to the 
 parameters $(\PP_{u}, \chi_{u}, \CC_{u})$ as $(\varphi_{u}, \mu_{u}, \sigma_{u})$, 
 we now establish the Fr\'{e}chet differentiability of the 
 solution operator $\mathcal{S}$ with respect to $(\PP, \chi, \CC)$.  

\subsection{Solvability of the linearized state equations}
For fixed constants $u_{\PP}$, $u_{\chi}$ and $u_{\CC}$, we study the solvability of following linearized state equations for the variables $(\Phi_{u}, \Xi_{u}, \Sigma_{u})$:
\begin{subequations}\label{Lin:State}
\begin{alignat}{3}
\notag (\Phi_{u})_{t} & = \div ( m(\varphi) \nabla (\Xi_{u} - \chi \Sigma_{u} - u_{\chi} \nabla \sigma) + m'(\varphi) \Phi_{u} \nabla (\mu - \chi \sigma)) \\
& \quad + \PP ( g(\sigma) f'(\varphi) \Phi_{u} + f(\varphi) g'(\sigma) \Sigma_{u}) + u_{\PP} f(\varphi) g(\sigma) & \text{ in } Q, \label{Lin:1} \\
\Xi_{u} & = \beta \eps^{-1} \Psi''(\varphi) \Phi_{u} - \beta \eps \Laplace \Phi_{u}, & \text{ in } Q,  \label{Lin:2} \\
(\Sigma_{u})_{t} & = \Laplace \Sigma_{u} - \CC( h'(\varphi) \Phi_{u} \sigma + h(\varphi) \Sigma_{u}) - u_{\CC} h(\varphi) \sigma & \text{ in } Q,  \label{Lin:3} \\
0 & = \pdnu \Phi_{u} = \pdnu \Xi_{u} = \pdnu \Sigma_{u} & \text{ on } \Gamma, \\
\Phi_{u}(0) & = 0, \quad \Sigma_{u}(0) = 0 & \text{ in } \Omega.
\end{alignat}
\end{subequations}

\begin{thm}\label{thm:Lin:state}
For any $(u_{\PP}, u_{\chi}, u_{\CC}) \in \R^{3}$, there exists a unique triplet $(\Phi_{u}, \Xi_{u}, \Sigma_{u})$ with
\begin{align*}
\Phi_{u} & \in L^{\infty}(0,T;H^{1}(\Omega)) \cap L^{2}(0,T;H^{3}(\Omega) \cap H^{2}_{N}(\Omega)) \cap H^{1}(0,T;(H^{1}(\Omega))'), \\
\Xi_{u} & \in L^{2}(0,T;H^{1}(\Omega)), \\
\Sigma_{u} & \in L^{\infty}(0,T;H^{1}(\Omega)) \cap L^{2}(0,T;H^{2}_{N}(\Omega)) \cap H^{1}(0,T;L^{2}(\Omega)),
\end{align*}
satisfying $\Phi_{u}(0) = 0$, $\Sigma_{u}(0) = 0$ in $L^{2}(\Omega)$, \eqref{Lin:2}, \eqref{Lin:3} a.e. in $Q$,
\begin{equation}
  \label{Lin:1:weak}
\begin{aligned}
\inner{(\Phi_{u})_{t}}{\zeta}_{H^{1}} & = \int_{\Omega} - \left ( m(\varphi) \nabla (\Xi_{u} - \chi \Sigma_{u} - u_{\chi} \nabla \sigma) + m'(\varphi) \Phi_{u} \nabla (\mu - \chi \sigma) \right ) \cdot \nabla \zeta \dx \\
& \quad + \int_{\Omega} \left ( \PP ( g(\sigma) f'(\varphi) \Phi_{u} + f(\varphi) g'(\sigma) \Sigma_{u}) + u_{\PP} f(\varphi) g(\sigma) \right ) \zeta \dx
\end{aligned}
\end{equation}
for a.e. $t \in (0,T)$ and for all $\zeta \in H^{1}(\Omega)$.  Furthermore, there exists a positive constant $C$, not depending on $(\Phi_{u}, \Xi_{u}, \Sigma_{u}, u_{\PP}, u_{\chi}, u_{\CC})$ such that
\begin{equation}
\label{Bounds:Lin}
\begin{aligned}
& \norm{\Phi_{u}}_{L^{\infty}(0,T;H^{1}) \cap L^{2}(0,T;H^{3}) \cap H^{1}(0,T;(H^{1})')} + \norm{\Xi_{u}}_{L^{2}(0,T;H^{1})} \\
& \qquad + \norm{\Sigma_{u}}_{L^{\infty}(0,T;H^{1}) \cap L^{2}(0,T;H^{2}) \cap H^{1}(0,T;L^{2})} \\
& \quad \leq C \left ( \abs{u_{\PP}} + \abs{u_{\chi}} + \abs{u_{\CC}} \right ).
\end{aligned}
\end{equation}
\end{thm}

\begin{proof}
It suffices to derive the a priori estimates that are necessary for a Galerkin procedure.  In the following the symbol $C$ will denote positive constants not depending on $(\Phi_{u}, \Xi_{u}, \Sigma_{u}, u_{\PP}, u_{\chi}, u_{\CC})$ and may vary from line to line.

\paragraph{First estimate.} 
Let $D, F$ be  positive constants yet to be determined.  
Testing \eqref{Lin:1} with $D \beta \eps \Phi_{u}$, \eqref{Lin:2} 
with $-\Laplace \Phi_{u}$ and also with $D m(\varphi) \Xi_{u}$, and \eqref{Lin:3} with $F \Sigma_{u}$, 
then we obtain after summing up the resulting equalities
\begin{align*}
& \frac{\dd}{\dt} \frac{1}{2} \left ( D \beta \eps \norm{\Phi_{u}}_{L^{2}}^{2} + F \norm{\Sigma_{u}}_{L^{2}}^{2} \right ) + \beta \eps \norm{\Laplace \Phi_{u}}_{L^{2}}^{2}  \\
& \qquad + F \norm{\nabla \Sigma_{u}}_{L^{2}}^{2} + D \norm{m^{\frac{1}{2}}(\varphi) \Xi_{u}}_{L^{2}}^{2} + F \CC \norm{h^{\frac{1}{2}}(\varphi) \Sigma_{u}}_{L^{2}}^{2} \\
& \quad = \int_{\Omega} - F \CC h'(\varphi) \sigma \Phi_{u} \Sigma_{u} - F u_{\CC} h(\varphi) \sigma \Sigma_{u} + \beta \eps^{-1} \Psi''(\varphi) \Phi_{u} \Laplace \Phi_{u} - \Xi_{u} \Laplace \Phi_{u} \dx \\
& \qquad + D \int_{\Omega}  m(\varphi) \beta \eps^{-1} \Psi''(\varphi) \Phi_{u} \Xi_{u} +  \beta \eps m'(\varphi) \Xi_{u}  \nabla \varphi \cdot \nabla \Phi_{u} + \beta \eps m(\varphi) \chi \nabla \Sigma_{u} \cdot \nabla \Phi_{u} \dx \\
& \qquad + D \beta \eps \int_{\Omega} u_{\chi} m(\varphi) \nabla \sigma \cdot \nabla \Phi_{u} - m'(\varphi) \Phi_{u} \nabla (\mu - \chi \sigma) \cdot \nabla \Phi_{u} \dx \\
& \qquad + D \beta \eps \int_{\Omega} \PP g(\sigma) f'(\varphi) \abs{\Phi_{u}}^{2} + \PP f(\varphi) g'(\sigma) \Sigma_{u} \Phi_{u} + u_{\PP} f(\varphi) g(\sigma) \Phi_{u} \dx \\
& \quad =: J_{1} + J_{2} + \dots + J_{12}.
\end{align*}
At this point let us mention a slight technicality.  Consider the eigenfunctions of the Neumann-Laplacian as a basis for a Galerkin approximation, and denoting as $W_{n}$ the finite dimensional subspace of $H^{1}(\Omega)$ spanned by the first $n$ basis functions with the corresponding projection operator $\Pi_{n}$.  Let $\Phi_{u,n}$, $\Xi_{u,n}$ denote the Galerkin approximation of $\Phi_{u}$ and $\Xi_{u}$, respectively, which are finite linear combinations of the basis functions.  Then, to arrive at the above equality for the Galerkin approximations, one should test \eqref{Lin:2} with $D \Pi_{n}(m(\varphi) \Xi_{u,n})$.  Since $\Laplace \Phi_{u,n} \in W_{n}$, it holds that $\Pi_{n}(\Laplace \Phi_{u,n}) = \Laplace \Phi_{u,n}$ and so
\begin{align*}
\int_{\Omega} \beta \eps \Laplace \Phi_{u,n} \Pi_{n}( m(\varphi) \Xi_{u,n}) \dx = \int_{\Omega} \beta \eps \Laplace \Phi_{u,n} m(\varphi) \Xi_{u,n} \dx.
\end{align*}

Let us recall that from Theorem \ref{thm:Exist}, $\varphi$ and $\nabla \varphi$ are bounded a.e. in $Q$, and so $h$, $f$ and $\Psi'$, and their first derivatives are bounded a.e. in $Q$.  We now estimate the terms $J_{1}, \dots, J_{12}$ in the following way:  
\begin{align*}
J_{1} + J_{2} & \leq C F \norm{\sigma}_{L^{\infty}} \norm{\Phi_{u}}_{L^{2}} \norm{\Sigma_{u}}_{L^{2}} + C\abs{u_{\CC}} \norm{\sigma}_{L^{2}} \norm{\Sigma_{u}}_{L^{2}} \\
& \leq C F \left (\norm{\sigma}_{H^{2}}^{2} \norm{\Sigma_{u}}_{L^{2}}^{2} + \norm{\Phi_{u}}_{L^{2}}^{2} + \abs{u_{\CC}}^{2} \right ),  \\
J_{3} + J_{4} & \leq C \norm{\Phi_{u}}_{L^{2}}^{2} + \frac{\beta \eps}{2} \norm{\Laplace \Phi_{u}}_{L^{2}}^{2} + (\beta \eps)^{-1} \norm{\Xi_{u}}_{L^{2}}^{2}, \\
J_{5} + J_{6} & \leq DC \norm{\Phi_{u}}_{L^{2}}^{2} + \frac{D n_{0}}{2} \norm{\Xi_{u}}_{L^{2}}^{2} + D C \norm{\nabla \Phi_{u}}_{L^{2}}^{2}, \\
J_{7} + J_{8} & \leq DC \left (\norm{\nabla \Sigma_{u}}_{L^{2}}^{2} + \norm{\nabla \Phi_{u}}_{L^{2}}^{2} + \abs{u_{\chi}}^{2} \norm{\nabla \sigma}_{L^{2}}^{2} \right ), \\
J_{9} & \leq DC \norm{\nabla (\mu - \chi \sigma)}_{L^{4}} \norm{\nabla \Phi_{u}}_{L^{2}} \left (\norm{\Phi_{u}}_{L^{2}}^{\frac{1}{2}} \norm{\nabla \Phi_{u}}_{L^{2}}^{\frac{1}{2}} + \norm{\Phi_{u}}_{L^{2}} \right ) \\
& \leq DC \left ( 1 + \norm{\nabla (\mu - \chi \sigma)}_{L^{4}}^{4} \norm{\Phi_{u}}_{L^{2}}^{2}
\right ) + \norm{\nabla \Phi_{u}}_{L^{2}}^{2}, \\
J_{10} + J_{11} + J_{12} & \leq DC \left ( \norm{\Phi_{u}}_{L^{2}}^{2} + \norm{\Sigma_{u}}_{L^{2}}^{2} + \abs{u_{\PP}}^{2} \right ).
\end{align*}
Collecting the terms we obtain
\begin{align*}
& \frac{1}{2} \frac{\dd}{\dt} \left ( D \beta \eps \norm{\Phi_{u}}_{L^{2}}^{2} + \norm{\Sigma_{u}}_{L^{2}}^{2} \right ) - C \left ( 1 + \norm{\sigma}_{H^{2}}^{2} + D\norm{\nabla (\mu - \chi \sigma)}_{L^{4}}^{4} \right ) \left ( \norm{\Phi_{u}}_{L^{2}}^{2} + \norm{\Sigma_{u}}_{L^{2}}^{2} \right ) \\
& \quad + \frac{\beta \eps}{2} \norm{\Laplace \Phi_{u}}_{L^{2}}^{2} +  \left ( \frac{Dn_{0}}{2} - (\beta \eps)^{-1} \right ) \norm{\Xi_{u}}_{L^{2}}^{2} + \left ( F - DC \right ) \norm{\nabla \Sigma_{u}}_{L^{2}}^{2} \\
& \quad \leq (1 + DC) \norm{\nabla \Phi_{u}}_{L^{2}}^{2} + DC \left ( \norm{\nabla \sigma}_{L^{\infty}(L^{2})}^{2} \abs{u_{\chi}}^{2} + \abs{u_{\PP}}^{2} + \abs{u_{\CC}}^{2} \right ).
\end{align*}
Choosing $D > \frac{2}{n_{0} \beta \eps }$ and then $F > DC$ so that upon using the inequality
\begin{align}\label{nablaL2:Laplace}
\norm{\nabla \Phi}_{L^{2}}^{2} = \int_{\Omega} \Phi \Laplace \Phi \dx \leq \norm{\Phi}_{L^{2}} \norm{\Laplace \Phi}_{L^{2}},
\end{align}
we obtain
\begin{align*}
& \frac{\dd}{\dt} \left ( \norm{\Phi_{u}}_{L^{2}}^{2} + \norm{\Sigma_{u}}_{L^{2}}^{2} \right ) + \norm{\Laplace \Phi_{u}}_{L^{2}}^{2} + \norm{\Xi_{u}}_{L^{2}}^{2} + \norm{\nabla \Sigma_{u}}_{L^{2}}^{2} \\
& \quad \leq C \left ( 1 + \norm{\sigma}_{H^{2}}^{2} + \norm{\nabla (\mu - \chi \sigma)}_{L^{4}}^{4} \right ) \left ( \norm{\Phi_{u}}_{L^{2}}^{2} + \norm{\Sigma_{u}}_{L^{2}}^{2} \right )  + C \left ( \abs{u_{\chi}}^{2} + \abs{u_{\PP}}^{2} + \abs{u_{\CC}}^{2} \right ).
\end{align*}
Applying a Gronwall argument and recalling that $\sigma \in L^{2}(0,T;H^{2})$, $\nabla (\mu - \chi \sigma) \in L^{4}(Q)$ yields 
\begin{equation}\label{Apri:lin:1}
\begin{aligned}
& \norm{\Phi_{u}}_{L^{\infty}(0,T;L^{2}) \cap L^{2}(0,T;H^{2})} + \norm{\Sigma_{u}}_{L^{\infty}(0,T;L^{2}) \cap L^{2}(0,T;H^{1})} + \norm{\Xi_{u}}_{L^{2}(Q)} \\
& \quad \leq C \left ( \abs{u_{\PP}} + \abs{u_{\chi}} + \abs{u_{\CC}} \right ).
\end{aligned}
\end{equation}

\paragraph{Second estimate.} Testing \eqref{Lin:2} with $(\Sigma_{u})_{t}$ gives
\begin{align*}
\frac{\dd}{\dt} \frac{1}{2} \norm{\nabla \Sigma_{u}}_{L^{2}}^{2} + \norm{(\Sigma_{u})_{t}}_{L^{2}}^{2} \leq C \norm{(\Sigma_{u})_{t}}_{L^{2}} \left ( \norm{\sigma}_{L^{\infty}(L^{4})} \norm{\Phi_{u}}_{L^{4}} + \norm{\Sigma_{u}}_{L^{2}} + \abs{u_{\CC}} \norm{\sigma}_{L^{\infty}(L^{2})} \right ).
\end{align*}
Applying Young's inequality and using \eqref{Apri:lin:1} yields that $\Sigma_{u}$ is bounded in $L^{\infty}(0,T;H^{1}) \cap H^{1}(0,T;L^{2})$.  Then, viewing \eqref{Lin:3} as an elliptic equation with right-hand side belonging to $L^{2}(Q)$, we obtain altogether
\begin{align}\label{Apri:lin:2}
\norm{\Sigma_{u}}_{L^{\infty}(0,T;H^{1}) \cap L^{2}(0,T;H^{2}) \cap H^{1}(0,T;L^{2})} \leq C \left ( \abs{u_{\PP}} + \abs{u_{\chi}} + \abs{u_{\CC}} \right ).
\end{align}

\paragraph{Third estimate.}  Testing \eqref{Lin:1} with an arbitrary test function $\zeta \in L^{2}(0,T;H^{1})$ yields that
\begin{equation}\label{pdt:Phi:u}
\begin{aligned}
\norm{(\Phi_{u})_{t}}_{L^{2}(0,T;(H^{1})')} & \leq n_{1} \norm{\nabla \Xi_{u}}_{L^{2}(Q)} + n_{1} \chi \norm{\nabla \Sigma_{u}}_{L^{2}(Q)} + \abs{u_{\chi}} n_{1} \norm{\nabla \sigma}_{L^{2}(Q)} \\
& \quad + C \norm{\Phi_{u}}_{L^{2}(L^{\infty})} \norm{\nabla (\mu - \chi \sigma)}_{L^{\infty}(L^{2})} \\
& \quad + C \left ( \norm{\Phi_{u}}_{L^{2}(Q)} + \norm{\Sigma_{u}}_{L^{2}(Q)} + \abs{u_{\PP}} \right ).
\end{aligned}
\end{equation}
Then, upon testing \eqref{Lin:1} with $\Xi_{u}$ and \eqref{Lin:2} with $-(\Phi_{u})_{t}$ leads to
\begin{align*}
& \frac{\dd}{\dt} \frac{\beta \eps}{2} \norm{\nabla \Phi_{u}}_{L^{2}}^{2} + \norm{m^{\frac{1}{2}}(\varphi) \nabla \Xi_{u}}_{L^{2}}^{2} \\
& \quad = \int_{\Omega} m(\varphi) \left ( \chi \nabla \Sigma_{u} + u_{\chi} \nabla \sigma \right ) \cdot \nabla \Xi_{u} - m'(\varphi) \Phi_{u} \nabla (\mu - \chi \sigma) \cdot \nabla \Xi_{u} \dx \\
& \qquad + \int_{\Omega} \PP (g(\sigma) f'(\varphi) \Phi_{u} + f(\varphi) g'(\sigma) \Sigma_{u} ) \Xi_{u} + u_{\PP} f(\varphi) g(\sigma) \Xi_{u} \dx \\
& \qquad - \int_{\Omega} \beta \eps^{-1} \Psi''(\varphi) \Phi_{u} (\Phi_{u})_{t} \dx \\
& = : K_{1} + K_{2} + K_{3}.
\end{align*}
Thanks to the fact that $\mu, \sigma \in L^{\infty}(0,T;H^{1})$, in applying the estimate \eqref{pdt:Phi:u} we have
\begin{align*}
K_{1} & \leq C \left ( \norm{\nabla \Sigma_{u}}_{L^{2}}^{2} + \abs{u_{\chi}}^{2} \norm{\nabla \sigma}_{L^{2}}^{2} + \norm{\nabla (\mu - \chi \sigma)}_{L^{\infty}(L^{2})}^{2} \norm{\Phi_{u}}_{L^{\infty}}^{2} \right ) + \frac{n_{0}}{4} \norm{\nabla \Xi_{u}}_{L^{2}}^{2}, \\
K_{2} & \leq C \left ( \norm{\Phi_{u}}_{L^{2}}^{2} + \norm{\Sigma_{u}}_{L^{2}}^{2} + \norm{\Xi_{u}}_{L^{2}}^{2} + \abs{u_{\PP}}^{2} \right ), \\
K_{3} & \leq C \left ( \norm{\Psi''(\varphi)}_{L^{\infty}(L^{\infty})}\norm{\Phi_{u}}_{H^{1}} + \norm{\Psi'''(\varphi)}_{L^{\infty}(L^{\infty})} \norm{\nabla \varphi}_{L^{\infty}(L^{3})} \norm{\Phi_{u}}_{L^{6}} \right ) \norm{(\Phi_{u})_{t}}_{(H^{1})'} \\
& \leq C \left ( \norm{\Phi_{u}}_{H^{1}}^{2} + \norm{ \Sigma_{u}}_{H^{1}}^{2} + \abs{u_{\chi}}^{2} \norm{\nabla \sigma}_{L^{2}}^{2} \right) + \frac{n_{0}}{4} \norm{\nabla \Xi_{u}}_{L^{2}}^{2}\\
& \quad + C \left ( \norm{\Phi_{u}}_{L^{\infty}}^{2} \norm{\nabla (\mu - \chi \sigma)}_{L^{\infty}(L^{2})}^{2} + \abs{u_{\PP}}^{2} \right ) .
\end{align*}
Collecting the terms yields the differential inequality
\begin{align*}
\frac{\dd}{\dt} \norm{\nabla \Phi_{u}}_{L^{2}}^{2} + \norm{\nabla \Xi_{u}}_{L^{2}}^{2} \leq C \left ( \norm{\Sigma_{u}}_{H^{1}}^{2} + \norm{\Phi_{u}}_{H^{2}}^{2} +  \norm{\Xi_{u}}_{L^{2}}^{2} + \abs{u_{\PP}}^{2} + \abs{u_{\chi}}^{2} \right ),
\end{align*}
and a Gronwall argument with \eqref{Apri:lin:1}, \eqref{Apri:lin:2}, whilst keeping in mind \eqref{pdt:Phi:u} leads to
\begin{align}
\label{Apri:lin:3}
\norm{\Phi_{u}}_{L^{\infty}(0,T;H^{1}) \cap H^{1}(0,T;(H^{1})')} + \norm{\Xi_{u}}_{L^{2}(0,T;H^{1})}  \leq C \left ( \abs{u_{\PP}} + \abs{u_{\chi}} + \abs{u_{\CC}} \right ).
\end{align}

\paragraph{Fourth estimate.} Viewing \eqref{Lin:2} as an elliptic equation for $\Phi_{u}$ and observing that the right-hand side now belongs to $L^{2}(0,T;H^{1})$ (thanks to \eqref{Apri:lin:3} and the boundedness of $\nabla \varphi$ a.e. in $Q$) it holds that
\begin{align}
\label{Apri:lin:4}
\norm{\Phi_{u}}_{L^{2}(0,T;H^{3})} \leq C\left ( \abs{u_{\PP}} + \abs{u_{\chi}} + \abs{u_{\CC}} \right ).
\end{align}

\paragraph{Uniqueness.} As \eqref{Lin:State} is a system of equations that is linear in $(\Phi_{u}, \Xi_{u}, \Sigma_{u})$, it suffices to show that $\Phi_{u} = \Xi_{u} = \Sigma_{u} = 0$ when $u_{\PP} = u_{\chi} = u_{\CC} = 0$.  Thanks to the regularities stated in Theorem \ref{thm:Lin:state}, the testing procedures to derive \eqref{Apri:lin:1} remain valid.  Then, substituting $u_{\PP} = u_{\chi} = u_{\CC} = 0$ yields that $\Phi_{u} = \Xi_{u} = \Sigma_{u} = 0$ a.e. in $Q$.
\end{proof}

\subsection{Fr\'{e}chet differentiability of the control-to-state map}\label{sec:Fdiff}

\begin{thm}\label{thm:Fdiff}
Under Assumption \ref{assump:Wellposed}, for any $(u_{\PP}, u_{\chi}, u_{\CC}) \in \R^{3}$ such that $(\PP_{u}, \chi_{u}, \CC_{u}) \in \mathcal{U}_{\mathrm{ad}}$, there exists a positive constant $C$ not depending on $(u_{\PP}, u_{\chi}, u_{\CC})$ such that
\begin{align*}
\norm{(\theta_{u}, \rho_{u}, \xi_{u})}_{\mathcal{Y}} \leq C \left ( \abs{u_{\PP}}^{2} + \abs{u_{\chi}}^{2} + \abs{u_{\CC}}^{2} \right ),
\end{align*}
where $\theta_{u} := \varphi_{u} - \varphi - \Phi_{u}$, $\rho_{u} := \mu_{u} - \mu - \Xi_{u}$, $\xi_{u} := \sigma_{u} - \sigma - \Sigma_{u}$, and $\mathcal{Y}$ is the product Banach space
\begin{align*}
\mathcal{Y} & := \left [ L^{2}(0,T;H^{2}(\Omega)) \cap H^{1}(0,T;(H^{2}_{N}(\Omega))') \cap C^{0}([0,T];L^{2}(\Omega)) \right ] \\
& \quad \times L^{2}(Q) \times \left [ L^{2}(0,T;H^{2}(\Omega)) \cap L^{\infty}(0,T;H^{1}(\Omega)) \cap H^{1}(0,T;L^{2}(\Omega)) \right ].
\end{align*}
In particular, the solution operator $\mathcal{S} : \R^{3} \to \mathcal{Y}$ is Fr\'{e}chet differentiable.
\end{thm}

\begin{proof}
First we recall Taylor's theorem with integral remainder for $F \in C^{2}(\R)$ and $a, x  \in \R$:
\begin{align*}
F(x) = F(a) + F'(a)(x-a) + (x-a)^{2} \int_{0}^{1} F''(a+z(x-a))(1-z) \dz.
\end{align*}
Then, using the definitions of $\theta_{u}$, $\rho_{u}$ and $\xi_{u}$ we have
\begin{align*}
f(\varphi_{u}) - f(\varphi) - f'(\varphi) \Phi_{u} & = f'(\varphi) \theta_{u} + (\varphi_{u} - \varphi)^{2} R_{f},
\end{align*}
where 
\begin{align*}
R_{f} := \int_{0}^{1} f''(\varphi + z(\varphi_{u} - \varphi)) (1-z) \dz,
\end{align*}
and similar relations for $m$, $h$, $g$, $\Psi'$ with remainders $R_{m}$, $R_{h}$, $R_{g}$ and $R_{\Psi}$, respectively also hold.  Thanks to the boundedness of $\varphi_{u}$ and $\varphi$ a.e. in $Q$, and the boundedness of $g''$, it is easily to infer that $R_{f}$, $R_{h}$, $R_{g}$, $R_{\Psi}$ are bounded a.e. in $Q$.

Next, to determine the equations satisfied by $(\theta_{u}, \rho_{u}, \xi_{u})$, note that
\begin{align*}
& \CC_{u} h(\varphi_{u}) \sigma_{u} - \CC h(\varphi) \sigma - \CC \left ( h'(\varphi) \Phi_{u} \sigma + h(\varphi) \Sigma_{u} \right ) - u_{\CC} h(\varphi) \sigma \\
& \quad = (\CC_{u} - \CC)(h(\varphi_{u}) - h(\varphi))(\sigma_{u} - \sigma) + \CC (h(\varphi_{u}) - h(\varphi))(\sigma_{u} - \sigma)  \\
& \qquad  + (\CC_{u} - \CC) (h(\varphi_{u}) - h(\varphi)) \sigma + \CC( h(\varphi_{u}) - h(\varphi) - h'(\varphi) \Phi_{u}) \sigma \\
& \qquad + (\CC_{u} - \CC) h(\varphi) (\sigma_{u} - \sigma) + \CC h(\varphi) (\sigma_{u} - \sigma - \Sigma_{u}) + \underbrace{(\CC_{u} - \CC - u_{\CC})}_{=0} h(\varphi) \sigma \\
& \quad = u_{\CC} \left ( (h(\varphi_{u}) - h(\varphi)) (\sigma_{u} - \sigma) + \sigma (h(\varphi_{u}) - h(\varphi)) + h(\varphi)(\sigma_{u} - \sigma) \right ) \\
& \qquad + \CC \left ( (h(\varphi_{u}) - h(\varphi))(\sigma_{u} - \sigma) + \sigma \left [h'(\varphi)\theta_{u} + (\varphi_{u} - \varphi)^{2} R_{h} \right ] + h(\varphi) \xi_{u} \right ) \\
& \quad =: X_{\sigma}.
\end{align*}
With a similar calculation we have
\begin{align*}
& \PP_{u} f(\varphi_{u}) g(\sigma_{u}) - \PP f(\varphi) g(\sigma) \\
& \qquad  - \PP (g(\sigma) f'(\varphi) \Phi_{u} + f(\varphi) g'(\sigma) \Sigma_{u}) - u_{\PP} f(\varphi) g(\sigma) \\
& \quad = u_{\PP} \left ( (f(\varphi_{u}) - f(\varphi))(g(\sigma_{u}) - g(\sigma)) + g(\sigma)(f(\varphi_{u}) - f(\varphi)) \right ) \\
& \qquad + u_{\PP} f(\varphi)(g(\sigma_{u}) - g(\sigma)) + \PP f(\varphi) \left [g'(\sigma) \xi_{u} + (\sigma_{u} - \sigma)^{2} R_{g} \right ] \\
& \qquad + \PP \left ( (f(\varphi_{u}) - f(\varphi))(g(\sigma_{u}) - g(\sigma)) + g(\sigma) \left [ f'(\varphi) \theta_{u} + (\varphi_{u} - \varphi)^{2} R_{f} \right ] \right ) \\
& \quad =: X_{\varphi}
\end{align*}
and
\begin{align*}
& m(\varphi_{u}) \nabla (\mu_{u} - \chi_{u} \sigma_{u}) - m(\varphi) \nabla (\mu - \chi \sigma) \\
& \qquad - m(\varphi) \nabla (\Xi_{u} - \chi \Sigma_{u} - u_{\chi} \sigma) - m'(\varphi) \Phi_{u} \nabla (\mu - \chi \sigma) \\
& \quad = (m(\varphi_{u}) - m(\varphi)) \nabla (\mu_{u} - \mu - (\chi_{u} - \chi)(\sigma_{u} - \sigma)) \\
& \qquad + \left [ m''(\varphi) \theta_{u} + (\varphi_{u} - \varphi)^{2} R_{m} \right ] \nabla (\mu - \chi \sigma) + m(\varphi) \nabla (\rho_{u} - \chi \xi_{u}) \\
& \qquad - \chi (m(\varphi_{u}) - m(\varphi)) \nabla (\sigma_{u} - \sigma) \\
& \qquad - u_{\chi} \left ( (m(\varphi_{u}) - m(\varphi)) \nabla \sigma + m(\varphi) \nabla (\sigma_{u} - \sigma) \right ) \\
& \quad =: \bm{X}_{m} + m(\varphi) \nabla \rho_{u}.
\end{align*}
Then, from the regularities stated in Theorem \ref{thm:Exist} and Theorem \ref{thm:Lin:state}, $(\theta_{u}, \rho_{u}, \xi_{u})$ satisfies 
\begin{align*}
\theta_{u} & \in L^{\infty}(0,T;H^{1}(\Omega)) \cap L^{2}(0,T;H^{3}(\Omega) \cap H^{2}_{N}(\Omega)) \cap H^{1}(0,T;(H^{1}(\Omega))'), \\
\rho_{u} & \in L^{2}(0,T;H^{1}(\Omega)), \\
\xi_{u} & \in L^{\infty}(0,T;H^{1}(\Omega)) \cap L^{2}(0,T;H^{2}_{N}(\Omega)) \cap H^{1}(0,T;L^{2}(\Omega)),
\end{align*}
with $\theta_{u}(0) = 0$, $\xi_{u}(0) = 0$, $\pdnu \theta_{u} = \pdnu \rho_{u} = \pdnu \xi_{u} = 0$ on $\pd \Omega$,
\begin{subequations}
\begin{alignat}{3}
\rho_{u} & = \beta \eps^{-1} (\Psi''(\varphi) \theta_{u} + (\varphi_{u} - \varphi)^{2} R_{\Psi}) - \beta \eps \Laplace \theta_{u} && \text{ in } Q,\label{Fdiff:1}  \\
(\xi_{u})_{t} & = \Laplace \xi_{u} - X_{\sigma} && \text{ in } Q,  \label{Fdiff:2}
\end{alignat}
\end{subequations}
and
\begin{align}\label{Fdiff:3}
\inner{(\theta_{u})_{t}}{\zeta}_{H^{1}} & = \int_{\Omega} -\bm{X}_{m} \cdot \nabla \zeta - m(\varphi) \nabla \rho_{u} \cdot \nabla \zeta + X_{\varphi} \zeta \dx
\end{align}
for a.e. $t \in (0,T)$ and for all $\zeta \in H^{1}(\Omega)$.

We now derive a priori estimates for $(\theta_{u}, \rho_{u}, \xi_{u})$.  Below the symbol $C$ denotes positive constants that are not dependent on $(\theta_{u}, \rho_{u}, \xi_{u}, u_{\PP}, u_{\chi}, u_{\CC})$ and may vary from line to line, and the symbol $\mathcal{F}$ is short for $\left ( \abs{u_{\PP}} + \abs{u_{\chi}} + \abs{u_{\CC}} \right )$.

\paragraph{First estimate.}
Testing \eqref{Fdiff:2} with $\xi_{u}$ and integrating in time gives
\begin{align*}
\frac{1}{2} \norm{\xi_{u}(s)}_{L^{2}}^{2} + \int_{0}^{s} \norm{\nabla \xi_{u}}_{L^{2}}^{2} \dt = \int_{0}^{s} \int_{\Omega} - X_{\sigma} \xi_{u} \dx \dt,
\end{align*}
for any $s \in (0,T)$.  By the Lipschitz continuity of $h$, and the continuous dependence result \eqref{Ctsdep:result} from Theorem \ref{thm:Ctsdep}, we find that
\begin{equation}\label{Xsigma}
\begin{aligned}
\norm{X_{\sigma}}_{L^{2}(0,s;L^{2})} & \leq C \left ( \norm{\varphi_{u} - \varphi}_{L^{2}(L^{\infty})} \norm{\sigma_{u} - \sigma}_{L^{\infty}(L^{2})} + \norm{\sigma}_{L^{\infty}(L^{\infty})} \norm{\theta_{u}}_{L^{2}(0,s;L^{2})} \right )  \\
& \quad +  C \left (  \norm{\sigma}_{L^{\infty}(L^{\infty})} \norm{\varphi_{u} - \varphi}_{L^{2}(L^{\infty})} \norm{\varphi_{u} - \varphi}_{L^{\infty}(L^{2})} + \norm{\xi_{u}}_{L^{2}(0,s;L^{2})}\right ) \\
& \quad + C \abs{u_{\CC}} \left ( \norm{\sigma}_{L^{\infty}(L^{2})} \norm{\varphi_{u} - \varphi}_{L^{2}(L^{\infty})} + \norm{\sigma_{u} - \sigma}_{L^{2}(Q)} \right ) \\
& \quad + C \abs{u_{\CC}} \norm{\varphi_{u} - \varphi}_{L^{2}(L^{\infty})} \norm{\sigma_{u} - \sigma}_{L^{\infty}(L^{2})} \\
& \leq C \left ( \norm{\xi_{u}}_{L^{2}(0,s;L^{2})} + \norm{\theta_{u}}_{L^{2}(0,s;L^{2})} \right ) + C \left ( \mathcal{F}^{2} + \mathcal{F}^{3} \right ).
\end{aligned}
\end{equation}
Hence, we obtain for any $s \in (0,T)$,
\begin{align}\label{Fdiff:xi}
\norm{\xi_{u}(s)}_{L^{2}}^{2} + \norm{\nabla \xi_{u}}_{L^{2}(0,s;L^{2})}^{2} \leq  C \left ( \norm{\xi_{u}}_{L^{2}(0,s;L^{2})}^{2} + \norm{\theta_{u}}_{L^{2}(0,s;L^{2})}^{2} \right ) + C \left (\mathcal{F}^{4} + \mathcal{F}^{6} \right ).
\end{align}
Next, let us compute
\begin{equation}\label{bmX:m:est}
\begin{aligned}
\norm{\bm{X}_{m}}_{L^{2}(0,s;L^{2})}^{2} & \leq \int_{0}^{s} \int_{\Omega} \abs{ m(\varphi) \chi \nabla \xi_{u}}^{2} + \abs{u_{\chi} m(\varphi) \nabla (\sigma_{u} - \sigma)}^{2} \dx \dt \\
& \quad + \int_{0}^{s} \int_{\Omega}  \abs{m(\varphi_{u}) - m(\varphi)}^{2} \abs{\chi \nabla (\sigma_{u} - \sigma) + u_{\chi} \nabla \sigma}^{2} \dx \dt \\
& \quad + \int_{0}^{s} \int_{\Omega} \abs{m(\varphi_{u}) - m(\varphi)}^{2} \abs{\nabla (\mu_{u} - \mu) - u_{\chi}(\sigma_{u} - \sigma))}^{2} \dx \dt \\
& \quad + \int_{0}^{s} \int_{\Omega} \abs{m''(\varphi) \theta_{u} + (\varphi_{u} - \varphi)^{2} R_{m}}^{2} \abs{\nabla (\mu - \chi \sigma)}^{2} \dx \dt \\
& =: L_{1} + L_{2} + L_{3} + L_{4},
\end{aligned}
\end{equation}
and upon using the boundedness of $\sigma$, $\nabla \sigma$ a.e. in $Q$ and $\nabla \mu \in L^{2}(0,T;H^{2}) \cap L^{\infty}(0,T;L^{2})$ from Theorem \ref{thm:Exist}, we find that
\begin{align*}
L_{1} & \leq \chi^{2} n_{1}^{2} \norm{\nabla \xi_{u}}_{L^{2}(0,s;L^{2})}^{2} + C \abs{u_{\chi}}^{2} \norm{\nabla (\sigma_{u} - \sigma)}_{L^{2}(Q)}^{2}, \\
L_{2} & \leq C \norm{\varphi_{u} - \varphi}_{L^{2}(L^{\infty})}^{2} \left ( \norm{\nabla (\sigma_{u} - \sigma)}_{L^{\infty}(L^{2})}^{2} + \abs{u_{\chi}}^{2} \right ), \\
L_{3} & \leq C \norm{\varphi_{u} - \varphi}_{L^{\infty}(L^{3})}^{2} \left ( \norm{\nabla (\mu_{u} - \mu)}_{L^{2}(L^{6})}^{2} + \abs{u_{\chi}}^{2} \norm{\nabla (\sigma_{u} - \sigma)}_{L^{2}(L^{6})}^{2} \right ), \\
L_{4} & \leq C\int_{0}^{s} \norm{\theta_{u}}_{L^{2}}^{2} \norm{\nabla (\mu - \chi \sigma)}_{L^{\infty}}^{2} \dt \\
& \quad + C\norm{\varphi_{u} - \varphi}_{L^{2}(L^{\infty})}^{2} \norm{\varphi_{u} - \varphi}_{L^{\infty}(L^{\infty})}^{2} \norm{\nabla (\mu - \chi \sigma)}_{L^{\infty}(L^{2})}^{2}.
\end{align*}
Let us remark here that we require continuous dependence for $\varphi$ in $L^{\infty}(0,T;L^{\infty})$ and for $\mu$ in $L^{2}(0,T;W^{1,6})$, and the analogous continuous dependence result stated in \cite[Thm. 3]{LamWu} is not sufficient to control $L_{3}$ and $L_{4}$.

Thanks to the continuous dependence results in Theorem \ref{thm:Ctsdep} we easily infer
\begin{align}\label{bmX:m}
\norm{\bm{X}_{m}}_{L^{2}(0,s;L^{2})} \leq C \left ( \mathcal{F}^{2} + \mathcal{F}^{3} + \left ( \int_{0}^{s} \norm{\theta_{u}}_{L^{2}}^{2} \norm{\mu - \chi \sigma}_{H^{3}}^{2} \dt \right)^{\frac{1}{2}} \right ) + \chi n_{1} \norm{\nabla \xi_{u}}_{L^{2}(0,s;L^{2})},
\end{align}
and so 
\begin{equation}\label{Fdiff:Xm}
\begin{aligned}
\abs{\int_{0}^{s} \int_{\Omega} \bm{X}_{m} \cdot \nabla \theta_{u} \dx \dt} & \leq C \left ( \mathcal{F}^{4} + \mathcal{F}^{6} + \int_{0}^{s} \norm{\mu - \chi \sigma}_{H^{3}}^{2} \norm{\theta_{u}}_{L^{2}}^{2} \dt\right ) \\
& \quad + \norm{\nabla \theta_{u}}_{L^{2}(0,s;L^{2})}^{2} +  \frac{\chi^{2} n_{1}^{2}}{2}\norm{\nabla \xi_{u}}_{L^{2}(0,s;L^{2})}^{2}.
\end{aligned}
\end{equation}
Similarly, it holds that
\begin{align}
\label{Xvarphi}
\norm{X_{\varphi}}_{L^{2}(0,s;L^{2})} & \leq C \left (  \mathcal{F}^{2} + \mathcal{F}^{3} + \norm{\theta_{u}}_{L^{2}(0,s;L^{2})} + \norm{\xi_{u}}_{L^{2}(0,s;L^{2})} \right ).
\end{align}
Then, testing \eqref{Fdiff:3} with  $\zeta = \beta \eps \theta_{u}$ and \eqref{Fdiff:1} with $D \theta_{u}$ and $m(\varphi)\rho_{u}$ for some positive constant $D$ yet to be determined, upon summing and integrating in time gives
\begin{align*}
& \frac{\beta \eps}{2} \norm{\theta_{u}(s)}_{L^{2}}^{2}  + D \beta \eps \norm{\nabla \theta_{u}}_{L^{2}(0,s;L^{2})}^{2} + \norm{m^{\frac{1}{2}}(\varphi)\rho_{u}}_{L^{2}(0,s;L^{2})}^{2} \\
& \quad = \int_{0}^{s} \int_{\Omega} - \beta \eps \bm{X}_{m} \cdot \nabla \theta_{u} + \beta \eps X_{\varphi} \theta_{u} + D \rho_{u} \theta_{u} \dx \dt \\
& \qquad -  \int_{0}^{s} \int_{\Omega} D \beta \eps^{-1} \left ( \Psi''(\varphi) \abs{\theta_{u}}^{2} + (\varphi_{u} - \varphi)^{2} R_{\Psi} \theta_{u} \right ) \dx \dt \\
& \qquad + \int_{0}^{s} \int_{\Omega} m(\varphi) \beta \eps^{-1} \left ( \Psi''(\varphi) \theta_{u} \rho_{u} + (\varphi_{u} - \varphi)^{2} R_{\Psi} \rho_{u} \right )  \dx \dt \\
& \qquad + \int_{0}^{s} \int_{\Omega} \beta \eps m'(\varphi) \rho_{u} \nabla \varphi \cdot \nabla \theta_{u}  \dx \dt \\
& \quad \leq C \left ( \mathcal{F}^{4} + \mathcal{F}^{6} \right ) + \left ( 1 + \frac{\beta^{2} \eps^{2} \norm{m'(\varphi)\nabla \varphi}_{L^{\infty}(L^{\infty})}^{2}}{n_{0}} \right ) \norm{\nabla \theta_{u}}_{L^{2}(0,s;L^{2})}^{2} + \frac{n_{0}}{2} \norm{\rho_{u}}_{L^{2}(0,s;L^{2})}^{2} \\
& \qquad + \frac{\chi^{2} n_{1}^{2}}{2} \norm{\nabla \xi_{u}}_{L^{2}(0,s;L^{2})}^{2}  + C \norm{\xi_{u}}_{L^{2}(0,s;L^{2})}^{2} \\
& \qquad + C(1 + D^{2}) \int_{0}^{s} \left ( 1 + \norm{\mu - \chi \sigma}_{H^{3}}^{2} \right ) \norm{\theta_{u}}_{L^{2}}^{2} \dt.
\end{align*}
Adding \eqref{Fdiff:xi} multiplied by a constant $E > \frac{\chi^{2} n_{1}^{2}}{2}$ to the above, and choosing $D$ sufficiently large yields the integral inequality
\begin{align*}
& \norm{\xi_{u}(s)}_{L^{2}}^{2} + \norm{\theta_{u}(s)}_{L^{2}}^{2} + \norm{\nabla \xi_{u}}_{L^{2}(0,s;L^{2})}^{2} + \norm{\nabla \theta_{u}}_{L^{2}(0,s;L^{2})}^{2} + \norm{\rho_{u}}_{L^{2}(0,s;L^{2})}^{2} \\
& \quad \leq C \int_{0}^{s} \left ( 1 + \norm{\mu - \chi \sigma}_{H^{3}}^{2} \right ) \left ( \norm{\xi_{u}}_{L^{2}}^{2} + \norm{\theta_{u}}_{L^{2}}^{2} \right ) \dt + C \left ( \mathcal{F}^{4} + \mathcal{F}^{6} \right )
\end{align*}
for any $s \in (0,T)$, and by a Gronwall inequality in integral form we obtain
\begin{align}\label{Fdiff:est:1}
\norm{\xi_{u}}_{L^{\infty}(0,T;L^{2}) \cap L^{2}(0,T;H^{1})} + \norm{\theta_{u}}_{L^{\infty}(0,T;L^{2})\cap L^{2}(0,T;H^{1})} + \norm{\rho_{u}}_{L^{2}(Q)} \leq C  \left ( \mathcal{F}^{2} + \mathcal{F}^{3} \right ).
\end{align}

\paragraph{Second estimate.}
Testing \eqref{Fdiff:2} with $( \xi_{u})_{t}$, and using \eqref{Xsigma} and \eqref{Fdiff:est:1} leads to
\begin{align}\label{Fdiff:est:2}
\norm{(\xi_{u})_{t}}_{L^{2}(Q)} + \norm{\nabla \xi_{u}}_{L^{\infty}(0,T;L^{2})} \leq C \left ( \mathcal{F}^{2} + \mathcal{F}^{3} \right ).
\end{align}
It is worth pointing out that here we used $\nabla (\xi_{u})_{t} \in L^{2}(0,T;(H^{1})')$ so that 
\begin{align*}
\inner{\nabla \xi_{u}}{\nabla (\xi_{u})_{t}}_{H^{1}} = \frac{1}{2} \frac{\dd}{\dt} \norm{\nabla \xi_{u}}_{L^{2}}^{2}.
\end{align*}

\paragraph{Third estimate.}
By virtue of elliptic regularity, from \eqref{Fdiff:1} and \eqref{Fdiff:2} it is easy to see that
\begin{align*}
\norm{\xi_{u}}_{L^{2}(H^{2})} & \leq C \left ( \norm{(\xi_{u})_{t}}_{L^{2}(Q)} + \norm{X_{\sigma}}_{L^{2}(Q)} \right ) \leq C  \left ( \mathcal{F}^{2} + \mathcal{F}^{3} \right ), \\
\norm{\theta_{u}}_{L^{2}(H^{2})} & \leq C \left ( \norm{\theta_{u}}_{L^{2}(H^{1})} + \norm{\rho_{u}}_{L^{2}(Q)} + \norm{\Psi''(\varphi) \theta_{u} + (\varphi_{u} - \varphi)^{2} R_{\Psi}}_{L^{2}(Q)} \right ) \leq C  \left ( \mathcal{F}^{2} + \mathcal{F}^{3} \right ) .
\end{align*}
Next, consider an arbitrary test function $\zeta \in L^{2}(0,T;H^{2}_{N})$ in \eqref{Fdiff:3}, integrating by parts and recalling \eqref{bmX:m}, \eqref{Xvarphi}, \eqref{Fdiff:est:1} yields
\begin{align*}
\abs{\int_{0}^{T} \inner{(\theta_{u})_{t}}{\zeta} \dt} & = \int_{0}^{T} \int_{\Omega} - \bm{X}_{m} \cdot \nabla \zeta + m(\varphi) \rho_{u} \Laplace \zeta + \rho_{u} m'(\varphi) \nabla \varphi \cdot \nabla \zeta + X_{\varphi} \zeta \dx \dt \\
& \leq \left ( \norm{\bm{X}_{m}}_{L^{2}(Q)} + C\norm{\rho_{u}}_{L^{2}(Q)} + \norm{X_{\varphi}}_{L^{2}(Q)} \right ) \norm{\zeta}_{L^{2}(0,T;H^{2})} \\
& \leq C  \left ( \mathcal{F}^{2} + \mathcal{F}^{3} \right ) \norm{\zeta}_{L^{2}(0,T;H^{2})},
\end{align*}
where by \eqref{Fdiff:est:1} it holds that
\begin{align*}
\int_{0}^{s} \norm{\theta_{u}}_{L^{2}}^{2} \norm{\mu - \chi \sigma}_{H^{3}}^{2} \dt \leq \norm{\mu - \chi \sigma}_{L^{2}(H^{3})}^{2} \norm{\theta_{u}}_{L^{\infty}(L^{2})}^{2} \leq C  \left ( \mathcal{F}^{4} + \mathcal{F}^{6} \right ).
\end{align*}
Altogether we have
\begin{align}\label{Fdiff:est:3}
\norm{\xi_{u}}_{L^{2}(0,T;H^{2})} + \norm{\theta_{u}}_{L^{2}(0,T;H^{2}) \cap H^{1}(0,T;(H^{2}_{N})')} \leq C  \left ( \mathcal{F}^{2} + \mathcal{F}^{3} \right ).
\end{align}
\end{proof}

\subsection{Solvability of the adjoint system}
In order to simplify the necessary optimality conditions for the minimizer $(\PP_{*}, \chi_{*}, \CC_{*})$ obtained in Theorem \ref{thm:mini}, it is convenient to first study the following adjoint system:
\begin{subequations}\label{Adjoint}
\begin{alignat}{3}
\notag -p_{t} + \beta \eps \Laplace q & = \beta \eps^{-1} \Psi''(\varphi) q - m'(\varphi) \nabla (\mu - \chi \sigma) \cdot \nabla p && \\
& \quad + \PP f'(\varphi) g(\sigma) p - \CC h'(\varphi) \sigma r + \beta_{Q}(\varphi - \varphi_{Q}) && \text{ in } Q, \label{adjoint:p} \\
q & = \div (m(\varphi) \nabla p) && \text{ in } Q, \label{adjoint:q} \\
-r_{t} & = \Laplace r - \div (\chi m(\varphi) \nabla p) - \CC h(\varphi) r + \PP f(\varphi) g'(\sigma) p && \text{ in } Q, \label{adjoint:r} \\
0 & = \pdnu p = \pdnu q = \pdnu r && \text{ on } \Gamma, \\
p(T) & = \beta_{\Omega}(\varphi(T) - \varphi_{\Omega}), \quad r(T) = 0 && \text{ in } \Omega.
\end{alignat}
\end{subequations}
The unique solvability of \eqref{Adjoint} is given in the following theorem.

\begin{thm}
Under Assumption \ref{assump:Wellposed}, $\varphi_{Q} \in L^{2}(Q)$, $\varphi_{\Omega} \in L^{2}(\Omega)$, for any $(\PP, \chi, \CC) \in \mathcal{U}_{\mathrm{ad}}$, there exists a unique triplet $(p,q,r)$ associated to $\mathcal{S}(\PP, \chi, \CC) = (\varphi, \mu, \sigma)$ with
\begin{align*}
p & \in L^{2}(0,T;H^{2}_{N}(\Omega)) \cap H^{1}(0,T;(H^{2}_{N}(\Omega))') \cap C^{0}([0,T];L^{2}(\Omega)), \\
q & \in L^{2}(0,T;L^{2}(\Omega)), \\
r & \in L^{2}(0,T;H^{2}_{N}(\Omega)) \cap L^{\infty}(0,T;H^{1}(\Omega)) \cap H^{1}(0,T;L^{2}(\Omega)),
\end{align*}
satisfying $p(T) = \beta_{\Omega}(\varphi(T) - \varphi_{\Omega})$, $r(T) = 0$ in $L^{2}(\Omega)$, \eqref{adjoint:q}, \eqref{adjoint:r}, and
\begin{equation}\label{adjoint:p:weak}
\begin{aligned}
0 & = \inner{-p_{t}}{\zeta}_{H^{2}} + \int_{\Omega} \beta \eps q \Laplace \zeta  - \beta \eps^{-1} \Psi''(\varphi) q \zeta + m'(\varphi) \zeta \nabla (\mu - \chi \sigma) \cdot \nabla p  \dx \\
& \quad + \int_{\Omega} \CC h'(\varphi) \sigma r \zeta - \PP f'(\varphi) g(\sigma) p \zeta  - \beta_{Q}(\varphi - \varphi_{Q}) \zeta \dx
\end{aligned}
\end{equation}
for a.e. $t \in (0,T)$ and for all $\zeta \in H^{2}_{N}(\Omega)$.
\end{thm}

\begin{proof}
Once again the proof employs a Galerkin approximation.  In the following the symbol $C$ denotes positive constants not depending on $p$, $q$ and $r$, and may vary from line to line.

\paragraph{First estimate.} Testing \eqref{adjoint:r} with $r$, integrating in time from $s \in [0,T)$ to $T$ yields
\begin{align}\label{adjoint:est:r}
\frac{1}{2} \norm{r(s)}_{L^{2}}^{2} + \frac{1}{2} \norm{\nabla r}_{L^{2}(s,T;L^{2})}^{2} \leq C \norm{p}_{L^{2}(s,T;L^{2})} \norm{r}_{L^{2}(s,T;L^{2})} + \frac{\chi^{2} n_{1}^{2}}{2} \norm{\nabla p}_{L^{2}(s,T;L^{2})}^{2},
\end{align}
where we have neglected the non-negative term $\CC h(\varphi) \abs{r}^{2}$.  Meanwhile, testing \eqref{adjoint:p} with $m(\varphi) p$, testing \eqref{adjoint:q} with $ \beta \eps q$ and $D p$ for some positive constant $D$ yet to be determined, and upon summing leads to
\begin{equation}\label{adjoint:est:pq}
\begin{aligned}
& -\frac{1}{2} \frac{\dd}{\dt} \left ( \int_{\Omega} m(\varphi)  \abs{p}^{2} \dx \right ) + \beta \eps \norm{q}_{L^{2}}^{2} + D \norm{m^{\frac{1}{2}}(\varphi) \nabla p}_{L^{2}}^{2} \\
& \quad = \int_{\Omega} - m'(\varphi) \varphi_{t} \frac{\abs{p}^{2}}{2} + \beta \eps m'(\varphi) p \nabla q \cdot \nabla \varphi - m(\varphi) m'(\varphi) p \nabla (\mu - \chi \sigma) \cdot \nabla p \dx \\
& \qquad + \int_{\Omega} \beta \eps^{-1} \Psi''(\varphi) m(\varphi) p q - D p q  + \beta_{Q}(\varphi - \varphi_{Q}) m(\varphi) p\dx \\
& \qquad + \int_{\Omega} \PP f'(\varphi) g(\sigma) m(\varphi) \abs{p}^{2} - \CC h'(\varphi) m(\varphi) \sigma p r \dx \\
& \quad =: M_{1} + M_{2} + M_{3} .
\end{aligned}
\end{equation}
The boundedness of $f$, $g$, $m$, $h$ and their derivatives, as well as the boundedness of $\sigma$ and $\Psi''(\varphi)$ a.e. in $Q$ allow us to infer
\begin{align}\label{M2:M3}
M_{2} + M_{3} & \leq C(1 + D^{2})\left ( \norm{p}_{L^{2}}^{2} + \norm{r}_{L^{2}}^{2} + \norm{\varphi - \varphi_{Q}}_{L^{2}}^{2} \right ) + \frac{\beta \eps}{4} \norm{q}_{L^{2}}^{2}. 
\end{align}
Next, thanks to Theorem \ref{thm:Exist}, we have $\varphi_{t} \in L^{2}(Q)$, $\nabla \varphi \in L^{\infty}(0,T;L^{\infty})$, $\Laplace \varphi \in  L^{2}(0,T;H^{2})$ and $\nabla (\mu - \chi \sigma) \in L^{2}(0,T;H^{2})$, and so after integrating by parts and applying the Gagliardo--Nirenberg inequality \eqref{GN:alt} we see that
\begin{align*}
M_{1} & = \int_{\Omega} - m'(\varphi) \varphi_{t} \frac{\abs{p}^{2}}{2} - \beta \eps \left ( m''(\varphi) \abs{\nabla \varphi}^{2} + m'(\varphi) \Laplace \varphi  \right ) p q \dx \\
& \quad - \int_{\Omega} \beta \eps m'(\varphi) q \nabla p \cdot \nabla \varphi + m(\varphi) m'(\varphi) p \nabla (\mu - \chi \sigma) \cdot \nabla p \dx \\
& \leq C \norm{\varphi_{t}}_{L^{2}} \left ( \norm{p}_{L^{2}} \norm{\nabla p}_{L^{2}} + \norm{p}_{L^{2}}^{2} \right ) + C \left ( 1 + \norm{\Laplace \varphi}_{L^{\infty}} \right ) \norm{p}_{L^{2}} \norm{q}_{L^{2}} \\
& \quad + \beta \eps \norm{m'(\varphi) \nabla \varphi}_{L^{\infty}(L^{\infty})} \norm{q}_{L^{2}} \norm{\nabla p}_{L^{2}} + C \norm{\nabla (\mu - \chi \sigma)}_{L^{\infty}} \norm{p}_{L^{2}} \norm{\nabla p}_{L^{2}} \\
&  \leq \frac{\beta \eps}{4} \norm{q}_{L^{2}}^{2} + C \left ( 1 + \norm{m'(\varphi) \nabla \varphi}_{L^{\infty}(L^{\infty})}^{2} \right ) \norm{\nabla p}_{L^{2}}^{2} \\
& \quad + C \left ( 1 + \norm{\varphi_{t}}_{L^{2}}^{2} + \norm{\Laplace \varphi}_{L^{\infty}}^{2} + \norm{\nabla (\mu - \chi \sigma)}_{L^{\infty}}^{2} \right ) \norm{p}_{L^{2}}^{2}.
\end{align*}
Hence, integrating \eqref{adjoint:est:pq} in time from $s \in [0,T)$ to $T$ and using the above estimates we have
\begin{align*}
& \frac{n_{0}}{2} \norm{p(s)}_{L^{2}}^{2} - C \int_{s}^{T} \left ( 1 + D^{2} + \norm{\varphi_{t}}_{L^{2}}^{2} + \norm{\Laplace \varphi}_{L^{\infty}}^{2} + \norm{\nabla (\mu - \chi \sigma)}_{L^{\infty}}^{2} \right ) \norm{p}_{L^{2}}^{2} \dt \\
& \qquad + \frac{\beta \eps}{2} \norm{q}_{L^{2}(s,T;L^{2})}^{2} + \left ( D n_{0} - C - C \norm{m'(\varphi) \nabla \varphi}_{L^{\infty}(L^{\infty})}^{2} \right )  \norm{\nabla p}_{L^{2}(s,T;L^{2})}^{2} \\
& \quad \leq C(1 + D^{2}) \left ( \norm{r}_{L^{2}(s,T;L^{2})}^{2} + \norm{\varphi - \varphi_{Q}}_{L^{2}(Q)}^{2} + \norm{\varphi(T) - \varphi_{\Omega}}_{L^{2}}^{2} \right )
\end{align*}
Adding the above inequality to \eqref{adjoint:est:r}, and choose $D$ sufficiently large, 
so that the pre factor of $\norm{\nabla p}_{L^{2}(s,T;L^{2})}^{2}$ 
is positive then yields
\begin{equation}\label{adjoint:main}
\begin{aligned}
& \norm{p(s)}_{L^{2}}^{2} + \norm{r(s)}_{L^{2}}^{2} + \norm{q}_{L^{2}(s,T;L^{2})}^{2} + \norm{\nabla p}_{L^{2}(s,T;L^{2})}^{2} + \norm{\nabla r}_{L^{2}(s,T;L^{2})}^{2} \\
& \quad \leq C \int_{s}^{T} \left ( 1 + \norm{\varphi_{t}}_{L^{2}}^{2} + \norm{\Laplace \varphi}_{L^{\infty}}^{2} + \norm{\nabla (\mu - \chi \sigma)}_{L^{\infty}}^{2} \right ) \left ( \norm{p}_{L^{2}}^{2} + \norm{r}_{L^{2}}^{2} \right ) \dt \\
& \qquad + C \norm{\varphi - \varphi_{Q}}_{L^{2}(Q)}^{2} + C \norm{\varphi(T) - \varphi_{\Omega}}_{L^{2}}^{2},
\end{aligned}
\end{equation}
for any $s \in [0,T)$, and a Gronwall argument leads to
\begin{align}\label{adjoint:1}
\norm{p}_{L^{\infty}(0,T;L^{2}) \cap L^{2}(0,T;H^{1})} + \norm{r}_{L^{\infty}(0,T;L^{2}) \cap L^{2}(0,T;H^{1})} + \norm{q}_{L^{2}(Q)} \leq C.
\end{align}

\paragraph{Second estimate.} As $q \in L^{2}(Q)$ and $m(\varphi)$ is uniformly Lipschitz, owning to elliptic regularity \cite[Thm. 2.4.2.7]{Grisvard} we have 
\begin{align*}
\norm{p}_{L^{2}(0,T;H^{2})} \leq C.
\end{align*}
Then, it is easy to see that $\div (m(\varphi) \nabla p) \in L^{2}(Q)$.  By testing \eqref{adjoint:r} with $r_{t}$ we have that $\nabla r \in L^{\infty}(0,T;L^{2})$ and $r_{t} \in L^{2}(Q)$.  Elliptic regularity then provides $r \in L^{2}(0,T;H^{2})$, and altogether it holds that
\begin{align}\label{adjoint:2}
\norm{r}_{L^{2}(0,T;H^{2}) \cap L^{\infty}(0,T;H^{1}) \cap H^{1}(0,T;L^{2})} + \norm{p}_{L^{2}(0,T;H^{2})} \leq C.
\end{align}

\paragraph{Third estimate.}
Testing \eqref{adjoint:p} with an arbitrary test function $\zeta \in L^{2}(0,T;H^{2}_{N})$ and integrating by parts yields
\begin{align*}
\abs{\int_{0}^{T} \inner{p_{t}}{\zeta} \dt} & \leq C \norm{q}_{L^{2}(Q)} \left (\norm{\Laplace \zeta}_{L^{2}(Q)} + \norm{\zeta}_{L^{2}(Q)} \right ) \\
& \quad + C\norm{\nabla (\mu - \chi \sigma)}_{L^{\infty}(L^{2})} \norm{\nabla p}_{L^{2}(Q)} \norm{\zeta}_{L^{2}(L^{\infty})} \\
& \quad + C \left ( \norm{p}_{L^{2}(Q)} + \norm{r}_{L^{2}(Q)} + \norm{\varphi - \varphi_{Q}}_{L^{2}(Q)} \right ) \norm{\zeta}_{L^{2}(Q)} \\
& \leq C \norm{\zeta}_{L^{2}(0,T;H^{2})},
\end{align*}
and so
\begin{align}\label{adjoint:3}
\norm{p_{t}}_{L^{2}(0,T;(H^{2}_{N})')} \leq C.
\end{align}

\paragraph{Uniqueness.} Let $p := p_{1} - p_{2}$, $q := q_{1} - q_{2}$ and $r := r_{1} - r_{2}$ denote the difference between two solutions to \eqref{Adjoint} with the same initial data, then it holds that
\begin{align*}
p & \in L^{2}(0,T;H^{2}_{N}(\Omega)) \cap H^{1}(0,T;(H^{2}_{N}(\Omega))') \cap L^{\infty}(0,T;L^{2}(\Omega)), \\
q & \in L^{2}(0,T;L^{2}(\Omega)), \\
r & \in L^{2}(0,T;H^{2}_{N}(\Omega)) \cap L^{\infty}(0,T;H^{1}(\Omega)) \cap H^{1}(0,T;L^{2}(\Omega)),
\end{align*}
satisfying $r(T) = p(T) = 0$, \eqref{adjoint:q}, \eqref{adjoint:r} and
\begin{equation}\label{adjoint:uniq:p}
\begin{aligned}
0 & = \inner{-p_{t}}{\zeta}_{H^{2}} + \int_{\Omega} \beta \eps q \Laplace \zeta  - \beta \eps^{-1} \Psi''(\varphi) q \zeta + m'(\varphi) \zeta \nabla \mu \cdot \nabla p  \dx \\
& \quad + \int_{\Omega} \CC h'(\varphi) \sigma r \zeta - \PP f'(\varphi) g(\sigma) p \zeta  \dx
\end{aligned}
\end{equation}
for a.e. $t \in (0,T)$ and for all $\zeta \in H^{2}(\Omega)$.  Using the boundedness of $\varphi$ and $\nabla \varphi$ in $Q$, the boundedness of $p$ in $L^{\infty}(0,T;L^{2}) \cap L^{2}(0,T;H^{2}_{N})$, and the boundedness of the second derivatives $\pd_{i} \pd_{j} \varphi$ in $L^{2}(0,T;H^{2}) \cap L^{\infty}(0,T;H^{1})$, it is easy to see that $m(\varphi) p \in L^{2}(0,T;H^{2}_{N})$, and so we can substitute $\zeta = m(\varphi) p$ in \eqref{adjoint:uniq:p}.  This gives
\begin{align*}
& -\frac{1}{2} \frac{\dd}{\dt} \int_{\Omega} m(\varphi) \abs{p}^{2} \dx + \int_{\Omega} m'(\varphi) \varphi_{t} \frac{1}{2} \abs{p}^{2} + \left (  \CC h'(\varphi) \sigma r - \PP f'(\varphi) g(\sigma) p \right ) m(\varphi) p \dx \\
& \quad + \int_{\Omega} \beta \eps q \left ( \div (m(\varphi) \nabla p) + m''(\varphi) p \abs{\nabla \varphi}^{2} + m'(\varphi) \nabla p \cdot \nabla \varphi + m'(\varphi) p \Laplace \varphi \right ) \dx \\
& \quad + \int_{\Omega} m'(\varphi) m(\varphi) p \nabla (\mu - \chi \sigma) \cdot \nabla p - \beta \eps^{-1} \Psi''(\varphi) m(\varphi) p  q \dx = 0.
\end{align*}
Using \eqref{adjoint:q} in the above equality allows us to simplify the product $q \div (m(\varphi) \nabla p)$ to $\abs{q}^{2}$.  Then, testing \eqref{adjoint:q} with $D p$ and \eqref{adjoint:r} with $r$ in the same spirit as the derivation of \eqref{adjoint:1} leads to the integral inequality \eqref{adjoint:main} without the second and third terms on the right-hand side.  A Gronwall argument (\cite[Lem. 3.1]{GLNeu} with $\alpha = 0$) then gives
\begin{align*}
\norm{p}_{L^{\infty}(0,T;L^{2}) \cap L^{2}(0,T;H^{1})} + \norm{r}_{L^{\infty}(0,T;L^{2}) \cap L^{2}(0,T;H^{1})} + \norm{q}_{L^{2}(Q)} \leq 0,
\end{align*}
which yields the uniqueness of solutions.
\end{proof}

\subsection{Necessary optimality conditions}
Let $(a,b,c) \in \mathcal{U}_{\mathrm{ad}}$ be arbitrary and set $\bm{u} = (u_{\PP}, u_{\chi}, u_{\CC}) \in \R^{3}$ as $u_{\PP} = a - \PP_{*}$, $u_{\chi} = b - \chi_{*}$, $u_{\CC} = c - \CC_{*}$, where $(\PP_{*}, \chi_{*}, \CC_{*})$ denotes a minimizer obtained from Theorem \ref{thm:mini}.  Recalling the linearized state variables $(\Phi_{u}, \Xi_{u}, \Sigma_{u})$ associated to $\bm{u}$ and introducing the reduced functional as
\begin{align*}
\mathcal{J}(\PP, \chi, \CC) := J( \mathcal{S}_{1}(\PP, \chi, \CC), \PP, \chi, \CC),
\end{align*}
then by Theorem \ref{thm:Fdiff} it is Fr\'{e}chet differentiable with respect to $\PP$, $\chi$ and $\CC$ with the derivative in the direction $\bm{u}$ given as
\begin{align*}
\mathrm{D} \mathcal{J}(\PP_{*}, \chi_{*}, \CC_{*}) [\bm{u}] & = \int_{0}^{T} \int_{\Omega} \beta_{Q} ( \varphi_{*} - \varphi_{Q}) \Phi_{u} \dx \dt + \int_{\Omega}  \beta_{\Omega} (\varphi_{*}(T) - \varphi_{\Omega}) \Phi_{u}(T) \dx \\
& \quad + \beta_{\PP} \PP_{*} u_{\PP} + \beta_{\chi} \chi_{*} u_{\chi} + \beta_{\CC} \CC_{*} u_{\CC}.
\end{align*}
The above expression is non-negative as $(\PP_{*}, \chi_{*}, \CC_{*})$ is a minimizer.  Let us point out that thanks to the (compact) embedding $L^{\infty}(0,T;H^{1}) \cap H^{1}(0,T;(H^{1})') \subset C^{0}([0,T];L^{2})$, the value of $\Phi_{u}(T)$ is well-defined.  The goal is to use the linearized state equations \eqref{Lin:State} and the adjoint equations \eqref{Adjoint} to simplify the above expression.

\begin{thm}
Under Assumption \ref{assump:Wellposed}, $\varphi_{Q} \in L^{2}(Q)$, $\varphi_{\Omega} \in L^{2}(\Omega)$, let $(\PP_{*}, \chi_{*}, \CC_{*}) \in \mathcal{U}_{\mathrm{ad}}$ denote a minimizer to \eqref{Minimizer} with corresponding state variables $\mathcal{S}(\PP_{*}, \chi_{*}, \CC_{*}) = (\varphi_{*}, \mu_{*}, \sigma_{*})$ and adjoint variables $(p,q,r)$.  Then, $(\PP_{*}, \chi_{*}, \CC_{*})$ necessarily satisfies
\begin{align*}
& \int_{0}^{T} \int_{\Omega} (a - \PP_{*}) f(\varphi_{*}) g(\sigma_{*})p + (b - \chi_{*}) m(\varphi_{*}) \nabla \sigma_{*} \cdot \nabla p - (c - \CC_{*}) h(\varphi_{*}) \sigma_{*} r \dx \dt \\
& \quad + \beta_{\PP} \, \PP_{*} (a - \PP_{*}) + \beta_{\chi} \, \chi_{*} (b - \chi_{*}) + \beta_{\CC} \, \CC_{*} (c - \CC_{*}) \geq 0
\end{align*}
for all $(a,b,c) \in \mathcal{U}_{\mathrm{ad}}$.
\end{thm}

\begin{proof}
Testing \eqref{adjoint:p:weak} with $\Phi_{u} \in L^{2}(0,T;H^{3})$ leads to
\begin{align*}
& \int_{\Omega} \beta_{\Omega}(\varphi_{*}(T) - \varphi_{\Omega}) \Phi_{u}(T) \dx + \int_{0}^{T} \int_{\Omega} \beta_{Q}(\varphi_{*} - \varphi_{Q}) \Phi_{u} \dx \dt \\
& \quad = \int_{0}^{T} \inner{(\Phi_{u})_{t}}{p}_{H^{1}} + \int_{\Omega} \underbrace{\left ( \beta \eps \Laplace \Phi_{u} - \beta \eps^{-1} \Psi''(\varphi) \Phi_{u} \right )}_{= \; - \Xi_{u} \text{ by } \eqref{Lin:2}} q \dx \dt \\
& \qquad + \int_{\Omega} m'(\varphi_{*}) \Phi_{u} \nabla (\mu_{*} - \chi \sigma_{*}) \cdot \nabla p - \PP_{*} f'(\varphi_{*}) g(\sigma_{*}) p \; \Phi_{u} + \CC_{*} h'(\varphi_{*}) \sigma_{*} \Phi_{u} r \dx \dt.
\end{align*}
Meanwhile, testing \eqref{Lin:1:weak} with $p$ yields
\begin{align*}
& \int_{0}^{T} \inner{(\Phi_{u})_{t}}{p}_{H^{1}} \dt + \int_{\Omega} m'(\varphi_{*}) \Phi_{u} \nabla (\mu_{*} - \chi \sigma_{*}) \cdot \nabla p - \PP_{*} f'(\varphi_{*}) g(\sigma_{*}) \Phi_{u} p \dx \dt \\
& \quad = \int_{0}^{T} \int_{\Omega} - m(\varphi_{*}) \nabla (\Xi_{u} - \chi_{*} \Sigma_{u} - u_{\chi} \nabla \sigma_{*}) \cdot \nabla p  \dx \dt \\
& \qquad +  \int_{0}^{T} \int_{\Omega} \PP_{*} f(\varphi_{*}) g'(\sigma_{*}) \Sigma_{u} p + u_{\PP} f(\varphi_{*}) g(\sigma_{*}) p \dx \dt.
\end{align*} 
We obtain from combining the above two equalities:
\begin{align*}
& \int_{\Omega} \beta_{\Omega} (\varphi_{*}(T) -\varphi_{\Omega}) \Phi_{u}(T) \dx + \int_{0}^{T} \int_{\Omega} \beta_{Q} (\varphi_{*} - \varphi_{Q}) \Phi_{u} \dx \dt \\
& \quad = \int_{0}^{T} \int_{\Omega} \underbrace{- \Xi_{u} q - m(\varphi_{*}) \nabla \Xi_{u} \cdot \nabla p}_{= \; 0 \text{ by } \eqref{adjoint:q}} \; + \; m(\varphi_{*}) \nabla (\chi_{*} \Sigma_{u} + u_{\chi} \nabla \sigma_{*}) \cdot \nabla p  \dx \dt \\
& \qquad +  \int_{0}^{T} \int_{\Omega} \PP_{*} f(\varphi_{*}) g'(\sigma_{*}) \Sigma_{u} p + u_{\PP} f(\varphi_{*}) g(\sigma_{*}) p  + \CC_{*} h'(\varphi_{*}) \sigma_{*} \Phi_{u} r \dx \dt.
\end{align*}
Then, testing \eqref{Lin:3} with $r$ and \eqref{adjoint:r} with $\Sigma_{u}$ gives
\begin{align*}
& \int_{0}^{T} \int_{\Omega}  - \CC_{*} (h'(\varphi_{*}) \Phi_{u} \sigma_{*} r + h(\varphi_{*}) \Sigma_{u} r ) - u_{\CC} h(\varphi_{*}) \sigma_{*} r \dx \dt \\
& \quad = \int_{0}^{T} \int_{\Omega} (\Sigma_{u})_{t} r + \nabla \Sigma_{u} \cdot \nabla r \dx \dt \\
& \quad = \int_{0}^{T} \int_{\Omega} \chi_{*} m(\varphi_{*}) \nabla p \cdot \nabla \Sigma_{u} - \CC_{*} h(\varphi_{*}) \Sigma_{u} r + \PP_{*} f(\varphi_{*}) g'(\sigma_{*}) p \Sigma_{u} \dx \dt.
\end{align*}
This implies that
\begin{align*}
& \int_{\Omega} \beta_{\Omega} (\varphi_{*}(T) -\varphi_{\Omega}) \Phi_{u}(T) \dx + \int_{0}^{T} \int_{\Omega} \beta_{Q} (\varphi_{*} - \varphi_{Q}) \Phi_{u} \dx \dt \\
& \quad = \int_{0}^{T} \int_{\Omega} u_{\chi} m(\varphi_{*}) \nabla \sigma_{*} \cdot \nabla p + u_{\PP} f(\varphi_{*}) g(\sigma_{*}) p - u_{\CC} h(\varphi_{*}) \sigma_{*} r \dx \dt.
\end{align*}
\end{proof}

\section{Results for constant mobility}
\label{sec:ConstMob}
In the proof of Theorem \ref{thm:Exist} we required a $C^{4}$-boundary to deduce
the high order estimates $\nabla \varphi \in L^{\infty}(0,T;L^{\infty})$ and $\mu \in L^{\infty}(0,T;H^{1})$ so that we can prove continuous dependence with a variable mobility $m(\varphi)$.  A natural question is whether a similar well-posedness result holds for general convex domain $\Omega \subset \R^{2}$ with polygonal boundary, which may be more suited for numerical analysis and simulations.  The answer is positive when the mobility $m(\varphi)$ is constant.  Due to consideration of a possibly non-smooth boundary, at best we can expect is $H^{2}(\Omega)$-regularity.  However, in light of the lower regularity, the assumptions on the potential $\Psi$ can be further relaxed:

\begin{enumerate}[label=$(\mathrm{A5})$, ref = $\mathrm{A5}$]
\item \label{ass:convex:Psi} The potential $\Psi \in C^{2}(\R)$ is non-negative and there exists positive constants $R_{1}$, $R_{2}$, $R_{3}$, $R_{4}$ and $R_{5}$ such that for all $s, t \in \R$,
\begin{align*}
\Psi(s) \geq R_{1} \abs{s}^{2} - R_{2}, \quad \abs{\Psi'(s)} & \leq R_{3} \left ( 1 + \Psi(s) \right ), \quad \abs{\Psi''(s)} \leq R_{4} \left ( 1 + \abs{s}^{q} \right ), \\
\abs{\Psi'(s) - \Psi'(t)} & \leq R_{5} \left ( 1 + \abs{s}^{r} + \abs{t}^{r} \right ) \abs{s - t}, \\
\abs{\Psi''(s) - \Psi''(t)} & \leq R_{5} \left ( 1 + \abs{s}^{r-1} + \abs{t}^{r-1} \right ) \abs{s - t}
\end{align*}
for some exponents $q \in [1,\infty)$ and $r \in [1,\infty)$.
\end{enumerate}

The assertion is given as follows.
 
\begin{thm}[Well-posedness for convex polygonal domains]\label{thm:convex:dom}
Let $\Omega \subset \R^{2}$ be a convex polygonal domain, and in addition assume $m(\cdot) \equiv 1$ and $\varphi_{0}, \sigma_{0} \in H^{2}_{N}(\Omega)$, then under \eqref{assump:h}, \eqref{assump:para} and \eqref{ass:convex:Psi}, there exists a triplet of functions $(\varphi, \mu, \sigma)$ with
\begin{align*}
\varphi & \in L^{\infty}(0,T;H^{2}_{N}(\Omega)) \cap H^{1}(0,T;L^{2}(\Omega)) \cap C^{0}(\overline{Q}), \\
\mu & \in L^{\infty}(0,T;L^{2}(\Omega)) \cap L^{2}(0,T;H^{2}_{N}(\Omega)), \\
\sigma & \in L^{\infty}(0,T;H^{2}_{N}(\Omega)) \cap H^{1}(0,T;H^{2}_{N}(\Omega)) \cap C^{0}(\overline{Q}),
\end{align*}
and satisfies $\varphi(0) = \varphi_{0}$, $\sigma(0) = \sigma_{0}$ in $L^{2}(\Omega)$, and \eqref{CH:1}-\eqref{CH:3} a.e. in $Q$.  Furthermore, let $\{(\varphi_{i}, \mu_{i}, \sigma_{i})\}_{i = 1}^{2}$ denote two solutions to \eqref{Intro:CH} with the above regularities corresponding to data $\{ \varphi_{0,i}, \sigma_{0,i}, \PP_{i}, \CC_{i}, \chi_{i}\}_{i=1}^{2}$.  Then, there exists a positive constant C, not depending on the differences $\varphi_{1} - \varphi_{2}$, $\mu_{1} - \mu_{2}$, $\sigma_{1} - \sigma_{2}$, $\PP_{1} - \PP_{2}$, $\chi_{1} - \chi_{2}$, $\CC_{1} - \CC_{2}$, $\varphi_{0,1} - \varphi_{0,2}$ and $\sigma_{0,1} - \sigma_{0,2}$, such that
\begin{equation}
\label{Ctsdep:result:convex}
\begin{aligned}
& \norm{\varphi_{1} - \varphi_{2}}_{L^{\infty}(H^{1}) \cap L^{2}(H^{2})} + \norm{\sigma_{1} - \sigma_{2}}_{L^{\infty}(H^{1}) \cap L^{2}(H^{2})} + \norm{\mu_{1} - \mu_{2}}_{L^{2}(H^{1})} \\
& \quad \leq C \left ( \abs{\PP_{1} - \PP_{2}} + \abs{\CC_{1} - \CC_{2}} + \abs{\chi_{1} - \chi_{2}} + \norm{\varphi_{0,1} - \varphi_{0,2}}_{H^{1}} + \norm{\sigma_{0,1} - \sigma_{0,2}}_{H^{1}} \right ).
\end{aligned}
\end{equation}
\end{thm}

\begin{proof}
We will sketch the derivation of the a priori estimates required by a Galerkin approximation.  Below, the symbol $C$ denotes positive constants that are independent of $(\varphi, \mu, \sigma)$ and may vary from line to line. 
\paragraph{First estimate.}  Testing \eqref{CH:1} (for $m = 1$) with $\mu + \chi \sigma$, \eqref{CH:2} with $\varphi_{t}$ and \eqref{CH:3} with $D \sigma$, for some positive constant $D$ yet to be determined, leads to the energy identity
\begin{equation}\label{1st:equality}
\begin{aligned}
& \frac{\dd}{\dt}  \int_{\Omega} \left ( A \Psi(\varphi) + \frac{B}{2} \abs{\nabla \varphi}^{2} + \frac{D}{2} \abs{\sigma}^{2} \right ) \dx \\
& \qquad + \int_{\Omega} \left (  \abs{\nabla \mu}^{2} + D\abs{\nabla \sigma}^{2} + D \CC h(\varphi) \abs{\sigma}^{2} \right ) \dx \\
& \quad = \int_{\Omega} \left ( \PP f(\varphi) g(\sigma) (\mu + \chi \sigma) - \chi \nabla \mu \cdot \nabla \sigma \right ) \dx.
\end{aligned}
\end{equation}
We estimate the right-hand side as follows:
\begin{equation}\label{1st:RHS}
\begin{aligned}
\abs{\mathrm{RHS}} & \leq C \norm{\mu - \mean{\mu}}_{L^{1}} + C \abs{\mean{\mu}} + C \norm{\sigma}_{L^{2}}^{2} + \frac{1}{4}\norm{\nabla \mu}_{L^{2}}^{2} + \chi^{2} \norm{\nabla \sigma}_{L^{2}}^{2} \\
& \leq C + \frac{1}{2} \norm{\nabla \mu}_{L^{2}}^{2} + \chi^{2} \norm{\nabla \sigma}_{L^{2}}^{2} + C \norm{\sigma}_{L^{2}}^{2} + C \norm{\Psi(\varphi)}_{L^{1}},
\end{aligned}
\end{equation}
where we have used \eqref{ass:convex:Psi} so that
\begin{align*}
& C \norm{\mu - \mean{\mu}}_{L^{1}} \leq C \norm{\nabla \mu}_{L^{2}} \leq \frac{1}{4} \norm{\nabla \mu}_{L^{2}}^{2} + C,  \\
& \text{ and } \abs{\mean{\mu}} \leq \beta \eps^{-1} \norm{\Psi'(\varphi)}_{L^{1}} \leq C \left ( 1 + \norm{\Psi(\varphi)}_{L^{1}} \right ).
\end{align*}
Choosing $D > \chi^{2}$ and using \eqref{1st:RHS} to estimate the right-hand side of \eqref{1st:equality}, applying Gronwall's inequality in integral form \cite[Lem. 3.1]{GLNeu}, the Poincar\'{e} inequality and \eqref{ass:convex:Psi}, we obtain 
\begin{align}\label{Apri:1}
\norm{\Psi(\varphi)}_{L^{\infty}(0,T;L^{1})} + \norm{\varphi}_{L^{\infty}(0,T;H^{1})} + \norm{\sigma}_{L^{\infty}(0,T;L^{2}) \cap L^{2}(0,T;H^{1})} + \norm{\mu}_{L^{2}(0,T;H^{1})} \leq C.
\end{align}

\paragraph{Second estimate.}
Testing \eqref{CH:3} with $\sigma_{t}$ shows that $\sigma_{t}$ is bounded in $L^{2}(Q)$ and $\sigma$ is bounded in $L^{\infty}(0,T;H^{1})$.  Then, viewing \eqref{CH:3} as an elliptic equation for $\sigma$ with right-hand side belonging to $L^{2}(Q)$ and applying the elliptic regularity for convex domains \cite[\S 3, Thm. 3.2.1.3]{Grisvard} we see that $\sigma$ is bounded in $L^{2}(0,T;H^{2})$.  Furthermore, by the Sobolev embedding $H^{1}(\Omega) \subset L^{p}(\Omega)$ for all $p < \infty$, it is easy to see that $\mu - \beta \eps^{-1} \Psi'(\varphi) \in L^{2}(Q)$ and hence $\varphi$ is bounded in $L^{2}(0,T;H^{2})$.  This yields
\begin{align}\label{Apri:2}
\norm{\varphi}_{L^{2}(0,T;H^{2})} + \norm{\sigma}_{L^{2}(0,T;H^{2}) \cap L^{\infty}(0,T;H^{1}) \cap H^{1}(0,T;L^{2})} \leq C.
\end{align}

\paragraph{Third estimate.}
Next, testing \eqref{CH:1} (for $m = 1)$ with $\beta \eps \varphi_{t}$ and the time derivative of \eqref{CH:2} with $\mu$, upon adding leads to
\begin{align*}
\frac{\dd}{\dt} \frac{1}{2} \norm{\mu}_{L^{2}}^{2} + \beta \eps \norm{\varphi_{t}}_{L^{2}}^{2} & = \int_{\Omega} - \chi \beta \eps \Laplace \sigma \varphi_{t} + \PP \beta \eps f(\varphi) g(\sigma) \varphi_{t} + \beta \eps^{-1} \Psi''(\varphi) \varphi_{t} \mu \dx \\
& \leq C \left ( 1 + \norm{\Laplace \sigma}_{L^{2}} \right ) \norm{\varphi_{t}}_{L^{2}} + C \left ( 1 + \norm{\varphi}_{L^{\infty}}^{q} \right ) \norm{\mu}_{L^{2}} \norm{\varphi_{t}}_{L^{2}} \\
& \leq C \left ( 1 + \norm{\Laplace \sigma}_{L^{2}}^{2} + \norm{\Laplace \varphi}_{L^{2}} \norm{\mu}_{L^{2}}^{2} \right ) + \frac{\beta \eps}{2} \norm{\varphi_{t}}_{L^{2}}^{2},
\end{align*}
where we used the elliptic estimate \eqref{H2EllEst}, the boundedness of $\varphi \in L^{\infty}(0,T;H^{1})$ from \eqref{Apri:1} and the Br\'{e}zis--Gallouet inequality \eqref{Bre} to estimate
\begin{align*}
\norm{\varphi}_{L^{\infty}}^{q} \leq C \left ( 1 + \left ( \ln \left ( 1 + \norm{\Laplace \varphi}_{L^{2}} \right ) \right )^{\frac{q}{2}} \right ) \leq C \left ( 1 + \norm{\Laplace \varphi}_{L^{2}}^{\frac{1}{2}} \right ),
\end{align*}
compare \cite[\S 5, First estimate]{LamWu}.  The above differential inequality yields $\mu \in L^{\infty}(0,T;L^{2})$ and $\varphi_{t} \in L^{2}(Q)$.  Let us point out that $\mu(0) := \beta \eps^{-1} \Psi'(\varphi_{0}) - \beta \eps \Laplace \varphi_{0} \in L^{2}(\Omega)$ by the assumption $\varphi_{0} \in H^{2}(\Omega)$.  Then, viewing \eqref{CH:1} as an elliptic equation for $\mu$ with right-hand side in $L^{2}(Q)$, and employing \cite[\S 3, Thm. 3.2.1.3]{Grisvard} yields $\mu \in L^{2}(0,T;H^{2})$.  Similarly, viewing \eqref{CH:2} as an elliptic equation for $\varphi$ with $\mu - \beta \eps^{-1} \Psi'(\varphi) \in L^{\infty}(0,T;L^{2})$ we have $\varphi \in L^{\infty}(0,T;H^{2})$.  Altogether we have
\begin{align}
\label{Apri:3}
\norm{\varphi}_{L^{\infty}(0,T;H^{2}) \cap H^{1}(0,T;L^{2})} + \norm{\mu}_{L^{2}(0,T;H^{2}) \cap L^{\infty}(0,T;L^{2})} \leq C.
\end{align}

\paragraph{Fourth estimate.}
Following the proof of Theorem \ref{thm:Exist}, i.e., testing \eqref{CH:3} with $\Laplace \sigma_{t}$ yields $\sigma \in L^{\infty}(0,T;H^{2}) \cap H^{1}(0,T;H^{1})$.  Then, differentiating \eqref{CH:3} with respect to time, and testing with $\sigma_{t}$ and also with $\Laplace \sigma_{t}$ yields $\sigma_{t} \in L^{\infty}(0,T;H^{1}) \cap L^{2}(0,T;H^{2})$, and so we have
\begin{align}\label{Apri:4}
\norm{\sigma}_{L^{\infty}(0,T;H^{2}) \cap H^{1}(0,T;H^{2})} \leq C.
\end{align}
Then, the a priori estimates \eqref{Apri:1}, \eqref{Apri:2}, \eqref{Apri:3} and \eqref{Apri:4} are sufficient to deduce a strong solution to \eqref{Intro:CH} (for $m = 1$) with the regularities stated in Theorem \ref{thm:convex:dom}.

\paragraph{Continuous dependence.} Since $\varphi_{i} \in C^{0}(\overline{Q})$, we can still follow the proof of Theorem \ref{thm:Ctsdep}, where we note that due to a constant mobility $m(\cdot) = 1$, the term $I_{4}$ for the First estimate vanish.  The high order estimates $\nabla \varphi_{i} \in L^{\infty}(0,T;L^{\infty})$ and $\mu_{i} \in L^{\infty}(0,T;H^{1})$ are only required when estimating $I_{4}$ and thus in its absence we still obtain \eqref{ctsdep:est:1}.  The assertions \eqref{ctsdep:est:2}, \eqref{ctsdep:est:3} and \eqref{Ctsdep:sigmaH2} remain valid.  Furthermore, for a constant mobility the terms involving $\norm{\nabla (\mu_{1} - \chi_{1} \sigma_{1})}_{L^{4}(Q)}$ in \eqref{Ctsdep:varphi_t} and \eqref{Ctsdep:est:4:pre} vanish, and the estimate \eqref{Ctsdep:nablamu} remains valid.

\end{proof}

The assertion of Theorem \ref{thm:mini} for the existence of a minimizer to the parameter identification optimal control problem \eqref{Minimizer} still holds for the case of a general convex polygonal domain.  We now state the analogue of Theorem \ref{thm:Lin:state}, i.e., the solvability of the linearized state equation.

\begin{thm}\label{thm:Lin:state:Convex}
For any $(u_{\PP}, u_{\chi}, u_{\CC}) \in \R^{3}$, there exists a unique triplet $(\Phi_{u}, \Xi_{u}, \Sigma_{u})$ with
\begin{align*}
\Phi_{u} & \in L^{\infty}(0,T;H^{1}(\Omega)) \cap L^{2}(0,T;H^{2}_{N}(\Omega)) \cap H^{1}(0,T;(H^{1}(\Omega))'), \\
\Xi_{u} & \in L^{2}(0,T;H^{1}(\Omega)), \\
\Sigma_{u} & \in L^{\infty}(0,T;H^{1}(\Omega)) \cap L^{2}(0,T;H^{2}_{N}(\Omega)) \cap H^{1}(0,T;L^{2}(\Omega)),
\end{align*}
satisfying $\Phi_{u}(0) = 0$, $\Sigma_{u}(0) = 0$ in $L^{2}(\Omega)$, \eqref{Lin:2}, \eqref{Lin:3} a.e. in $Q$,
\begin{equation}\label{Lin:1:weak:Convex}
\begin{aligned}
\inner{(\Phi_{u})_{t}}{\zeta}_{H^{1}} & = \int_{\Omega} - \nabla (\Xi_{u} - \chi \Sigma_{u} - u_{\chi} \nabla \sigma) \cdot \nabla \zeta \dx \\
& \quad + \int_{\Omega} \left ( \PP ( g(\sigma) f'(\varphi) \Phi_{u} + f(\varphi) g'(\sigma) \Sigma_{u}) + u_{\PP} f(\varphi) g(\sigma) \right ) \zeta \dx
\end{aligned}
\end{equation}
for a.e. $t \in (0,T)$ and for all $\zeta \in H^{1}(\Omega)$.  Furthermore, there exists a positive constant $C$, not depending on $(\Phi_{u}, \Xi_{u}, \Sigma_{u}, u_{\PP}, u_{\chi}, u_{\CC})$ such that
\begin{equation}
\label{Bounds:Lin:Convex}
\begin{aligned}
& \norm{\Phi_{u}}_{L^{\infty}(0,T;H^{1}) \cap L^{2}(0,T;H^{2}) \cap H^{1}(0,T;(H^{1})')} + \norm{\Xi_{u}}_{L^{2}(0,T;H^{1})} \\
& \qquad + \norm{\Sigma_{u}}_{L^{\infty}(0,T;H^{1}) \cap L^{2}(0,T;H^{2}) \cap H^{1}(0,T;L^{2})} \\
& \quad \leq C \left ( \abs{u_{\PP}} + \abs{u_{\chi}} + \abs{u_{\CC}} \right ).
\end{aligned}
\end{equation}
\end{thm}

\begin{proof}
The proof is almost the same as the proof of Theorem \ref{thm:Lin:state}, where the terms $J_{6}$ and $J_{9}$ in the first estimate vanish, the terms involving $\norm{\nabla (\mu - \chi \sigma)}_{L^{\infty}(0,T;L^{2})}$ in \eqref{pdt:Phi:u} and also in the estimation of $K_{1}$ and $K_{3}$ vanish, and instead of the $H^{3}(\Omega)$-estimate appearing in \eqref{Apri:lin:4}, we have instead boundedness in $L^{2}(0,T;H^{2})$.
\end{proof}

For the Fr\'{e}chet differentiability of the control-to-state mapping, let us first comment that the variables $(\theta_{u}, \rho_{u}, \xi_{u})$ satisfy \eqref{Fdiff:1}, \eqref{Fdiff:2} and 
\begin{align}
\label{Fdiff:3:Convex}
\inner{(\theta_{u})_{t}}{\zeta}_{H^{1}} & = \int_{\Omega} - \bm{X}_{m} \cdot \nabla \zeta - \nabla \rho_{u} \cdot \nabla \zeta  + X_{\varphi} \zeta \dx.
\end{align}
where in the case of constant mobility
\begin{align*}
\bm{X}_{m} : = - \nabla (\chi \xi_{u} + u_{\chi}(\sigma_{u} - \sigma)).
\end{align*}
From the continuous dependence result \eqref{Ctsdep:result:convex} and the fact that $\varphi, \sigma \in L^{\infty}(0,T;L^{\infty})$, we see that \eqref{Xsigma}, \eqref{Fdiff:xi} remain valid (without the higher order terms $\mathcal{F}^{3}$ and $\mathcal{F}^{6}$, respectively, that arise from a variable mobility).  Meanwhile, in \eqref{bmX:m:est} only the term $L_{1}$ remains for the case of a constant mobility, and so \eqref{Fdiff:Xm} becomes
\begin{align*}
\abs{\int_{0}^{s} \int_{\Omega} \bm{X}_{m} \cdot \nabla \theta_{u} \dx \dt} \leq C \mathcal{F}^{4} + \norm{\nabla \theta_{u}}_{L^{2}(0,s;L^{2})}^{2} + \frac{\chi^{2} n_{1}^{2}}{2} \norm{\nabla \xi_{u}}_{L^{2}(0,s;L^{2})}^{2}.
\end{align*}
Then, following the rest of the proof of Theorem \ref{thm:Fdiff}, we infer the estimates \eqref{Fdiff:est:1}-\eqref{Fdiff:est:3}, and obtain the results of Theorem \ref{thm:Fdiff}.

The unique solvability of the adjoint system for constant mobilities now reads as
\begin{thm}
Under the hypothesis of Theorem \ref{thm:convex:dom}, $\varphi_{Q} \in L^{2}(Q)$, $\varphi_{\Omega} \in L^{2}(\Omega)$, for any $(\PP, \chi, \CC) \in \mathcal{U}_{\mathrm{ad}}$, there exists a unique triplet $(p,q,r)$ associated to $\mathcal{S}(\PP, \chi, \CC) = (\varphi, \mu, \sigma)$ with
\begin{align*}
p & \in L^{2}(0,T;H^{2}_{N}(\Omega)) \cap H^{1}(0,T;(H^{2}_{N}(\Omega))') \cap C^{0}([0,T];L^{2}(\Omega)), \\
q & \in L^{2}(0,T;L^{2}(\Omega)), \\
r & \in L^{2}(0,T;H^{2}_{N}(\Omega)) \cap L^{\infty}(0,T;H^{1}(\Omega)) \cap H^{1}(0,T;L^{2}(\Omega)),
\end{align*}
satisfying $p(T) = \beta_{\Omega}(\varphi_{*}(T) - \varphi_{\Omega})$, $r(T) = 0$ in $L^{2}(\Omega)$, \eqref{adjoint:q},
\begin{align*}
-r_{t} & = \Laplace r - \chi \Laplace p - \CC h(\varphi)r + \PP f(\varphi) g'(\sigma) p \text{ in } Q, 
\end{align*} 
and
\begin{align*}
0 & = \inner{-p_{t}}{\zeta}_{H^{2}} + \int_{\Omega} \beta \eps q \Laplace \zeta  - \beta \eps^{-1} \Psi''(\varphi) q \zeta \dx \\
& \quad + \int_{\Omega} \CC h'(\varphi) \sigma r \zeta - \PP f'(\varphi) g(\sigma) p \zeta  - \beta_{Q}(\varphi - \varphi_{Q}) \zeta \dx
\end{align*}
for a.e. $t \in (0,T)$ and for all $\zeta \in H^{2}_{N}(\Omega)$.
\end{thm}
Let us comment that for a constant mobility, the term $M_{1}$ on the right-hand side of \eqref{adjoint:est:pq} vanishes, and by taking into account the estimate \eqref{M2:M3} for $M_{2} + M_{3}$, we arrive at
\begin{align*}
-\frac{\dd}{\dt} \frac{1}{2} \norm{p}_{L^{2}}^{2} + \frac{3\beta \eps}{4} \norm{q}_{L^{2}}^{2} + D \norm{\nabla p}_{L^{2}}^{2} \leq C ( 1 + D^{2}) \left ( \norm{p}_{L^{2}}^{2} + \norm{r}_{L^{2}}^{2} + \norm{\varphi - \varphi_{Q}}_{L^{2}}^{2} \right ).
\end{align*}
Then, adding the above inequality to \eqref{adjoint:est:r} and choosing $D > \frac{\chi^{2} n_{1}^{2}}{2}$ yields \eqref{adjoint:1}.  Subsequently, by elliptic regularity we have \eqref{adjoint:2} and a slight modification taking into account the constant mobility, we also have \eqref{adjoint:3}.

We now state the first order necessary optimality conditions for a constant mobility.
\begin{thm}
Under the hypothesis of Theorem \ref{thm:convex:dom}, $\varphi_{Q} \in L^{2}(Q)$, $\varphi_{\Omega} \in L^{2}(\Omega)$, let $(\PP_{*}, \chi_{*}, \CC_{*}) \in \mathcal{U}_{\mathrm{ad}}$ denote a minimizer to \eqref{Minimizer} with corresponding state variables $(\varphi_{*}, \mu_{*}, \sigma_{*})$ and adjoint variables $(p,q,r)$.  Then, $(\PP_{*}, \chi_{*}, \CC_{*})$ necessarily satisfies
\begin{align*}
& \int_{0}^{T} \int_{\Omega} (a - \PP_{*}) f(\varphi_{*}) g(\sigma_{*})p + (b - \chi_{*}) \nabla \sigma_{*} \cdot \nabla p - (c - \CC_{*}) h(\varphi_{*}) \sigma_{*} r \dx \dt \\
& \quad + \beta_{\PP} \PP_{*} (a - \PP_{*}) + \beta_{\chi} \chi_{*} (b - \chi_{*}) + \beta_{\CC} \CC_{*} (c - \CC_{*}) \geq 0
\end{align*}
for all $(a,b,c) \in \mathcal{U}_{\mathrm{ad}}$.
\end{thm}


\section{A finite element approximation}
\label{sec:discreteScheme}
To apply the analytical results derived in the previous sections to numerical examples, 
in this section we state a numerical scheme for seeking a solution of  
the parameter identification optimal control problem \eqref{prob:Min:P}.  
It employs a semi-implicit discretization with respect to time and a finite element discretization with respect to space.

\subsection{The fully discrete scheme}
Let $0 = t_0 < t_1 < \ldots < t_{k-1}< t_k <\ldots < t_K = T$ denote a subdivision of $I=[0,T]$. 
 At time $t_k$ let $\mathcal T^k_h = \{T^k_i\}_{i=1}^{N_k}$ denote a subdivision of $\Omega$ 
 into closed cells $T^k_i$ such that $\bigcup_{i=1,\ldots,N} T_i = \overline \Omega$ hold, i.e., we assume $\Omega$ is a bounded domain with polygonal boundary.  On $\mathcal T^k_h$ we define the finite element space $V^k_h$
 as
\begin{align*}
V^k_h = \{v \in C(\overline\Omega) \, | \, v|_{T^k_i} \in P^1(T_i^k),\, i=1,\ldots,N_k\},
\end{align*}
where $P^1(T)$ denote the set of linear functions defined on $T$, i.e., 
the set of standard piecewise linear and continuous finite elements. 
 Note that for a suitable approximation of the phase field variable $\varphi$, 
 adaptive meshing is indispensable as we expect that $\varphi$ takes constant values 
 in large areas of the domain and $\nabla \varphi$ is only non-zero in a thin region.  
 Thus we consider different meshes for $\overline{\Omega}$ at each time instance.

Let us point out that there is an inconsistency with the regularity of 
the boundary $\pd \Omega$ for the numerical simulations and for the
 analytical results of Sections~\ref{sec:State} - \ref{sec:Opt}.  
 The former is Lipschitz, while for the latter we require $C^{4}$-regularity.  
 However, in the numerical simulations we observe that $(\varphi, \mu, \sigma)$
are constant in a neighbourhood of $\pd \Omega$, therefore we can restrict our 
attention to a subset $\Omega_{*} \subset \Omega$ which 
has a $C^{4}$-boundary and contains the evolution of the tumour.  
Then, the analytical results on well-posedness and optimal conditions apply in the restricted domain $\Omega_{*}$.

Denoting by $\varphi^k_h, \mu^k_h, \sigma^k_h \in V^k_h$ the discrete approximations of $\varphi$, $\mu$ and $\sigma$ at time instance $t_{k}$, respectively,  
we introduce the abbreviations 
\begin{align*}
\varphi_{\tau,h} := (\varphi^k_h)_{k=1}^K, \quad 
\mu_{\tau,h} := (\mu^k_h)_{k=1}^K, \quad 
\sigma_{\tau,h} := (\sigma^k_h)_{k=1}^K.
\end{align*}
Next we define the numerical scheme for the numerical approximation of
\eqref{CH:1}-\eqref{CH:2} on time instance $t_k$. 
 Let $\varphi^{k-1}, \sigma^{k-1} \in V^{k-1}_h$ be given, 
 and let $I_h^k : C(\overline \Omega) \to V^k_h$ denote the Lagrangian interpolation operator.  
 On time instance $t_k$, for $\tau := t_{k} - t_{k-1}$, 
 we search for $\varphi^k_h,\mu^k_h,\sigma^k_h \in V^k_h$ such that for all $v \in V^k_h$:
\begin{subequations}
\label{eq:FD:scheme}
\begin{alignat}{3}
\notag (\varphi^k_h,v) + \tau (m(I_h^k\varphi^{k-1}) \nabla \mu^k_h,\nabla v) & = (I_h^k \varphi^{k-1},v)  + \tau \PP (f(I_h^k\varphi^{k-1})g(\sigma_h^k),v) \\
& \quad - \tau \chi(m(I_h^k\varphi^{k-1})\nabla \sigma^k_h, \nabla v) \label{eq:FD:1_CH}\\
 \eps \beta(\nabla \varphi_h^k, \nabla v)
  + \frac{\beta}{\eps}( \Psi'(\varphi_h^k),v)^h - (\mu_h^k,v) & = 0, \label{eq:FD:2_CH}\\
(\sigma^k_h,v) + \tau (\nabla \sigma^k_h,\nabla v) & = (I_h^k\sigma^{k-1},v) - \tau \CC(h(I_h^k\varphi^{k-1})\sigma^k_h,v),
  \label{eq:FD:3_NU}
\end{alignat}
\end{subequations}
where $(\phi,\psi) = \int_{\Omega} \phi \, \psi \dx$ denotes the $L^{2}(\Omega)$-inner product.  For $k=1$ we set $\varphi^0 := \Pi_h \varphi_0$ and $\sigma^0 := \Pi_h \sigma_0$, where
$\Pi_h$ denotes the $L^2$-projection onto $V^1_h$.  
In \eqref{eq:FD:2_CH} we use lumped integration $(u,v)^h = \int_\Omega I_h^k(uv)\dx$ 
for the integral involving $\Psi'(\varphi_{h}^{k})$. 

For the potential term $\Psi$ the double obstacle free energy density \cite{BloweyElliott, HHT}
\begin{align*}
\Psi_{\mathrm{do}}(\varphi) = \frac{1}{2}(1-\varphi^{2}) + I_{[-1,1]}(\varphi), \quad
I_{[-1,1]}(\varphi) = \begin{cases} 0 & \text{ if } \varphi \in [-1,1], \\
+\infty & \text{ otherwise},
\end{cases}
\end{align*}
would be an ideal choice since it has the property that $\varphi$ remains in the physically relevant interval $[-1,1]$. 
 However, for regularity reasons in the numerical simulations we work with a
 relaxed double obstacle potential $\Psi$ which is composed of a concave part
 $\Psi_{-}(\varphi) := \frac{1}{2}(1-\varphi^{2})$ 
 and a convex part $\Psi_{+}(\varphi) := \frac{s}{2}
 \Lambda_{\rho}(\varphi)$, where for positive constants $s \gg 0$ and
 $\rho > 0$, we define
$\Lambda_\rho^\prime(\varphi) = \lambda_{\rho}(\varphi) :=
\max_{\rho}(0,\varphi-1) + \min_{\rho}(0, \varphi+1)$,
with regularized max and min functions as considered in \cite[(2.5)]{HK}. 
Note that the parameter $s$ allows us to control
the violation of the bound $\varphi \in [-1,+1]$.

\begin{thm}
Let $\varphi_0,\sigma_0 \in H^1(\Omega)$ be given. Then there exists
a solution $(\varphi_h^k,\mu_h^k,\sigma_h^k)_{k=1}^K \in (V^k_h)_{k=1}^K$ to 
\eqref{eq:FD:scheme} for $k=1,\ldots,K$.
If $\tau$ is sufficiently small, then this solution is unique. 
 Moreover, there exists a constant $C>0$ depending only on $\PP$, $\chi$, $\CC$, $\norm{\varphi_{0}}_{H^{1}}$ and $\norm{\sigma_{0}}_{H^{1}}$, such that
\begin{align}
  \|\varphi_{\tau,h}\|_{l^2(H^1(\Omega))}
  +\|\mu_{\tau,h}\|_{l^2(H^1(\Omega))}
  +\|\sigma_{\tau,h}\|_{l^2(H^1(\Omega))} \leq C.
  \label{eq:FD:bnd}
\end{align}
\end{thm}
\begin{proof}
The unique existence of $\sigma^k_h$ for $k=1,\ldots,K$ follows from Lax--Milgram's theorem. Note
that \eqref{eq:FD:3_NU} is decoupled from \eqref{eq:FD:1_CH}--\eqref{eq:FD:2_CH} at every time
instance.
Then the existence of a solution $(\varphi^k_h,\mu^k_h)$ to \eqref{eq:FD:1_CH}--\eqref{eq:FD:2_CH} at every time instance follows from standard results for the Cahn--Hilliard equation \cite{Blank}.  This solution is unique for small $\tau$.  From the above, the solution is bounded in every time instance and upon summing yields the estimate \eqref{eq:FD:bnd}.
\end{proof}

\subsection{The discrete optimization problem and solution approach}
Now we can define the fully discrete analogue to our inverse problem, namely
\begin{equation}
  \tag{$P_h$}\label{prob:FD:Ph}
  \begin{aligned}
    \min_{(\PP,\chi,\CC) \, \in \, \mathcal{U}_{\mathrm{ad}} } J(\varphi_{\tau,h}, \PP,\chi,\CC)
& := \frac{\beta_Q}{2}\| \varphi_{\tau,h}-\varphi_Q\|^2_{L^2(Q)} +
\frac{\beta_\Omega}{2}\|\varphi^K_h - \varphi_\Omega\|^2_{L^2(\Omega)}\\
& \qquad + \frac{\beta_\PP}{2}|\PP-\PP_d|^2 + \frac{\beta_\chi}{2}|\chi-\chi_d|^2
    + \frac{\beta_\CC}{2}|\CC-\CC_d|^2\\
& \text{subject to } (\varphi_{\tau,h},\mu_{\tau,h}, \sigma_{\tau,h}) \text{ solving }
    \eqref{eq:FD:1_CH}-\eqref{eq:FD:3_NU}.  \end{aligned}
\end{equation}

The existence of at least one minimizer for \eqref{prob:FD:Ph} follows from the direct method as for
the continuous problem in Theorem~\ref{thm:mini}.  To actually find a minimizer for \ref{prob:FD:Ph} we use a Gauss--Newton approach in a trust region
frame work, following \cite{BBV, NocedalWright}.
At every step of this algorithm, we solve a linear-quadratic minimization problem obtained 
by substituting 
$\varphi_{\tau,h}$ in \eqref{prob:FD:Ph} by its linearization with respect to $\PP,\chi,\CC$ at the
current iterate.
We couple this with a trust-region approach to
restrict the lengths of the resulting steps, which guarantees that the linear-quadratic minimization problem is approximating \eqref{prob:FD:Ph} sufficiently well. We stress that such a sensitivity approach is
feasible as we only consider three controls.
We skip the linearization here for brevity, but it is the discrete analogue to \eqref{Lin:State}.
We stop the optimization routine as soon as $|\nabla J| \leq 10^{-2}$ or when the relative change
of the current iterate $(\PP^i,\chi^i,\CC^i)$ is smaller then $10^{-4}$.

\section{Numerical experiments}
\label{sec:Numerics}

Let us present numerical examples to illustrate our approach.
The implementation is written in C++ using the finite element toolbox FEniCS \cite{fenics_book}
and meshes provided by the finite element toolbox  ALBERTA \cite{alberta_book}.
Let us first specify some aspects of the implementation.

\subsection{Adaptive meshing} \label{ssec:N:adapt}
As the functions $\varphi$, $\mu $, and $\sigma$ may undergo large variations in small regions, such as the growing front of the tumour, adaptive meshing is necessary.
Here we use the sum of the $L^2$-norms of the jumps of the gradients of $\varphi_h$, $\mu_h$,
and $\sigma_h$ across edges in normal direction as indicator and apply a D\"orfler marking scheme
\cite{Doerfler, HHK, Verfuerth} to adapt the mesh at every time instance before proceeding 
to the next time instance.
One might also apply residual based error estimation as proposed for different phase field models in \cite{HHT,GarHK_CHNS_AGG_linearStableTimeDisc}. 
We fix $V_{\min} =  \frac{1}{2}(\frac{\pi \eps}{R})^2$ as smallest volume present in the
computational mesh, where $R$ denotes a resolution of the interfacial region.
Note that the
transition zone from $\varphi \approx -1$ to $\varphi \approx +1$ 
has a width of approximately $\pi \eps$, and $V_{\min}$ is chosen such that we resolve this zone
with $R=16$ elements.

\subsection{The fixed parameters}
We set $\Omega = (-5,5)^2$ and $T=8$. We resolve the time interval $I = [0,T]$ with steps of length
$\tau_k \equiv \tau = 0.05$, i.e., $K=160$, and the spatial domain with 50 triangles per
spatial direction as a macro triangulation with cells of size $V_{\max} = 0.01$, that we adapt
locally according to Section~\ref{ssec:N:adapt}. We further fix $\eps = 0.05$ and $\beta = 0.05$.
For the free energy density $\Psi$ we fix $s = 10^{4}$ and $\rho = 0.001$.

The functions $f$, $g$, $h$, and $m$ are  given by
\begin{align*}
  f(x)& = \frac{1}{2}(\cos (\pi \min(1,\max(x,-1))) +1 ),\\
  g(x)& = 
  \begin{cases}
  0 & \mbox{ if } x\leq 0,\\
  x^2(-\theta^{-2}x+2\theta^{-1}) & \mbox{ if } 0 < x<\theta,\\
  x & \mbox{ if } \theta \leq x \leq M-\theta,\\
  -\theta^{-2}(x-M)^3 -2\theta^{-1}(x-M)^2 + M & \mbox{ if } M-\theta < x < M,\\
  M & \mbox{ if } x\geq M,  
  \end{cases}\\
  h(x) &= \frac{1}{2}\left(\sin\left(\frac{\pi}{2} \min(1,\max(x,-1)) \right) + 1\right),\\
  m(x) &= (m_1-m_0)f(x) + m_0.
\end{align*}
Here $M = 10$ is a maximum amount of nutrition that can be used for proliferation,
 $\theta = 0.01$, and $m_1 = 1.0$, and $m_0 = 10^{-4}$.
 
The function $f$ is a smooth indicator function for the interface between tumour and healthy cells,
while the function $h$ is a smooth indicator function for the tumour cells, and $g$ is a smooth
cut-off function to limit the maximum amount of nutrients used for proliferation.
The mobility $m$ is chosen to be nearly-degenerate at $x = \pm 1$ in order to limit the growth of the tumour due to chemotaxis.  Finally the initial data is taken to be $\sigma_0 = 1$ uniformly in the domain, and 
\begin{align*}
  z_0& := \mbox{arctan}(\sqrt{s-1}),\\
  \Phi_0(z)& := 
  \begin{cases}
    -\Phi_0(-z) & \mbox{ if } z < 0,\\
    \sqrt{\frac{s}{s-1}}\sin(z) & \mbox{ if } 0 \leq z \leq z_0,\\
    \frac{1}{s-1}\left(s-\exp( \sqrt{s-1}(z_0-z)  )\right) & \mbox { else}, 
  \end{cases}\\
  \varphi_0(x)& = \Phi_0( \eps^{-1}(\|x\|_{l^8}  - 1)).
\end{align*}
Here $\Phi_0$ is the first order approximation of $\varphi$ for $\Psi$ 
(with $\rho = 0$), 
 and
$\varphi_0$ describes a rounded square centred at the origin which is realized by the unit circle in the $l^8$ norm.
Note that for $\rho=0$, as $s \to \infty$, $\Psi$ tends to the double-obstacle free energy
$\Psi_{do}$ and $\Phi_0$ tends to its well-known sinus-shaped optimal profile.

\subsection{The desired states}
\label{ssec:desiredState}
Currently we use synthetic data for the desired states $\varphi_Q$ and
$\varphi_\Omega$, i.e., we solve the system \eqref{eq:FD:scheme} for a given set of
parameters $(\PP,\chi,\CC)$.
Since in real-world applications such a function would be generated from measurements and thus
contains noise, we also consider adding uniformly distributed white noise with magnitude $\delta$
point wise at each nodes of the triangulation.

We generate the data $\varphi_Q := (\varphi_Q)_{\tau,h}$ and $\varphi_\Omega := (\varphi_Q)_h^K$
using the parameters $\PP = 7$, $\chi = 6$, and $\CC = 2$. 
In Figure~\ref{fig:N:cd} we show snapshots of the evolution without noise. Note that black and white
correspond to $\varphi \approx 1$ (tumour) and $\varphi \approx -1$ (non-tumour), respectively.
\begin{figure}
  \centering
  \fbox{
  \includegraphics[width=0.15\textwidth]{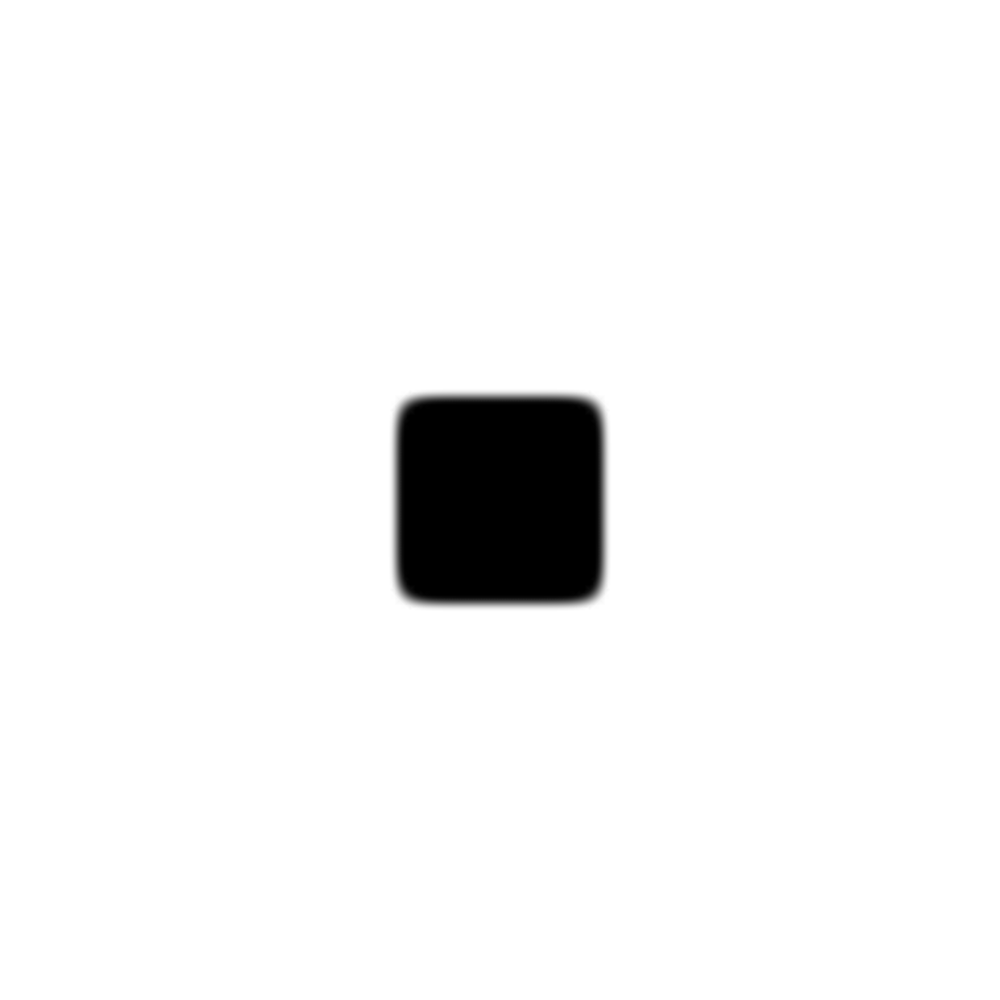}
  }
  \fbox{
  \includegraphics[width=0.15\textwidth]{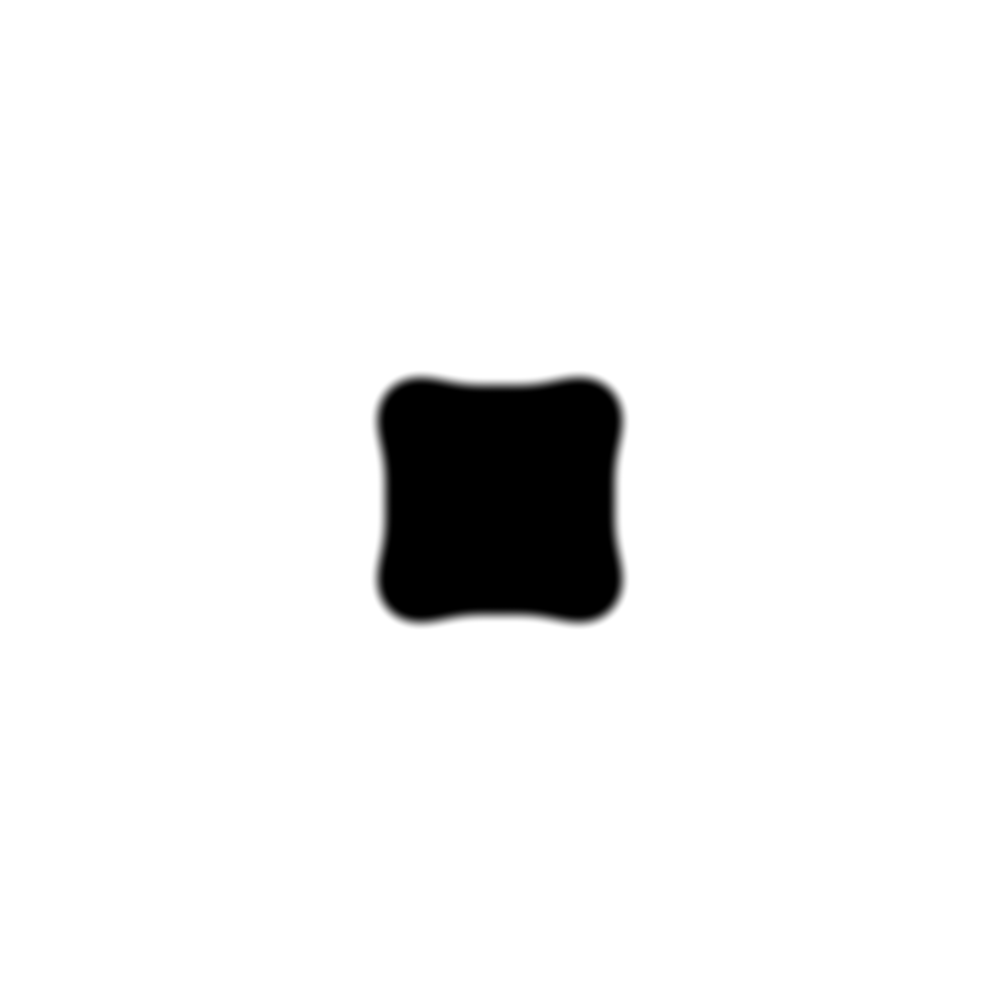}
  }
  \fbox{
  \includegraphics[width=0.15\textwidth]{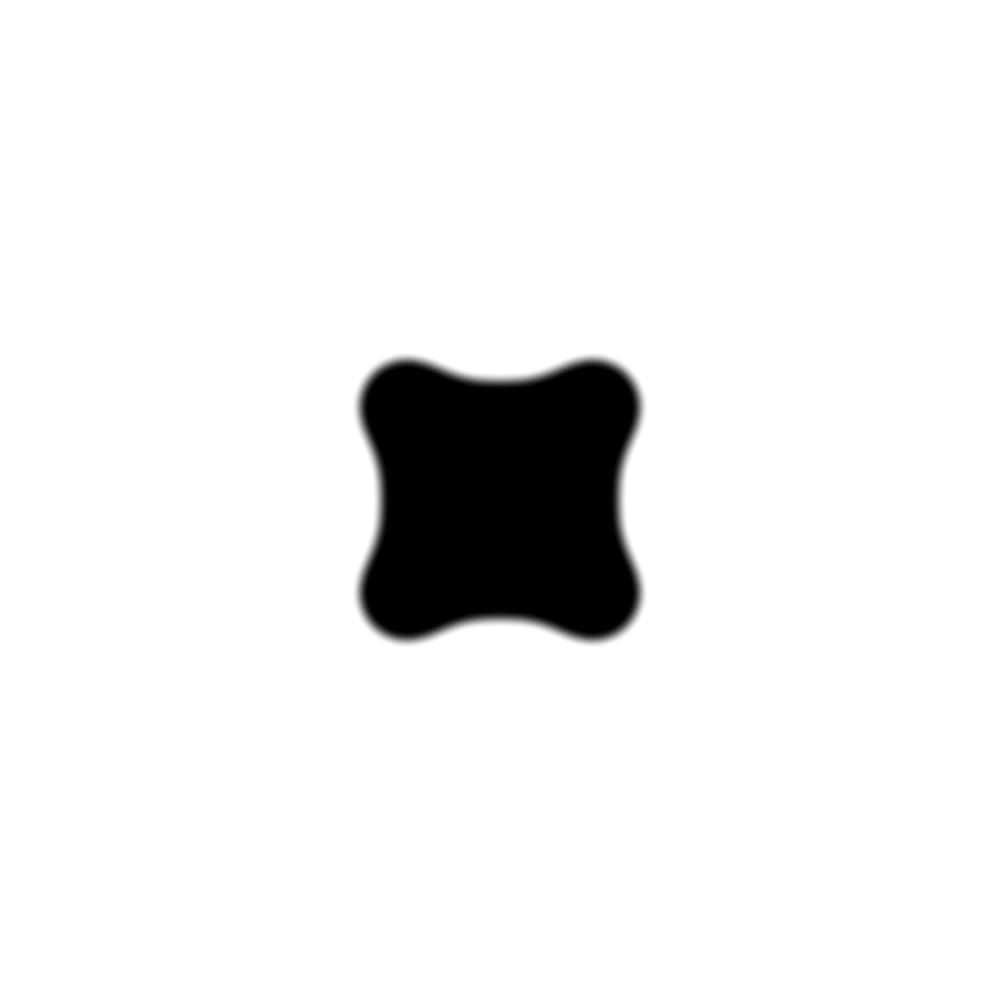}
  }

  \fbox{
  \includegraphics[width=0.15\textwidth]{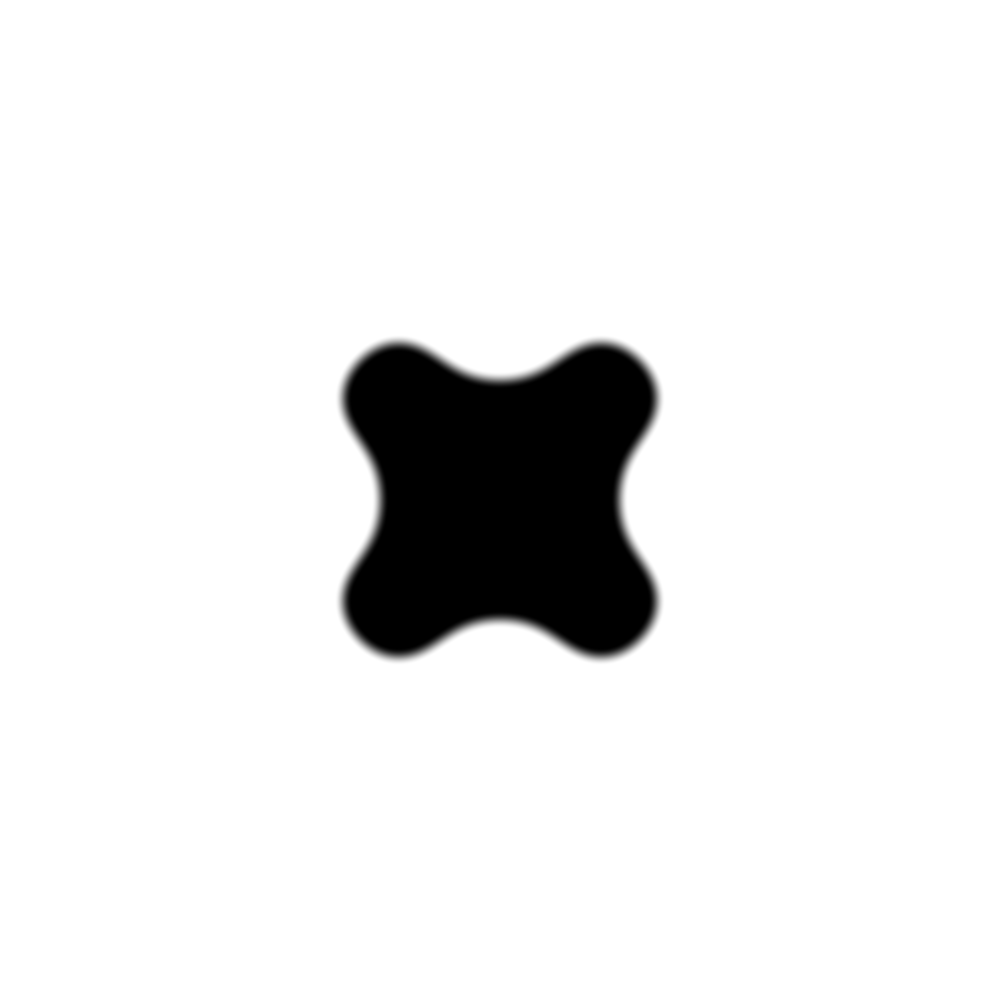}
  }
  \fbox{
  \includegraphics[width=0.15\textwidth]{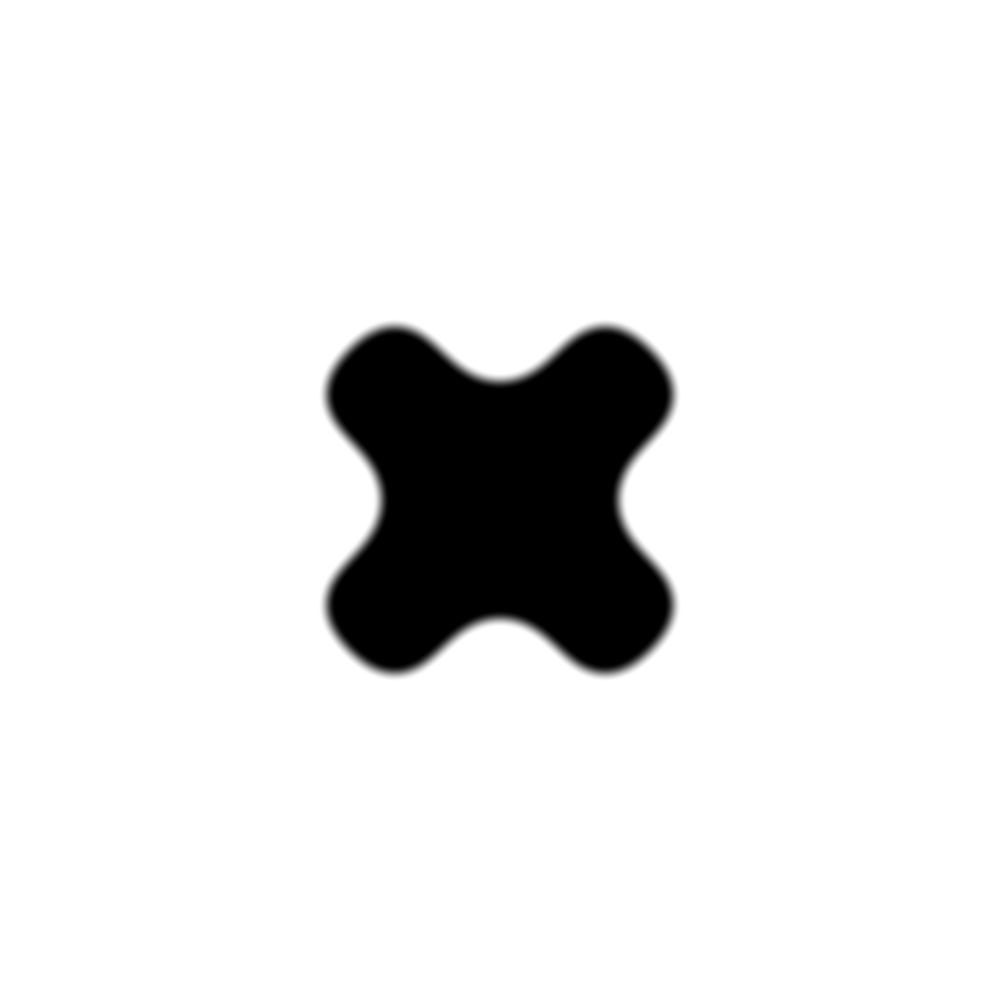}
  }
  \fbox{
  \includegraphics[width=0.15\textwidth]{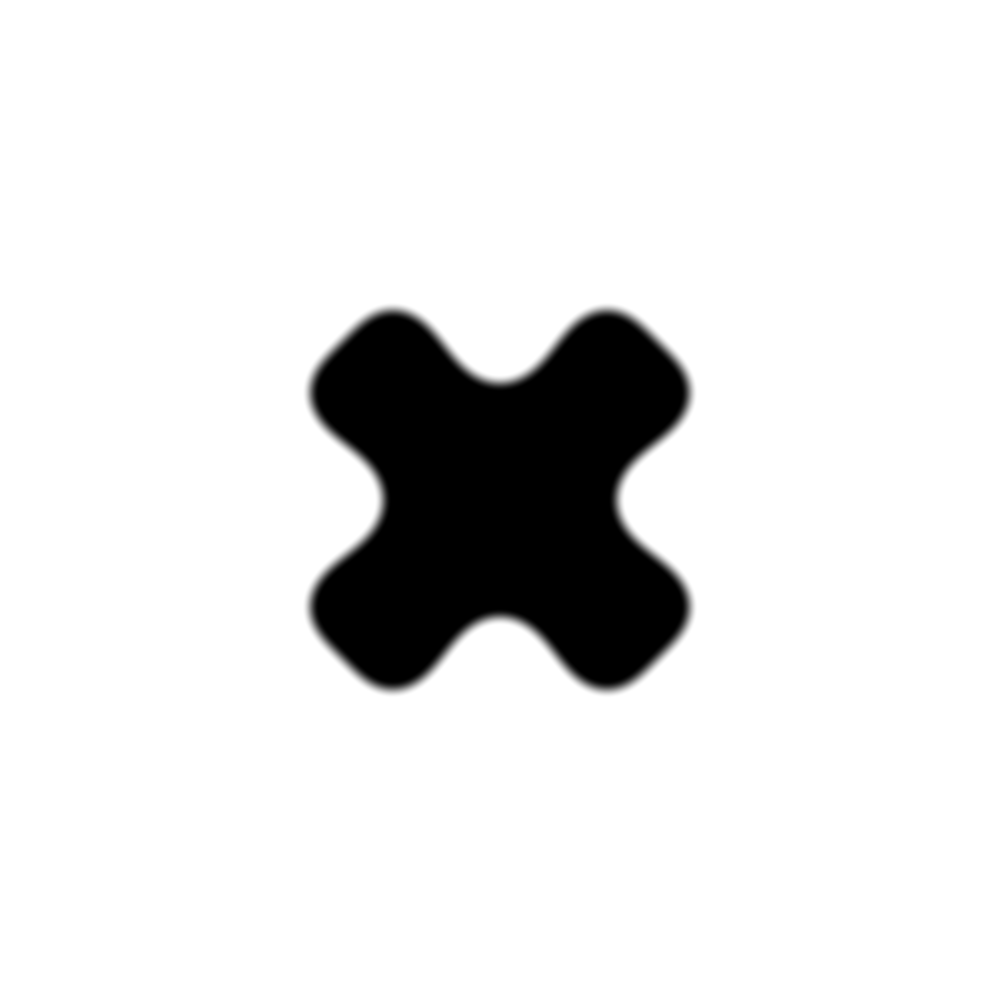}
  }

  \fbox{
  \includegraphics[width=0.15\textwidth]{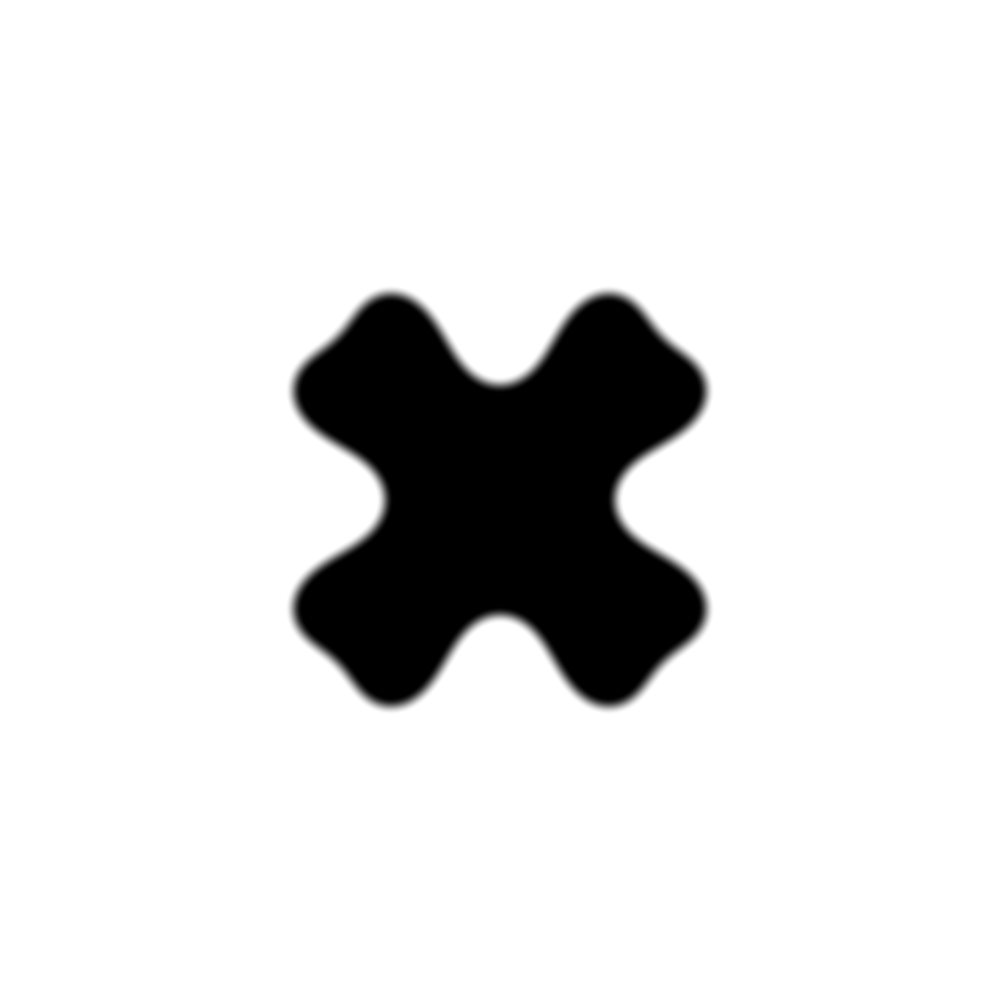}
  }
  \fbox{
  \includegraphics[width=0.15\textwidth]{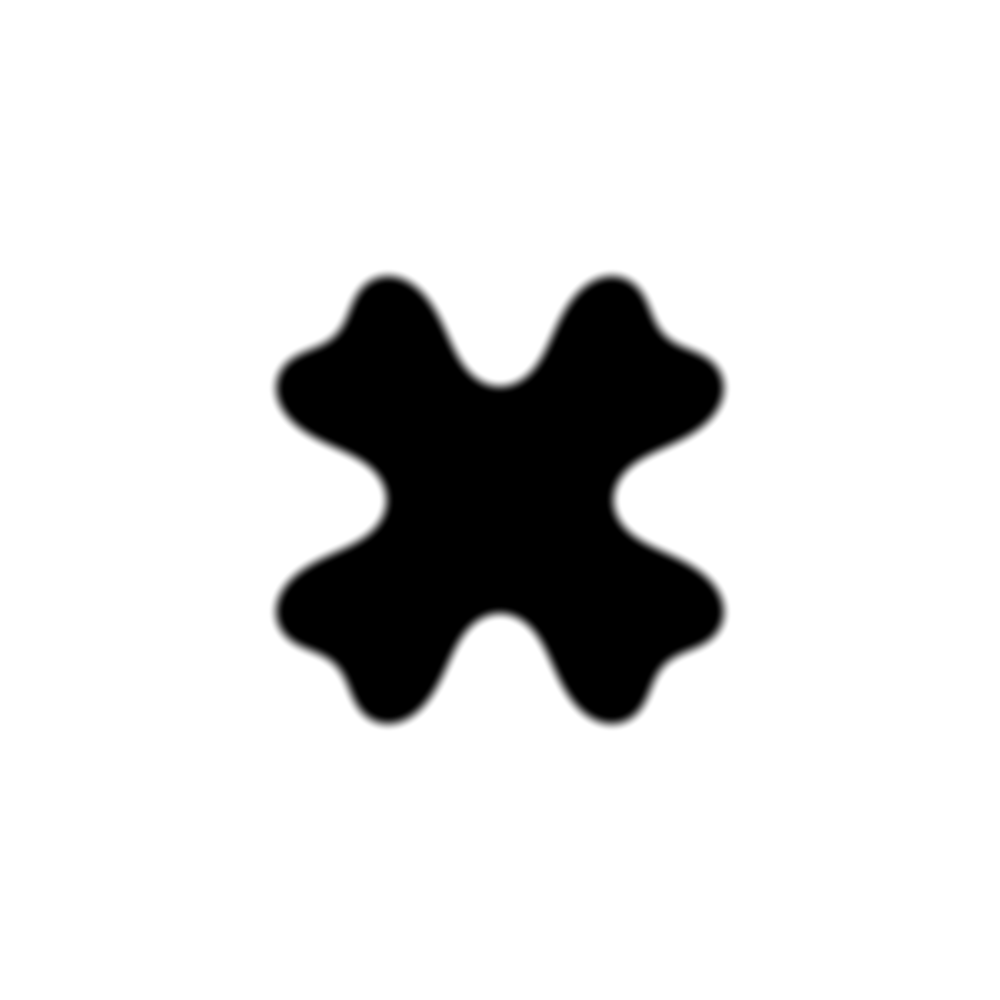}
  }
  \fbox{
  \includegraphics[width=0.15\textwidth]{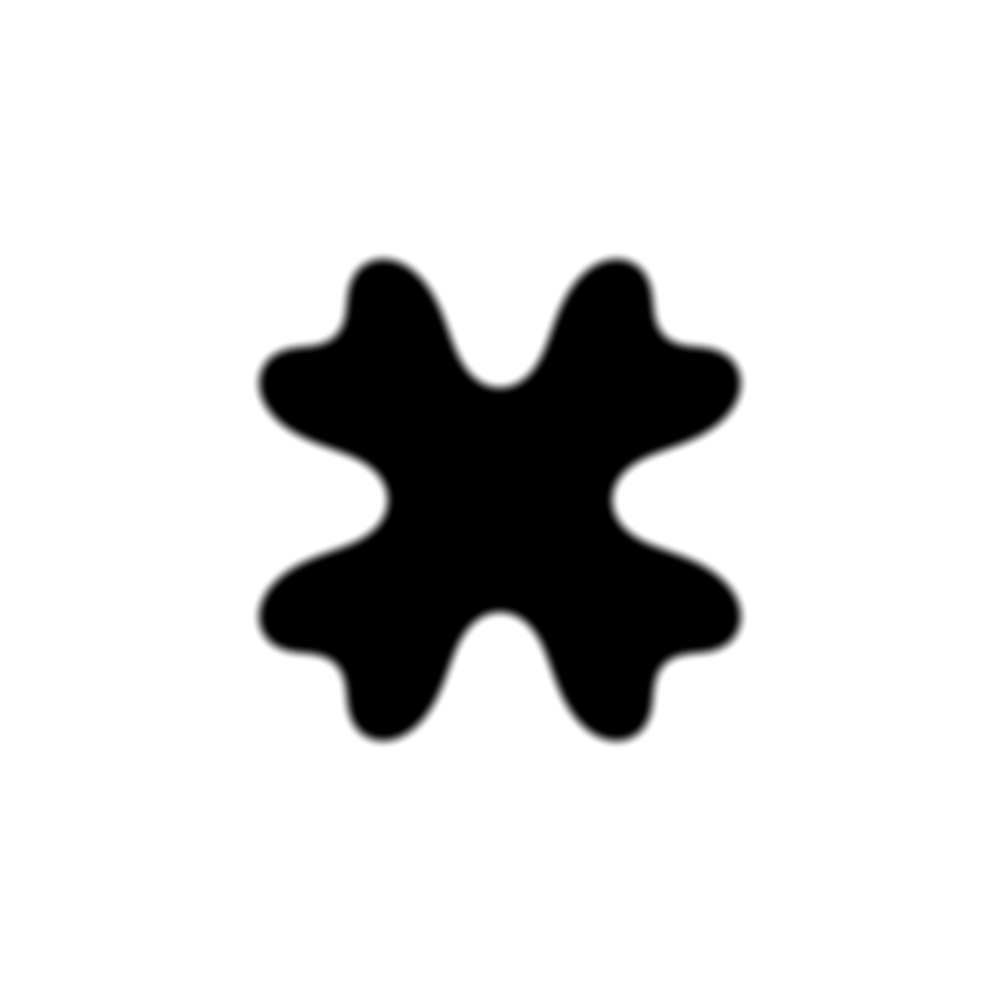}
  }

  \caption{
  Evolution of $\varphi_Q$ at 
  $t \in \{0.0, 1.0, 2.0, 3.0, 4.0, 5.0, 6.0, 7.0, 8.0\}$
  (top left to bottom right). We set $\varphi_\Omega := \varphi_Q(t = 8.0)$.
  }
  \label{fig:N:cd}
\end{figure}

\subsection{Recovery without noise}
\label{ssc:N:noNoise}
As a first test example let us consider the identification of parameters in the absence of noise,
i.e., $\delta = 0$.
We set $\PP_d =7$, $\chi_d = 6$, $\CC_d = 2$ and choose $\beta_{\PP} \equiv \beta_{\chi} \equiv \beta_{\CC}  = 10^{-8}$. 
Furthermore we set $\PP_{\infty} \equiv \chi_{\infty} \equiv \CC_{\infty} = 10$, and 
we initialize the iterative procedure with $\PP_0 \equiv \chi_0 \equiv \CC_0 = 0$.

 In Figure~\ref{fig:N:noNoise:PchiC} we show the evolution of $\PP$, $\chi$, and $\CC$ over the
 optimization steps for $\beta_Q = 1, \beta_\Omega = 0$ (left)
  and $\beta_\Omega = 1$ and $\beta_Q = 0$ (right). 
  On the left we observe a rapid increase of $\CC$ at the very beginning that is stopped by the bound
$\CC \leq \CC_{\infty} = 10$ at iteration 5, before its value is reduced again to the final value
$\CC^Q_* = 2.0003$ after 20 iterations. The increase and decrease is limited by the trust
region radius $\Delta =2$. Note that large values of $\CC$ generate large variations of $\sigma$
and contributes to a strong chemotaxis effect even when the value of $\chi$ is small. This might be the reason why $\chi$ is rather slow at increasing compared to the other parameters throughout the optimization.
The final values $\PP$ and $\chi$ are $\PP^Q_* = 7.0004$ and $\chi^Q_* = 5.9996$.
Note that the exact minimum is attained at $\PP = 7$, $\chi = 6$, $\CC = 2$.

On the right we have a similar behaviour, but now $\chi$ is increasing at the beginning of the
optimization, while $\PP$ and $\CC$ are approaching their final values quite monotonically. The
final values are $\PP^\Omega_* = 6.9987$, $\chi^\Omega_* = 6.0008$, and $\CC^\Omega_*
= 1.9990$.

Summarizing, we are able to recover the parameters of interest in the absence of noise in the data. 
Let us further point out that the choice $\beta_{\PP} = \beta_{\chi} = \beta_{\CC} = 10^{-8}$ 
implies we do not put significant weighting on the a priori 
knowledge $\PP_{d}, \chi_{d}, \CC_{d}$ for the recovery of parameters.  
Hence, any pollution in the form of errors in the a priori information 
has minimal influence in the parameter estimation, and we have
 observed similar final values for $\PP, \chi, \CC$ when we set $\PP_{d} = \chi_{d} = \CC_{d} = 0$.

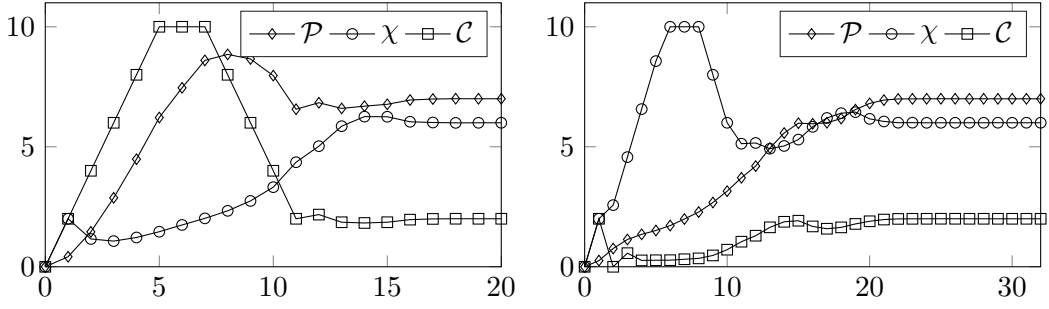
\begin{figure}
  \centering
  \input{ParaIdentFinger_tracking_noNoise_R16.tikz}
\input{ParaIdentFinger_endTime_noNoise_R16.tikz}

  \caption{The evolution of the parameters found by the optimizer over the iterations in the absence
  of noise, i.e.,  $\delta = 0$.
  On the left we use $\beta_Q = 1$ and $\beta_\Omega = 0$, and on the right we use $\beta_Q
  = 0$ and $\beta_\Omega = 1$. Note that the upper bound $\CC_\infty = 10$ is attained
  on the left for iterations $5,6,7$ and on the right the bound $\chi_\infty = 10$ is attained  for
  iterations $6,7,8$.
  The final values are $\PP^Q_* = 7.0004$, $\chi^Q_* = 5.9996$, and $\CC^Q_* = 2.0003$ after 20 iterations for the left
  setting, and $\PP^\Omega_* = 6.9987$, $\chi^\Omega_* = 6.0008$, and $\CC^\Omega_* = 1.9990$ after
  32 iterations for the right setting.  
  }
  \label{fig:N:noNoise:PchiC}
\end{figure}

\subsection{Recovery with noise}
Next we consider noisy data obtained from adding noise of maximum value $\delta =
0.05$,  which is $\approx 2.5\%$ to the amplitude of $\varphi$, to the given data $\varphi_Q$ 
as described in Section~\ref{ssec:desiredState}. 
Here we assume that the initial data for the numerical simulation is free of
noise. Otherwise, due to  the proliferation mechanism, the noise acts as seeds
for tumour growing all over the domain. Note that this might be suppressed  by
setting a threshold in $f$, such that proliferation is restricted to regions
with $\varphi > -1+\delta$.

In Figure~\ref{fig:N:noise:PchiC}
we show the evolution of the parameters for $\beta_Q = 1$ and $\beta_\Omega = 0$ on
the left and for $\beta_Q = 0$ and $\beta_\Omega = 1$ on the right. 
In both cases we
observe a similar evolution of the parameters as in the absence of noise, 
i.e., Section~\ref{ssc:N:noNoise}. 
For $\beta_Q = 1$ 
the solver now needs 22 iterations to reach the final values $\PP_*^Q = 7.0005$, $\chi_*^Q =
5.9996$, and $\CC_*^Q = 2.0003$,
while for $\beta_\Omega = 1$ we need 39 iterations to reach 
$\PP_*^\Omega = 6.9977$, $\chi_*^\Omega = 6.0023$, and $\CC_*^\Omega = 1.9988$.
Again the final values are close to the
desired ones $\PP_d = 7$, $\chi_d = 6$, $\CC_d = 2$.

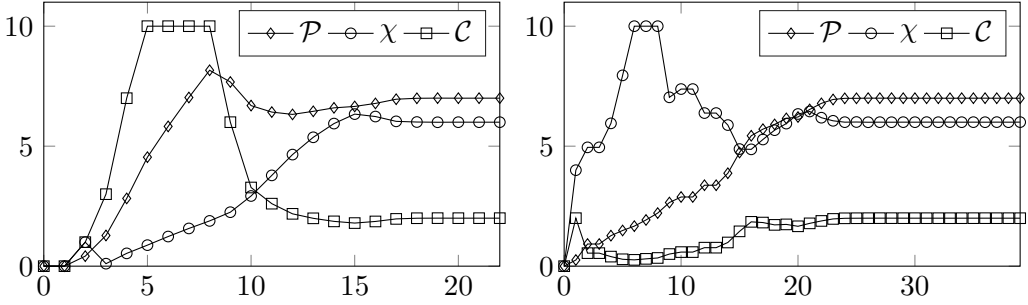
\begin{figure}
  \centering
  \input{ParaIdentFinger_tracking_noise_R16.tikz}
\input{ParaIdentFinger_endTime_noise_R16.tikz}

  \caption{The evolution of the parameters found by the optimizer over the iterations performed
  for $\delta = 0.05$.
 On the left we use $\beta_Q = 1$ and $\beta_\Omega = 0$, and on the right we use $\beta_Q
  = 0$ and $\beta_\Omega = 1$. The final values are
$\PP_*^Q = 7.0005$, $\chi_*^Q = 5.9996$, and $\CC_*^Q = 2.0003$ after 22 iterations
(left) and $\PP_*^\Omega = 6.9977$, $\chi_*^\Omega = 6.0023$, and $\CC_*^\Omega = 1.9988$
after 39 iterations (right).
  }
  \label{fig:N:noise:PchiC}  
\end{figure}

We also investigate the robustness of the parameter identification with respect to the level of noise, and the results are summarized in Table~\ref{tab:N:noise:severalDelta}. Here we
again use $\beta_Q = 1$ and $\beta_\Omega = 0$.
We identify the unknown parameters up to a noise level of
$\delta = 0.35$. This means that we vary the value of $\varphi$
point wise by up to 17.5\%. In our example, for $\delta = 0.35$ the
linearization close to optimum becomes a bad model for the actual equation and
the trust region method breaks down with the values shown in
Table~\ref{tab:N:noise:severalDelta}. We see that the algorithm is rather
robust with respect to size of $\delta$ and the number of iterations to find the optimal
values.  Furthermore, the optimal value of $J$ is increasing with the noise level as
expected.

\begin{table}
\centering
\begin{tabular}{c|ccc|ccc}
$\delta$ & $\PP_*^Q$ & $\chi_*^Q$ & $\CC_*^Q$ & \#it & $J$  \\
\hline
0.10&7.0006&5.9994&2.0004&20&0.66826\\
0.15&6.9994&6.0002&1.9997&23&1.4994\\
0.20&6.9984&6.0013&1.9988&27&2.6625\\
0.30&7.0002&5.9999&2.0003&22&5.9992\\
0.35&7.0014&5.9986&2.0007&28&8.1671
\end{tabular}
\caption{The identified parameter for several levels of noise $\delta$.
We also show the number of trust-region-Gauss-Newton steps (\#it) and the value
$J(\PP_*^Q,\chi_*^Q,\CC_*^Q)$. }
\label{tab:N:noise:severalDelta}
\end{table}

\newpage
%
%
%

\end{document}

%% file: ParaIdentFinger_tracking_noNoise_R16.tikz
%
%
\begin{tikzpicture}

\begin{axis}[%
width=0.4\textwidth,
height=3.5cm,
at={(0.0,0.0)},
scale only axis,
xmin=0,
xmax=20,
ymin=0,
ymax=11,
axis background/.style={fill=white},
legend style={legend cell align=left,align=left,draw=white!15!black, legend columns=-1},
legend pos = north east
]
\addplot [color=black,solid,mark=diamond,mark options={solid}]
  table[row sep=crcr]{%
0	0\\
1	0.41763217\\
2	1.4540706\\
3	2.87314064\\
4	4.4834656\\
5	6.21785996\\
6	7.45594378\\
7	8.60001324\\
8	8.84961997\\
9	8.65786455\\
10	7.96479554\\
11	6.56488567\\
12	6.83416235\\
13	6.59844437\\
14	6.69241102\\
15	6.77084711\\
16	6.94847405\\
17	6.99002385\\
18	7.00316543\\
19	7.00037638\\
20	7.00039231\\
};
\addlegendentry{$\PP$};

\addplot [color=black,solid,mark=o,mark options={solid}]
  table[row sep=crcr]{%
0	0\\
1	2\\
2	1.15845564\\
3	1.06952832\\
4	1.22044531\\
5	1.46031404\\
6	1.74241784\\
7	2.01663383\\
8	2.33082074\\
9	2.74356593\\
10	3.317049\\
11	4.35648215\\
12	5.02518114\\
13	5.85776983\\
14	6.25605868\\
15	6.2511874\\
16	6.04187109\\
17	6.01113696\\
18	5.99628182\\
19	5.99961338\\
20	5.99957452\\
};
\addlegendentry{$\chi$};

\addplot [color=black,solid,mark=square,mark options={solid}]
  table[row sep=crcr]{%
0	0\\
1	2\\
2	4\\
3	6\\
4	8\\
5	10\\
6	10\\
7	10\\
8	8\\
9	6\\
10	4\\
11	2\\
12	2.17392148\\
13	1.85558896\\
14	1.82479633\\
15	1.85585166\\
16	1.96349003\\
17	1.99382306\\
18	2.00186317\\
19	2.0002329\\
20	2.00030052\\
};
\addlegendentry{$\CC$};

\end{axis}
\end{tikzpicture}%

%% file: ParaIdentFinger_endTime_noNoise_R16.tikz
%
%
\begin{tikzpicture}

\begin{axis}[%
width=0.4\textwidth,
height=3.5cm,
at={(0.0,0.0)},
scale only axis,
xmin=0,
xmax=32,
ymin=0,
ymax=11,
axis background/.style={fill=white},
legend style={legend cell align=left,align=left,draw=white!15!black, legend columns=-1},
legend pos = north east
]
\addplot [color=black,solid,mark=diamond,mark options={solid}]
  table[row sep=crcr]{%
0	0\\
1	0.2564515\\
2	0.76025424\\
3	1.13631135\\
4	1.35037907\\
5	1.51176238\\
6	1.71687466\\
7	1.98944067\\
8	2.28019121\\
9	2.67831145\\
10	3.1510819\\
11	3.70766815\\
12	4.19604404\\
13	4.94740875\\
14	5.57244218\\
15	5.98326093\\
16	5.95751445\\
17	5.99938581\\
18	6.17572799\\
19	6.54844739\\
20	6.8072003\\
21	6.94555505\\
22	6.99308366\\
23	6.99606004\\
24	6.9993879\\
25	7.0008837\\
26	7.00056698\\
27	6.99833941\\
28	6.99959502\\
29	6.99734957\\
30	6.99888325\\
31	6.99818154\\
32	6.99866724\\
};
\addlegendentry{$\PP$};

\addplot [color=black,solid,mark=o,mark options={solid}]
  table[row sep=crcr]{%
0	0\\
1	2\\
2	2.56961531\\
3	4.56961531\\
4	6.56961531\\
5	8.56961531\\
6	10\\
7	10\\
8	10\\
9	8\\
10	6\\
11	5.12859679\\
12	5.15831826\\
13	4.91995445\\
14	5.03394878\\
15	5.29987625\\
16	5.83107809\\
17	6.19980714\\
18	6.39969795\\
19	6.44372694\\
20	6.16321016\\
21	6.05506421\\
22	6.00330485\\
23	6.00447488\\
24	5.99896435\\
25	5.9997702\\
26	5.99921953\\
27	6.00187477\\
28	5.99959951\\
29	6.00287781\\
30	6.00014407\\
31	6.0016714\\
32	6.00083985\\
};
\addlegendentry{$\chi$};

\addplot [color=black,solid,mark=square,mark options={solid}]
  table[row sep=crcr]{%
0	0\\
1	2\\
2	0\\
3	0.56637018\\
4	0.26518724\\
5	0.27197622\\
6	0.2746711\\
7	0.31300679\\
8	0.34972918\\
9	0.47192735\\
10	0.71129532\\
11	1.04412768\\
12	1.29419927\\
13	1.64322298\\
14	1.87415539\\
15	1.91833734\\
16	1.68036829\\
17	1.58475428\\
18	1.64119446\\
19	1.785182\\
20	1.89602941\\
21	1.96761073\\
22	1.99453982\\
23	1.99826976\\
24	1.99893822\\
25	2.00090862\\
26	2.00018798\\
27	1.99929141\\
28	1.99945249\\
29	1.99879849\\
30	1.99896953\\
31	1.99918573\\
32	1.99901308\\
};
\addlegendentry{$\CC$};

\end{axis}
\end{tikzpicture}%

%% file: ParaIdentFinger_tracking_noise_R16.tikz
%
%
\begin{tikzpicture}

\begin{axis}[%
width=0.4\textwidth,
height=3.5cm,
at={(0.0,0.0)},
scale only axis,
xmin=0,
xmax=22,
ymin=0,
ymax=11,
axis background/.style={fill=white},
legend style={legend cell align=left,align=left,draw=white!15!black, legend columns=-1},
legend pos = north east
]
\addplot [color=black,solid,mark=diamond,mark options={solid}]
  table[row sep=crcr]{%
0	0\\
1	0\\
2	0.41748411\\
3	1.2781582\\
4	2.82317227\\
5	4.54194939\\
6	5.82418398\\
7	7.02719843\\
8	8.16348801\\
9	7.67414302\\
10	6.6908963\\
11	6.42068692\\
12	6.32391026\\
13	6.45221232\\
14	6.59699916\\
15	6.65588172\\
16	6.78188009\\
17	6.95684905\\
18	6.99573383\\
19	7.00243944\\
20	6.99965823\\
21	7.00065096\\
22	7.00046778\\
};
\addlegendentry{$\PP$};

\addplot [color=black,solid,mark=o,mark options={solid}]
  table[row sep=crcr]{%
0	0\\
1	0\\
2	1\\
3	0.10357677\\
4	0.53310431\\
5	0.87616031\\
6	1.23460512\\
7	1.57184417\\
8	1.88581437\\
9	2.2488743\\
10	2.92399662\\
11	3.78245916\\
12	4.64938861\\
13	5.36976477\\
14	5.94155306\\
15	6.33634117\\
16	6.23507296\\
17	6.03442279\\
18	6.00624342\\
19	5.99673899\\
20	6.00047189\\
21	5.99923994\\
22	5.99956493\\
};
\addlegendentry{$\chi$};

\addplot [color=black,solid,mark=square,mark options={solid}]
  table[row sep=crcr]{%
0	0\\
1	0\\
2	1\\
3	3\\
4	7\\
5	10\\
6	10\\
7	10\\
8	10\\
9	6\\
10	3.27793936\\
11	2.60443936\\
12	2.18081017\\
13	1.99378869\\
14	1.86641082\\
15	1.79576675\\
16	1.86306585\\
17	1.96926502\\
18	1.99789144\\
19	2.00120473\\
20	1.99981101\\
21	2.00042861\\
22	2.00031386\\
};
\addlegendentry{$\CC$};

\end{axis}
\end{tikzpicture}%

%% file: ParaIdentFinger_endTime_noise_R16.tikz
%
%
\begin{tikzpicture}

\begin{axis}[%
width=0.4\textwidth,
height=3.5cm,
at={(0.0,0.0)},
scale only axis,
xmin=0,
xmax=39,
ymin=0,
ymax=11,
axis background/.style={fill=white},
legend style={legend cell align=left,align=left,draw=white!15!black, legend columns=-1},
legend pos = north east
]
\addplot [color=black,solid,mark=diamond,mark options={solid}]
  table[row sep=crcr]{%
0	0\\
1	0.256281\\
2	0.92300146\\
3	0.92300146\\
4	1.27115215\\
5	1.48684021\\
6	1.66792546\\
7	1.92517336\\
8	2.21200193\\
9	2.65159209\\
10	2.88245522\\
11	2.88245522\\
12	3.37264462\\
13	3.37264462\\
14	3.87264462\\
15	4.7365226\\
16	5.43383713\\
17	5.70929683\\
18	5.9067757\\
19	6.15808813\\
20	6.21453836\\
21	6.5315182\\
22	6.78391396\\
23	6.94552339\\
24	6.99419445\\
25	6.99498495\\
26	7.00095575\\
27	7.00049768\\
28	7.00139456\\
29	6.99937355\\
30	6.99977818\\
31	6.99810101\\
32	6.99939907\\
33	6.99741917\\
34	6.99967477\\
35	6.99962241\\
36	7.00057496\\
37	6.99848029\\
38	6.99942003\\
39	6.99773169\\
};
\addlegendentry{$\PP$};

\addplot [color=black,solid,mark=o,mark options={solid}]
  table[row sep=crcr]{%
0	0\\
1	4\\
2	4.95080969\\
3	4.95080969\\
4	5.95080969\\
5	7.95080969\\
6	10\\
7	10\\
8	10\\
9	7.03543064\\
10	7.37957484\\
11	7.37957484\\
12	6.37957484\\
13	6.37957484\\
14	5.87957484\\
15	4.87957484\\
16	4.86765557\\
17	5.27827258\\
18	5.66627558\\
19	5.94077249\\
20	6.3411478\\
21	6.44588385\\
22	6.18634579\\
23	6.0574217\\
24	6.00027588\\
25	6.00639204\\
26	5.99733431\\
27	6.00057875\\
28	5.99834254\\
29	6.00104667\\
30	5.9997334\\
31	6.002082\\
32	5.99948852\\
33	6.00255127\\
34	5.99944234\\
35	6.00114526\\
36	5.99910895\\
37	6.00200894\\
38	6.00006913\\
39	6.00234968\\
};
\addlegendentry{$\chi$};

\addplot [color=black,solid,mark=square,mark options={solid}]
  table[row sep=crcr]{%
0	0\\
1	2\\
2	0.54292874\\
3	0.54292874\\
4	0.39820655\\
5	0.28706411\\
6	0.26277524\\
7	0.30464641\\
8	0.34153123\\
9	0.50492488\\
10	0.58671584\\
11	0.58671584\\
12	0.77056514\\
13	0.77056514\\
14	0.98676547\\
15	1.45111591\\
16	1.84676253\\
17	1.81705299\\
18	1.72040316\\
19	1.74012974\\
20	1.66162945\\
21	1.78376188\\
22	1.88442136\\
23	1.96833839\\
24	1.99478216\\
25	1.99789833\\
26	1.99963983\\
27	2.00078658\\
28	2.00051714\\
29	1.99982197\\
30	1.99956468\\
31	1.99928102\\
32	1.99932779\\
33	1.99889228\\
34	1.99938947\\
35	2.00004888\\
36	2.00005071\\
37	1.9993982\\
38	1.99941087\\
39	1.99875972\\
};
\addlegendentry{$\CC$};

\end{axis}
\end{tikzpicture}%

%% file: KL_ParaIden.bbl
\begin{thebibliography}{10}

\bibitem{Allmaras}
M.~Allmaras, W.~Bangerth, J.M. Linhart, J.~Polanco, F.~Wang, K.~Wang,
  J.~Webster, and S.~Zedler.
\newblock {Estimating Parameters in Physical Models through Bayesian Inversion:
  A Complete Example}.
\newblock {\em SIAM Rev.}, 55(1):149--167, 2013.

\bibitem{BBV}
R.~Becker, M.~Braack, and B.~Vexler.
\newblock {Numerical parameter estimation for chemical models in
  multidimensional reactive flows}.
\newblock {\em Combust. Theory Model.}, 8(4):661--682, 2004.

\bibitem{Blank}
L.~Blank, M.~Butz, and H.~Garcke.
\newblock {Solving the Cahn--Hilliard variational inequality with a semi-smooth
  Newton method}.
\newblock {\em ESAIM: Control Optim. Calc. Var.}, 17(4):931--954, 2011.

\bibitem{BloweyElliott}
J.F. Blowey and C.M. Elliott.
\newblock {The Cahn--Hilliard gradient theory for phase separation with
  non-smooth free energy. Part I: Mathematical analysis}.
\newblock {\em European J. Appl. Math.}, 2(3):233--280, 1991.

\bibitem{BG}
H.~Br\'{e}zis and T.~Gallouet.
\newblock {Nonlinear Schr\"{o}dinger evolution equations}.
\newblock {\em Nonlinear Anal.}, 4(4):677--681, 1980.

\bibitem{Madzvamuse}
E.~Campillo-Funollet, C.~Venkataraman, and A.~Madzvamuse.
\newblock {A Bayesian approach to parameter identification with an application
  to Turing systems}.
\newblock Preprint arXiv:1605.04718, 2016.

\bibitem{Colli:tumour}
P.~Colli, G.~Gilardi, E.~Rocca, and J.~Sprekels.
\newblock {Optimal distributed control of a diffuse interface model of tumor
  growth}.
\newblock {\em Nonlinearity}, 30:2518--2546, 2017.

\bibitem{Collis}
J.~Collis, A.J. Connor, M.~Paczkowski, P.~Kannan, J.~Pitt-Francis, H.M. Byrne,
  and M.E. Hubbard.
\newblock {Bayesian Calibration, Validation and Uncertainty Quantification for
  Predictive Modelling of Tumour Growth: A Tutorial}.
\newblock {\em Bull. Math. Biol.}, 79(4):939--974, 2017.

\bibitem{CLLW}
V.~Cristini, X.~Li, J.~S. Lowengrub, and S.~M. Wise.
\newblock {Nonlinear simulations of solid tumor growth using a mixture model:
  invasion and branching}.
\newblock {\em J. Math. Biol.}, 58(4-5):723--763, 2009.

\bibitem{CLNie}
V.~Cristini, J.~Lowengrub, and Q.~Nie.
\newblock {Nonlinear simulations of tumor growth}.
\newblock {\em J. Math. Biol.}, 46(3):191--224, 2003.

\bibitem{Doerfler}
W.~D\"{o}rfler.
\newblock {A convergent adaptive algorithm for Poisson's equation}.
\newblock {\em SIAM J. Numer. Anal.}, 33(3):1106--1124, 1996.

\bibitem{Engler}
H.~Engler.
\newblock {An alternative proof of the Brezis-Wainger inequality}.
\newblock {\em Comm. Partial Differential Equations}, 14(4):541--544, 1989.

\bibitem{GarHK_CHNS_AGG_linearStableTimeDisc}
H.~Garcke, M.~Hinze, and C.~Kahle.
\newblock {A stable and linear time discretization for a thermodynamically
  consistent model for two-phase incompressible flow}.
\newblock {\em Appl. Numer. Math.}, 99:151--171, 2016.

\bibitem{GLDiri}
H.~Garcke and K.F. Lam.
\newblock {Analysis of a Cahn-Hilliard system with non-zero Dirichlet
  conditions modeling tumor growth with chemotaxis}.
\newblock {\em Discrete Contin. Dyn. Syst.}, 37(8):4277--4308, 2017.

\bibitem{GLNeu}
H.~Garcke and K.F. Lam.
\newblock {Well-posedness of a Cahn--Hilliard system modelling tumour growth
  with chemotaxis and active transport}.
\newblock {\em European. J. Appl. Math.}, 28(2):284--316, 2017.

\bibitem{GLNS}
H.~Garcke, K.F. Lam, R.~N\"{u}rnberg, and E.~Sitka.
\newblock {A multiphase Cahn--Hilliard--Darcy model for tumour growth with
  necrosis}.
\newblock Preprint arXiv:1701.06656, 2017.

\bibitem{GLR:Opt}
H.~Garcke, K.F. Lam, and E.~Rocca.
\newblock {Optimal control of treatment time in a diffuse interface model for
  tumour growth}.
\newblock Preprint arXiv:1608.00488, To appear in \emph{Appl. Math. Optim.}
  DOI:10.1007/s00245-017-9414-4, 2017.

\bibitem{GLSS}
H.~Garcke, K.F. Lam, E.~Sitka, and V.~Styles.
\newblock {A Cahn--Hilliard--Darcy model for tumour growth with chemotaxis and
  active transport}.
\newblock {\em Math. Models Methods Appl. Sci.}, 26(6):1095--1148, 2016.

\bibitem{Grisvard}
P.~Grisvard.
\newblock {\em {Elliptic problems in nonsmooth domains}}.
\newblock Vol. 69 of Classics in Applied Mathematics. SIAM, 2011.

\bibitem{Hawkins}
A.~Hawkins-Daarud, S.~Prudhomme, K.G. van~der Zee, and J.T. Oden.
\newblock {Bayesian calibration, validation, and uncertainty quantification of
  diffuse interface models of tumour growth}.
\newblock {\em J. Math. Biol.}, 67(6-7):1457--1485, 2013.

\bibitem{HHK}
M.~Hinterm\"{u}ller, M.~Hinze, and C.~Kahle.
\newblock {An adaptive finite element Moreau--Yosida-based solver for a coupled
  Cahn--Hilliard/Navier--Stokes system}.
\newblock {\em J. Comput. Phys.}, 235:810--827, 2013.

\bibitem{HHT}
M.~Hinterm\"{u}ller, M.~Hinze, and M.H. Tber.
\newblock {An adaptive finite element Moreau--Yosida-based solver for a
  non-smooth Cahn--Hilliard problem}.
\newblock {\em Optim. Methods Softw.}, 25(4-5):777--811, 2011.

\bibitem{HK}
M.~Hinterm\"{u}ller and I.~Kopacka.
\newblock {A smooth penalty approach and a nonlinear multigrid algorithm for
  elliptic MPECs}.
\newblock {\em Comput. Optim. Appl.}, 50(1):111--145, 2011.

\bibitem{HPUU}
M.~Hinze, R.~Pinnau, M.~Ulbrich, and S.~Ulbrich.
\newblock {\em {Optimization with PDE Constraints}}.
\newblock Mathematical Modelling: Theory and Applications. Springer
  Netherlands, 2009.

\bibitem{LamWu}
K.F. Lam and H.~Wu.
\newblock {Thermodynamically consistent Navier--Stokes--Cahn--Hilliard models
  with mass transfer and chemotaxis}.
\newblock Preprint arXiv:1702.06014, 2017.

\bibitem{Lima2}
E.A.B.F. Lima, J.T. Oden, D.A. Hormuth~II, T.E. Yankeelov, and R.C. Almeida.
\newblock {Selection, calibration, and validation of models of tumor growth}.
\newblock {\em Math. Models Methods Appl. Sci.}, 26:2341--2368, 2016.

\bibitem{Lima}
E.A.B.F. Lima, J.T. Oden, B.~Wohlmuth, A.~Shamoradi, D.A. Hormuth~II, T.E.
  Yankeelov, L.~Scarabosio, and T.~Horger.
\newblock {Selection and Validation of Predictive Models of Radiation Effects
  on Tumor Growth Based on Noninvasive Imaging Data}.
\newblock ICES report 17-14
  \url{https://www.ices.utexas.edu/media/reports/2017/1714.pdf}, 2017.

\bibitem{fenics_book}
A.~Logg, K.-A. Mardal, and G.~Wells.
\newblock {\em {Automated Solution of Differential Equations by the Finite
  Element Method - The FEniCS Book}}, volume~84 of {\em Lecture Notes in
  Computational Science and Engineering}.
\newblock Springer--Verlag Berlin Heidelberg, 2012.

\bibitem{NocedalWright}
J.~Nocedal and S.J. Wright.
\newblock {\em Numerical Optimization}.
\newblock Springer Series in Operation Research and Financial Engineering.
  Springer--Verlag New York, 2006.

\bibitem{Pennacchietti}
S.~Pennacchietti, P.~Michieli, M.~Galluzzo, M.~Mazzone, S.~Giordano, and P.~M.
  Comoglio.
\newblock {Hypoxia promotes invasive growth by transcriptional activation of
  the {\it met} protooncogene}.
\newblock {\em Cancer Cell}, 3(4):347--361, 2003.

\bibitem{alberta_book}
A.~Schmidt and K.G. Siebert.
\newblock {\em {Design of adaptive finite element software: The finite element
  toolbox ALBERTA}}, volume~42 of {\em Lecture Notes in Computational Science
  and Engineering}.
\newblock Springer--Verlag Berlin Heidelberg, 2005.

\bibitem{Stuart_Bayes}
A.M. Stuart.
\newblock {Inverse problems: A Bayesian perspective}.
\newblock {\em Acta Numerica}, 19:451--559, 2010.

\bibitem{Troltzsch}
F.~Tr{\"o}ltzsch.
\newblock {\em {Optimal Control of Partial Differential Equations: Theory,
  Methods, and Applications}}.
\newblock Graduate studies in mathematics. AMS, Providence, RI, 2010.

\bibitem{Verfuerth}
R.~Verf\"{u}rth.
\newblock {\em {A review of a posteriori error estimation and adaptive
  mesh-refinement techniques}}.
\newblock Wiley-Teubner series: Advances in Numerical Mathematics.
  Wiley-Teubner, New York, 1996.

\end{thebibliography}
